\documentclass{amsart}
\usepackage[utf8]{inputenc}

\usepackage[]{mdframed}

\usepackage[hidelinks]{hyperref}
\usepackage[marginratio={1:1, 1:1}]{geometry}

\usepackage{style}

\title[Non-symplectic automorphisms of O'Grady's tenfolds and cubic fourfolds]{Non-symplectic automorphisms of prime order of O'Grady's tenfolds and cubic fourfolds}

\author{Simone Billi}

\address[Simone Billi]{Università di Genova, Dipartimento di Matematica, Via Dodecaneso, 35, 16146 Genova, Italy}
\email{simone.billi@edu.unige.it}

\author{Annalisa Grossi}

\address[Annalisa Grossi]{Université Paris-Saclay, CNRS, Laboratoire de Mathématiques d'Orsay, Rue Michel Magat, Bât. 307, 91405 Orsay, France}
\curraddr{Alma Mater studiorum Università di Bologna Dipartimento di Matematica,
Piazza di Porta San Donato 5, Bologna, 40126 Italia}
\email{annalisa.grossi3@unibo.it}

\date{\today}

\makeatletter
\@namedef{subjclassname@2020}{%
 \textup{2020} Mathematics Subject Classification}
\makeatother



\thanks{Simone Billi was partially supported by the Curiosity Driven 2021 Project \textit{Varieties with trivial or
negative canonical bundle and the birational geometry of moduli spaces of curves: a
constructive approach} - Programma nazionale per la Ricerca (PNR) DM 737/2021. 
Annalisa Grossi was partially supported by the European Research Council (ERC) under the European Union’s Horizon 2020 research and innovation programme (ERC-2020-SyG-854361-HyperK) and by Funded by the European Union - NextGenerationEU under the National Recovery and Resilience Plan (PNRR) - Mission 4 Education and research - Component 2 From research to business - Investment 1.1 Notice Prin 2022 - DD N. 104 del 2/2/2022, from title "Symplectic varieties: their interplay with Fano manifolds and derived categories", proposal code 2022PEKYBJ – CUP J53D23003840006. Simone Billi and Annalisa Grossi are members of the INdAM group GNSAGA}

\usepackage{framed}

\begin{document}

\begin{abstract}
    We give a lattice-theoretic classification of non-symplectic automorphisms of prime order of irreducible holomorphic symplectic manifolds of \(\ogten\) type. We determine which automorphisms are induced by a non-symplectic automorphism of prime order of a cubic fourfold on the associated Laza--Saccà--Voisin manifolds, giving a geometric and lattice-theoretic description of the algebraic and transcendental lattices of the cubic fourfold. As an application we discuss the rationality conjecture for a general cubic fourfold with a non-symplectic automorphism of prime order.
\end{abstract}
\maketitle
\section{Introduction}
An irreducible holomorphic symplectic (ihs) manifold is a simply connected, compact, complex, Kähler manifold \(X\) carrying a non-degenerate holomorphic symplectic form \(\sigma_X\) which spans  \(\Homology^{0}(X,\Omega_{X}^2)\).
An automorphism of an ihs manifold is called \textit{symplectic} if it acts trivially on the symplectic form, \textit{non-symplectic} otherwise. A cyclic group \(G \subset \Aut(X)\) is called non-symplectic if it is generated by a non-symplectic automorphism. 

Cubic hypersurfaces in \(\mathbb{P}^5\) admit a Hodge decomposition of K3 type, i.e. \(\Homology^4(Y,\mathbb{C}) = \Homology^{3,1}(Y) \oplus \Homology^{2,2}(Y) \oplus \Homology^{1,3}(Y)\) and \(h^{3,1}(Y)=1\), hence there is a notion of symplectic and non-symplectic automorphism. Namely, an automorphism of a cubic fourfold \(Y\) is symplectic if the induced action on \(\Homology^{4}(Y,\mathbb{Z})\) acts trivially on \(\Homology^{3,1}(Y)\), non-symplectic otherwise.

 Giovenzana--Grossi--Onorati--Veniani \cite{2022symplecticrigidity} prove that any symplectic automorphism of finite order of an ihs manifold of O’Grady’s 10-dimensional deformation
type is trivial. 

In this paper we classify non-symplectic automorphisms of prime order of an ihs manifold \(X\) of \(\ogten\) type. If \(X\) can be realized as a compactification of the twisted intermediate Jacobian fibration of the hyperplane sections of a cubic fourfold, referring to the model due to Laza--Saccà--Voisin \cite{laza2017hyper,sacca2020birational},
we give a lattice-theoretic criterion to determine when such an automorphism is induced by an automorphism of the cubic fourfold. 
To this purpose, we study the induced action on \(\Homology^4(Y,\mathbb{Z})\) by a non-symplectic automorphism of prime order of a smooth cubic fourfold \(Y\).
Starting from these results we exhibit the algebraic and the transcendental lattice of a general cubic fourfold that admits a 
non-symplectic automorphism of prime order. Moreover, we give a geometric set of generators of the algebraic lattice in terms of planes or cubic scrolls. As a consequence, we discuss the rationality conjecture for cubic fourfolds with a non-symplectic automorphism of prime order. 

The classification of the algebraic and the transcendental lattices of a cubic fourfold with a non-symplectic involution, and the discussion of the rationality conjecture is the content of a recent paper by Marquand \cite{marquand2023cubic}.

The classification of automorphisms of ihs manifolds of \(\ogten\) type is an extension of a result by
Brandhorst–Cattaneo \cite{prime_order_brandhorst} that provides lattice-theoretic constraints of non-symplectic automorphisms of odd prime order of an ihs manifold in terms of isometries of unimodular lattices. 

One of the advantages in studying induced automorphisms is to control the fixed locus. Automorphisms of ihs manifolds with empty fixed locus are needed to construct Enriques manifolds, the higher dimensional analogue of Enriques surfaces, see \cite{oguiso2011enriques, Boissier_NW_Sarti_Enriques} for more details.
The authors together with Luca Giovenzana and Franco Giovenzana in an upcoming paper \cite{noi_e_giovenzani} use the results about induced automorphisms to investigate the existence of Enriques manifolds as free quotient of ihs manifolds of \(\ogten\) type.
\subsection{Automorphisms of ihs manifolds and cubic fourfolds}
Automorphisms of ihs manifolds can be classified studying the induced action on the second integral cohomology \(\Homology^2
(X,\mathbb{Z})\), which carries a lattice structure provided by the
Beauville–Bogomolov–Fujiki quadratic form. In particular, if \(X\) is an ihs manifold of \(\ogten\) type we know by \cite{rapagnetta2006beauville} that \(\Homology^{2}(X,\mathbb{Z})\) is isometric to the abstract lattice \(\bL\colonequals \U^{\oplus3} \oplus \E_{8}(-1)^{\oplus2} \oplus \A_2(-1)\). A marking is an isometry \(\eta \colon \Homology^2(X,\mathbb{Z}) \rightarrow \bL\) and, whenever we fix a marked pair \((X,\eta)\) of \(\ogten\) type, the representation map
\begin{equation}\label{main rep map}
\eta_*\colon \Aut(X) \rightarrow \bO(\bL), \quad f \mapsto \eta \circ (f^{-1})^*\circ \eta^{-1}
\end{equation}
is injective by a result of Mongardi--Wandel \cite{mongardi2015induced}. For this reason, we study automorphisms in terms of their induced action in cohomology. More precisely, if an isometry \(\varphi \in \bO(\bL)\) verifies the assumptions of Hodge--theoretic Torelli theorem \cite[Theorem 1.3]{markman2011survey}, then there exists an automorphism \(g \in \Aut(X)\) such that \(\eta_{*}(g)=\varphi\). 

If \((X,\eta)\) is a marked pair of \(\ogten\) type, an isometry \(\varphi \in \bO(\bL)\) is \textit{non-symplectic} if the action on  the Hodge structure of \(\bL\otimes \mathbb{C}\) (induced via the marking) is non-symplectic.
A cyclic group \(G \subset \bO(\bL)\) is called \textit{non-symplectic} if \(G\) is generated by a non-symplectic isometry.

Similarly if \(Y\) is a cubic fourfold, the integral cohomology \(\Homology^4(Y,\mathbb{Z})\) with the intersection form, is an odd unimodular lattice isometric to \([1]^{\oplus 21} \oplus [-1]^{\oplus 3}\). Consider the square of an hyperplane section \(\eta_Y\in\Homology^4(Y,\mathbb{Z})\), then the primitive cohomology \(\Homology^{4}_p(Y,\mathbb{Z})=\langle\eta_Y\rangle^\perp\) with the intersection form is an even lattice isometric to \(\F:=\U^{\oplus 2}\oplus \E_8^{\oplus 2}\oplus \A_2\). We denote by \(A(Y)\colonequals\Homology^{2,2}(Y,\mathbb{C}) \cap \Homology^{4}(Y, \mathbb{Z})\) the lattice of algebraic cycles of \(Y\), and by \(T(Y)=A(Y)^{\perp} \subset H^{4}(Y,\mathbb{Z})\) the transcendental lattice. We denote by \(A_p(Y) \colonequals \Homology^{2,2}(Y,\mathbb{C}) \cap \Homology^{4}_p(Y, \mathbb{Z})\) the even lattice of algebraic primitive cycles. A marking is an isometry  \(\gamma \colon \Homology^{4}_p(Y,\mathbb{Z})\to \F\). Let \((Y,\gamma)\) be a marked pair, an isometry \(\varphi \in \bO(\F)\) is \textit{non-symplectic} if the action on  the Hodge structure of \(\F\otimes \mathbb{C}\) (induced via the marking) is non-symplectic.
Automorphisms of \(Y\) are identified with Hodge isometries of \(\Homology^4_p(Y,\mathbb{Z})\) by the Hodge--theoretic Torelli theorem \cite{voisin1986theoreme}, and then similar techniques to the ones used for ihs manifolds can be applied.

We recall the known results about automorphisms of other deformations types of ihs manifolds.
Regarding
prime order  automorphisms of ihs manifolds of K3\(^{[2]}\) type we refer to \cite{camere2012symplectic, mongardi2012symplectic,beauville2011antisymplectic,ohashi2013non,boissiere2016classification,boissiere2016isometries,camere2019verra}.
Non-symplectic automorphisms of ihs manifolds of K3\(^{[n]}\) type are treated in \cite{camere2020non,camere2021non}, while automorphisms of ihs manifolds of K\(_n(A)\) type are classified in \cite{mongardi2018prime} for \(n=2\), and in \cite{prime_order_brandhorst} for every \(n\).
The classification of symplectic automorphisms of ihs manifolds of \(\ogsix\) type is given in \cite{grossi2020symplectic}, while non-symplectic automorphisms of prime order are classified in \cite{grossi2022nonsymplectic} by the second named author. Symplectic birational infolutions of ihs manifolds of \(\ogten\) type are studied in \cite{marquand2023classification} and groups of symplectic birational transformations are studied in \cite{marquand2024finite}.

Automorphisms of prime order of a smooth cubic fourfold are calssified in \cite{Gonzales_Liendo,fu2016classification,yu2020moduli} in terms of their actions on \(\mathbb{P}^{5}\) and giving a general equation of \(Y\). Groups of symplectic automorphisms of cubic fourfolds are studied in \cite{fu2016classification,laza2022automorphisms}. The classification of algebraic and transcendental lattices of a cubic fourfold with an involution is the content of \cite{marquand2023cubic}.

\subsection{Contents of the paper}
In \autoref{Preliminaries} we introduce basic notions and results on lattice theory to study automorphisms of ihs manifolds and cubic fourfolds. Moreover, we recall Hodge--theoretic Torelli theorem for these manifolds and the construction due to Laza--Saccà--Voisin. 

In \autoref{Classification of non-symplectic automorphisms of OG10} we classify cyclic groups \(G\subset \Aut(X)\) of non-symplectic automorphisms of prime order of an ihs manifold \(X\) of \(\ogten\) type. We denote by \(\bL^G\) the \textit{invariant lattice} and by \(\bL_G\) the \textit{coinvariant lattice}. We provide a list of invariant and coinvariant sublattices of the induced action of \(G\) on \(\bL\) up to isometry. To achieve this we consider the unique primitive embedding \(\bL \hookrightarrow\bLambda:=\U^{\oplus 5}\oplus \E_8(-1)^{\oplus 2}\) in the unimodular lattice \(\bLambda\), and we refer to \cite{prime_order_brandhorst} for a classification of odd prime order isometries of unimodular lattices.

\begin{theorem}[see \autoref{Main theorem section 3}]
Let \(G \subset \bO(\bL)\) be a group of non-symplectic isometries of prime order \(p\). Then there exists an ihs manifold \(X\) of \(\ogten\) type such that \(G \subset \Aut(X)\) if and only if \(2 \leq p \leq 23\) and the pairs \((\bL
^G,\bL_G)\) appear in \autoref{tab:Lid} and \autoref{tab:Lnotid} for \(p=2\), and in \autoref{tab:L_p} for \(p\geq 3\).
\end{theorem}

Throughout the paper we refer to a \textit{smooth} cubic fourfold, even if it is not always specified. 
By \cite{Gonzales_Liendo, yu2020moduli} if a non-symplectic automorphism of a cubic fourfold has prime order, it can be an involution or an automorphism of order three. In \autoref{Non-symplectic automorphisms of order three of a cubic fourfold} we give a lattice-theoretic and a geometric classification of non-symplectic automorphisms of prime order of a cubic fourfold.

Let \(G \subset \Aut(Y)\) be a group of non-symplectic automorphisms of prime order of a general cubic fourfold \(Y\), then the invariant lattice \(\F^G\) coincide with the primitive algebraic lattice \(A_p(Y)\), and the coinvariant lattice \(\F_G\) coincides with the transcendental lattice \(\T(Y)\).
In \autoref{Subsection alg and tras lattices} we classify the algebraic and the transcendental lattices of a general cubic fourfold with a non-symplectic automorphism of order three.
\begin{theorem}[See {\autoref{order_three_cubiche_prim}}]
Let \(G \subset\bO(\F)\) be a group of non-symplectic isometries of order three. Then there exists a cubic fourfold \(Y\) such that \(G\subset\Aut(Y)\) if and only if the pairs \((\F^G,\F_G)\) appears in \autoref{tab:order_three_non_sympl_cubic_prim}.  For such a general \(Y\) the algebraic lattice \(A(Y)\) appear in \autoref{tab:algebraic_transcendental_lattice_cubic_with_order3_action}, and the class of \(\eta_Y\) is expressed in a basis of \(A(Y)\). 
\end{theorem}

There are many examples of rational cubic
fourfolds \cite{hassett2000special,kuznetsov2016derived, BRS19, RS19}, and it is conjectured that a cubic fourfold is rational if and only
if it admits an associated K3 surface (i.e the transcendental cohomology is
induced from a K3 surface) \cite{hassett2000special, kuznetsov2010derived, kuznetsov2016derived}. 
In \autoref{Geometry of cubic fourfolds with a non-symplectic automorphism of order three} we provide a geometric set of generators for the algebraic lattice of a cubic fourfold that admits a non-symplectic automorphism of order three, see \autoref{theorem set of generators for phi_3^2}, \autoref{theorem set of generators for phi_3^5} and \autoref{theorem set of generators for phi_3^7}. Moreover, knowing the algebraic and the transcendental lattices, we determine if there exists an associated K3 surface via Torelli theorem. When there is an associated K3 surface we verify the conjecture proving the rationality of the cubic. We collect these results in \autoref{tab:rat_properties_cubics}. 

\begin{center}
{\begin{longtable}{lllllll}
\caption{Cubic fourfolds with a non-symplectic automorphism of order three, same notation as in {\autoref{class_cubiche}}.}
\label{tab:rat_properties_cubics} \\
\toprule
No.  &  \(\rk(A(Y))\)  & Associated K3 & Rational \\
\midrule
\endfirsthead
\multicolumn{4}{c}%
{\tablename\ \thetable{}, follows from previous page} \\
No.  &  \(\rk(A(Y))\) & Associated K3 & Rationality\\
\midrule
\endhead
\multicolumn{4}{c}{Continues on next page} \\
\endfoot
\bottomrule
\endlastfoot

\(\phi_3^1\)  & \(1\) & No & ?  \\
\(\phi_3^5\)  & \(7\) & No & ?\\
\(\phi_3^7\)  & \(9\) & Yes & Yes  \\
\(\phi_3^2\) & \(13\) & Yes & Yes \\

\end{longtable}}
\end{center}

In \autoref{Automorphisms on Laza-Sacca-Voisin manifolds of OG10-type}
we relate an automorphism of a cubic fourfold \(Y\) to the induced automorphism of \(J(Y)\) and \(J^t(Y)\), where \(J(Y)\) denotes the associated Laza--Saccà--Voisin manifold, and \(J^t(Y)\) denotes the associated twisted Laza--Saccà--Voisin manifold. See \autoref{Laza--Saccà--Voisin manifolds} for precise definitions. A crucial result that follows from \cite{li2022elliptic,giovenzana2022period} relates the primitive integral cohomology \(\Homology^{4}_p(Y,\mathbb{Z})\) to the integral cohomology \(\Homology^2(J^t(Y),\mathbb{Z})\), see \autoref{hodge_isometry_twist}.

We obtain that any non-symplectic automorphism of a general cubic fourfold \(Y\) induces a non-symplectic automorphism of the associated \(J(Y)\) and \(J^t(Y)\). Moreover, we characterize non-symplectic automorphisms of order two and three of Laza--Saccà--Voisin manifolds that are induced by non-symplectic automorphisms of a cubic fourfold.

\begin{theorem}[See \autoref{induced_by_cubic} and \autoref{regular_autom}]
Let \(Y\) be a general cubic fourfold and let \(G\subset\Aut(Y)\) be group of non-symplectic automorphisms of prime order \(p\). Then the automorphisms of \(G\) induce non-symplectic automorphisms of the ihs manifold of \(\ogten\) type \(J^t(Y)\), and the pairs \((\bL^G,\bL_G)\) for such induced actions are classified in \autoref{tab:induced_actions_LSV}. Viceversa, if \(X\) is a general ihs manifold of \(\ogten\) type and \(G\subset\Aut(X)\) is a group of non-symplectic automorphisms of prime order such that \((\bL^G,\bL_G)\) appears in \autoref{tab:induced_actions_LSV}, then there exists a smooth cubic fourfold \(Y\) and a group of non-symplectic automorphisms \(G \subset \Aut(Y)\) such that \(X\) is birational to \(J^t(Y)\) and the group of automorphisms is induced by automorphisms of the cubic fourfold.
\end{theorem}

\begin{center}
{\begin{longtable}{lllllll}
\caption{The column \(\Aut(Y)\) refers to non-symplectic automorphisms of a cubic fourfold \(Y\) with the same notation as in {\autoref{class_cubiche}}. The pairs \((\bL^G,\bL_G)\) are the invariant and the coinvariant lattices of the induced automorphism on the \(\ogten\) type manifold \(J^t(Y)\). In the column \(J(Y)\) we point out if \(J^t(Y)\) is birational to \(J(Y)\), while \(p\) is the order of \(G\).}
\label{tab:induced_actions_LSV} \\
\toprule
\(\Aut(Y)\)  &  \(\bL_G\) & \(\bL^G\) & \(\sgn(\bL^G)\) & \(J(Y)\) & \(p\)  \\
\midrule
\endfirsthead
\multicolumn{6}{c}%
{\tablename\ \thetable{}, follows from previous page} \\
\(\Aut(Y)\)  &  \(\bL_G\) & \(\bL^G\) & \(\sgn(\bL^G)\) & \(J(Y)\)& \(p\)\\
\midrule
\endhead
\multicolumn{6}{c}{Continues on next page} \\
\endfoot
\bottomrule
\endlastfoot

\(\phi_2^1\)  & \(\U^{\oplus 2}  \oplus \D_4(-1)^{\oplus 3}  \)& \(
\U  \oplus \E_6(-2)
\)& \((1,7)\) & yes & \(2\) \\

\(\phi_2^3\)  & \(\U  \oplus [2]\oplus[-2]^{\oplus 9} \)& \( [2]\oplus[-2]\oplus\E_6(-1)\oplus\D_4(-1)
\)& \((1,11)\) & yes & \(2\) \\

\hline

\(\phi_3^1\)  & \(\U^{\oplus} \oplus \E_8(-1)^{\oplus 2}\oplus\A_2(-1) \)& \(\U(3)
\)& \((1,1)\) & no& \(3\) \\


\(\phi_3^5\)  & \(\U\oplus \U(3) \oplus \E_6(-1)\oplus \A_2(-1)^{\oplus3} \)& \(\U(3)\oplus  \E_6^*(-3)
\)& \((1,7)\) & no&  \(3\) \\

\(\phi_3^7\)  & \(\U\oplus \U(3) \oplus \A_2(-1)^{\oplus 5} \)& \(\U(3)\oplus \E^*_6(-3)\oplus \A_2(-1)
\)& \((1,9)\) &yes&  \(3\) \\

\(\phi_3^2\) & \(\U\oplus \U(3) \oplus \A_2(-1)^{\oplus 3} \)& \(\U(3)\oplus \E_6(-1)\oplus \A_2(-1)^{\oplus 3}
\)& \((1,13)\) & yes&  \(3\) \\
\end{longtable}}
\end{center}

\subsection*{Acknowledgments} 
We wish to warmly thank Simon Brandhorst and Stevell Muller for fruitful discussions and for introducing us to the use of the open source computational software OSCAR, Emanuele Macrì and Lisa Marquand for their hints and suggestions, Giovanni Mongardi, Enrico Fatighenti and Franco Giovenzana for useful discussions.

\section{Preliminaries}\label{Preliminaries}

\subsection{Lattices}
A \textit{lattice} \(L\) is a free \(\mathbb{Z}\)-module of finite rank with an integral symmetric bilinear form \[(-,-):L\times L\rightarrow \mathbb{Z}\] which is non-degenerate.
The \textit{signature} \((l_+,l_-)\) of \(L\) is the signature of the real extension of \((-,-)\). The lattice is \textit{positive-definite} if \(l_-=0\) and \textit{negative-definite} if \(l_+=0\), otherwise it is called \textit{indefinite}. A lattice \(L\) is \textit{hyperbolic} if it is indefinite and \(l_+=1\). A lattice \(L\) is \textit{even} if \(x^2:=(x,x)\in 2\mathbb{Z}\) for any \(x\in L\), \textit{odd} otherwise.
The divisibility \(\divi(x,L)\) of an element \(x\in L\) is the positive generator of the ideal \(\{(x,y)|y\in L\}\subseteq \mathbb{Z}\).

Consider the \textit{dual lattice} \[L^\vee=\Hom_\mathbb{Z}(L,\mathbb{Z})\cong\{x\in L\otimes_\mathbb{Z}\mathbb{Q}|(x,l)\in\mathbb{Z} \  \forall \ l\in L\}\] and observe that \(L\subset L^\vee\) is a finite index subgroup, the quotient \(A_{L}:=L^\vee/L\) is called the \textit{discriminant group} of \(L\).
The \textit{length} \(l(A_{L})\) is the minimum number of generators of \(A_{L}\). There is a well-defined bilinear form \(b_{A_{L}}:A_{L}\times A_{L}\rightarrow \mathbb{Q}/\mathbb{Z}\), and if the lattice is even there is an associated quadratic form \(q_{A_{L}} : A_{L} \rightarrow \mathbb{Q}/2\mathbb{Z}\).

A lattice \(L\) is called \textit{unimodular} if the group \(A_{L}\) is the identity, or equivalently if \(|\det(L)|= 1\). The lattice is called \textit{p-elementary} for a prime number \(p\) if \(A_{L}\cong (\mathbb{Z}/p\mathbb{Z})^k\) for some positive integer \(k\).
\begin{definition}
     Let \(L\) be a \(2\)-elementary even lattice, we define 
     \[ \delta(L):=\begin{cases}
     0 & \text{ if } q_{A_{L}}(x)\in\mathbb{Z}/2\mathbb{Z} \text{ for any } x\in A_{L}\\
     1 & \text{otherwise}
     \end{cases}\]
 \end{definition}
\begin{definition}
Given a lattice \(L\), we call \textit{short root} a vector \(v\in L\) such that \(v^2=2\) and \(\divi(v,L)=1\). We call \textit{long root} a vector \(v\in L\) such that \(v^2=6\) and \(\divi(v,L)=3\). 
\end{definition}

If \(L\) is a lattice and \(k\) is an integer, we denote by \(L(k)\) the lattice whose bilinear form is obtained by the one of \(L\) multiplied by \(k\).
We denote by \(\A_n,\D_n,\E_n\) the positive-definite lattices associated to the ADE Dynkin diagrams. We recall that \(\U\) denotes the even unimodular lattice of rank two, and we denote by \([k]\) the rank one lattice generated by an element of square \(k\in\mathbb{Z}\).

Two even lattices \(L\)
and \(L'\) are in the same \textit{genus} if \( L \oplus  \U \cong  L' \oplus \U\) or equivalently if and only if they have
the same signature and discriminant quadratic form, see \cite[Corollary 1.9.4]{nikulin1980integral}. Isometric lattices have the same genus, but lattices with the same genus might not be isometric. 

A morphism of lattices \(L\rightarrow M\) is a linear map that preserves the bilinear forms, and if the morphism is injective it is called an \textit{embedding} of lattices. An embedding \(L\hookrightarrow M\) is called \textit{primitive} if its cokernel is torsion free, and in this case we denote by \(L^{\perp} \subset M\) the orthogonal complement. If the embedding \(L\hookrightarrow M\) has finite index, then we say that \(M\) is an \textit{overlattice} of \(L\).
Recall that by \cite[Proposition 1.4.1]{nikulin1980integral} there is a bijective correspondence between overlattices of \(L\) and isotropic subgroups of \(A_L\).
A primitive embedding \(L\hookrightarrow M\), where we denote by \(T=L^\perp\subset M\) the orthogonal complement, is given by a subgroup \(H\subset A_M\) called the \textit{embedding subgroup} and an isometry \(\gamma\colon H\to H'\subset A_L\) called the \textit{embedding isometry}, see \cite[Proposition 1.15.1]{nikulin1980integral}. If \(\Gamma\subset A_M\oplus A_{L(-1)}\) denotes the graph of \(\gamma\), then \[A_T\cong \Gamma^\perp/\Gamma.\] Similarly, by \cite[Proposition 1.5.1]{nikulin1980integral} we have that, if \(M\) is unique in its genus, a primitive embedding \(L\hookrightarrow M\) where \(T=L^\perp \subset M\) is given by a subgroup \(K\subset A_L\) called the \textit{gluing subgroup}, and an isometry \(\gamma\colon K\to K'\subset A_T\) called the \textit{gluing isometry}. If \(\Gamma\subset A_L\oplus A_{T(-1)}\) is the graph of \(\gamma\), then we have \[A_M\cong\Gamma^\perp/\Gamma.\] 
If we consider a primitive embedding \(L\hookrightarrow M\) such that the orthogonal complement \(T=L^\perp \subset M\) is not unique in its genus, then the embedding subgroup does not determine the primitive embedding, nor the isometry class of \(T\), see for instance \autoref{example}. 

\begin{example}\label{example}
    Consider the following positive definite lattices \[\A\colonequals  \begin{pmatrix}
        12 &1 \\
        1 & 2
    \end{pmatrix},  \B \colonequals \begin{pmatrix}
        6 &1 \\
        1 & 4
    \end{pmatrix}.\] 
    The lattices \(\U \oplus \A\) and \(\U \oplus \B\) are isometric, the lattices \(\A\) and \(\B\) have the same genus, but \(\A\) and \(\B\) are not isometric.
   This shows that \(\U\) admits two different primitive embeddings in \(M \colonequals \U \oplus \A \cong \U \oplus \B\) such that the respective orthogonal complements are not isometric, even if the gluing embedding is uniquely determined since \(\U\) is unimodular.
\end{example}

There is a natural morphism \(\bO(L)\rightarrow\bO(A_L)\) between the group of isometries of the lattice \(L\) and the group of isometries of the discriminant group \(A_L\).
The Cartan-Dieudonné theorem \cite[Theorem 9.10]{morrison2009embeddings} guarantees that \(\bO(L\otimes_{\mathbb{Z}}\mathbb{R})\) is generated by reflections with respect to non-isotropic vectors, hence it is possible to give the following definition.
\begin{definition}
\label{spinor_norm}
The \textit{spinor norm} \(\spin\colon\bO(L)\rightarrow \{\pm 1\}\) is the group homomorphism that takes value \(+1\) on reflections for a vector \(v\) with \(v^2<0\).
\end{definition}
The kernel of the spinor norm is denoted by \(\bO^+(L)\) and consists of elements that preserve the orientation of a positive-definite subspace of \(L\otimes_{\mathbb{Z}}\mathbb{R}\) of maximal rank.

\begin{definition}
If \(G\subseteq \bO(L)\) is a group of isometries, we call \(L^G:=\{x\in L\mid g(x)=x,\forall g\in G\}\) the \textit{invariant lattice} and we call \(L_G:=(L^G)^{\perp_{L}}\) the \textit{coinvariant lattice}.
\end{definition}

\begin{lemma}[see {\cite[Lemma 5.3]{boissiere2012automorphismes},  \cite[Lemma 1.8]{mongardi2018prime}}]\label{prime_action}
    Let \(L\) be a lattice and \(G\subset\bO(L)\) the group generated by an isometry of prime order \(p\). Then \(m:=\rk(L_G)/(p-1)\) is an integer and \[\frac{L}{L^G\oplus L_G}\cong (\mathbb{Z}/p\mathbb{Z})^a\] as groups, where \(a\leq m\). Moreover the finite group \(\frac{L}{L^G\oplus L_G}\) is a subgroup of the discriminant groups \(A_{L^G}\) and \(A_{L_G}\).
\end{lemma}

If \(p\) is a prime number the following \(p\)-elementary lattices are well defined, and they will be useful in  \autoref{Classification of non-symplectic automorphisms of OG10}.
 \[ \K_p\colonequals\begin{pmatrix}
     -(p+1)/2& 1\\
     1 &-2
 \end{pmatrix} \h_p\colonequals\begin{pmatrix}
     (p-1)/2& 1\\
     1 &-2
 \end{pmatrix} .\] 
 We also consider the \(3\)-elementary lattice
\[ \E_6^*(3)\colonequals\begin{pmatrix}
     4& 2 & -1& 2& -1 & 1\\
     2 & 4 & 1 & 1 & -2 & 2\\
     -1& 1 & 4 & 1 & -2 & -1\\
     2&1&1&4&-2&-1\\
     -1 & -2 & -2&-2&4&-1\\
     1&2&-1&-1&-1&4
 \end{pmatrix} ,\]
and the \(17\)-elementary lattice 
 \[ \bL_{17}\colonequals\begin{pmatrix}
     2& 1 & 0& 1&\\
     1 & 2 & 0 & 0\\
     0& 0 & 2 & -1\\
     1&0&-1&4
 \end{pmatrix}.\]
 Finally, we introduce the lattices 
 \[ \N_{69}\colonequals\begin{pmatrix}
     6& 3 \\
     3 & -10
 \end{pmatrix} , \N_{15}\colonequals\begin{pmatrix}
     4& -1 \\
     -1 & 4
 \end{pmatrix} .\]

\subsection{Automorphisms of ihs manifolds of OG10 type}\label{Automorphisms of ihs manifolds of OG10 type}
Let \(X\) be an ihs manifold with a marking \(\eta \colon \Homology^2 (X,\mathbb {Z}) \to L\), then there is a representation map 
\[\Aut(X)\rightarrow \bO(L),\]
which is injective if \(X\) is an ihs manifold of \(\ogten\) type, see \cite{Mongardi_Wandel}.
The following theorem allows to study automorphisms in terms of Hodge isometries using the previous representation map. We recall that the \textit{Néron--Severi} lattice is denoted by \(\NS(X) \colonequals \Homology^{1,1}(X,\mathbb{C}) \cap \Homology^{2}(X,\mathbb{Z})\) and the \textit{transcendental} lattice is the orthogonal complement \(\T(X)\colonequals\NS(X)^{{\perp}_{L}}\).

\begin{proposition}[see, for instance, {\cite[§3]{Nik_finitegroups}}] \label{inv_contenuto_in_NS} If \((X, \eta)\) is a marked ihs manifold with marking \(\eta \colon \Homology^2 (X,\mathbb {Z}) \to L\) and \(G \subseteq O(L)\) is a group of non-symplectic isometries, then \(L^G \subseteq \NS(X)\) and \(\T(X) \subseteq L_G\).
\end{proposition}

We say that an ihs manifold \(X\) endowed with the action of \(G\subseteq\Aut(X)\) a group of non-symplectic  automorphisms is \textit{general} if one of the above inclusions is an equality, according to \cite[\S3]{Nik_finitegroups}.

By \cite{onorati2020monodromy}  if \(X\) is an ihs manifold of \(\ogten\) type with a marking \(\eta \colon \Homology^2 (X,\mathbb {Z}) \to \bL\), the monodromy group \(\Mon^2(X)\) coincides with the subgroup of index two of orientation preserving isometries \(\bO^+(\bL) \subset \bO(\bL)\). 
\begin{lemma}\label{p_leq_23}
If \(X\) is an ihs manifold of \(\ogten\) type
and \(G \subseteq \Aut(X)\) is a group of non-symplectic automorphisms of prime order \(p\) then \(p \leq 23\) .
\end{lemma}
\begin{proof}
By \autoref{prime_action} we know that \(p-1\) divides \(\rk(\bL_G)\), since \(\rk(\bL)=24\) we conclude. 
\end{proof}
Moreover, if \(X\) is an ihs manifold of \(\ogten\) type, \textit{numerical wall divisors} and \textit{numerical prime exceptional divisors} are computed in \cite{mongardi2022birational}:
\[ \mathcal{W}^{pex}_{\ogten}=\{v\in\bL|v^2=-2\}\cup\{v\in\bL|v^2=-6, \divi(v,\bL)=3\},\]
\[ \mathcal{W}_{\ogten}=\mathcal{W}^{pex}_{\ogten}\cup\{v\in\bL|v^2=-4\}\cup\{v\in\bL|v^2=-24, \divi(v,\bL)=3\}.\]
Hyperplanes orthogonal to classes in \(\mathcal{W}_{\ogten}\cap \NS(X)\) give a wall--and--chamber decomposition of the positive cone. The chamber that contains a Kähler class is the Kähler cone of \(X\).
\subsection{Cubic fourfolds and their automorphisms}

Let \(Y\subset\mathbb{P}^5\) be a smooth cubic fourfold.
The moduli space of smooth cubic fourfold \(\mathcal{M}\) is constructed in \cite{Laza_modulispaces}. We denote by \[\mathcal{D}/\Gamma=\{x\in \mathbb{P}(\Homology^4_p(Y,\mathbb{Z}) \otimes \mathbb{C})|x^2=0,x\cdot \overline{x}<0\}^{+}/\Gamma\] the global period domain of cubic fourfolds. The period map \[\mathcal{P}:\mathcal{M}\rightarrow \mathcal{D}/\Gamma\] associates \([Y] \in \mathcal{M}\) to its Hodge structure of the middle cohomology. To describe the image of this period map we define two divisors in \(\mathcal{D}/\Gamma\). Namely 
the set of short roots in \(\Homology^{4}_p(Y,\mathbb{Z})\) determines a \(\Gamma\)-invariant hyperplane arrangement \(H_6\) in \(\mathcal{D}\). Let \(\mathcal{C}_6 \colonequals H_6/\Gamma \subset \mathcal{D}/\Gamma\) be the associated divisor. Similarly the set of long roots in \(\Homology^{4}_p(Y,\mathbb{Z})\) determines a \(\Gamma\)-invariant hyperplane arrangement \(H_2\) in \(\mathcal{D}\). Let \(\mathcal{C}_2 \colonequals H_2/\Gamma \subset \mathcal{D}/\Gamma\) be the associated divisor.
\begin{theorem}[Voisin, Hassett, Laza, Looijenga]\label{period_cubic_fourfolds}
The period map \[\mathcal{P}:\mathcal{M}\rightarrow \mathcal{D}/\Gamma\setminus(\mathcal{C}_2\cup\mathcal{C}_6)\] is an isomorphism.
\end{theorem}

\begin{definition}
    Let \(Y\) be a cubic fourfold. A \(d\)\textit{-labeling} is a rank \(2\) saturated sublattice \(K_d\subseteq A(Y)\) of discriminant \(d\) containing the class \(\eta_Y\).  Denote by \(\mathcal{C}_d\) the locus of the moduli space of smooth cubic fourfolds \(\mathcal{M}\) admitting a \(d\)-labeling.
\end{definition}
Cubic fourfolds with a labeling lie in a codimension one subspace of the moduli space \(\mathcal{M}\), as proved by Hassett \cite{hassett2000special}.
\begin{theorem}[see {\cite[{\(\S3\)}]{hassett2000special}}]
    The locus \(\mathcal{C}_d\) is non-empty if and only if \(d>6\) and \(d\equiv 0,2\) \((6)\). If it is not empty then it is an irreducible divisor in the moduli space \(\mathcal{M}\).
\end{theorem}

Cubic fourfolds which are related to a K3 surface are of remarkable interest, more precisely we recall the following definition.
\begin{definition}
We say that a polarized K3 surface \((S,h)\) of degree \(d\) is \textit{associated} with a cubic fourfold \(Y\) if there exists a \(d\)-labeling \(K_d\) and there is a Hodge isometry \[\Homology^4(Y,\mathbb{Z})\supset K_d^\perp\cong h^\perp(-1)\subset \Homology^2(S,\mathbb{Z})(-1).\]
\end{definition}
In fact, the existence of an associated K3 surface only depends on the discriminant \(d\) of the \(d\)-labeling and not on the lattice \(K_d\).
Labelings for which there exists an associated K3 surface are called \textit{admissible} and are numerically described in \cite{hassett2000special}.

We have a representation map \(\Aut(Y)\rightarrow \bO(\Homology^4(Y,\mathbb{Z}))\) which is injective.
The following theorem gives a precise description of the group of automorphisms of a smooth cubic fourfold.
\begin{theorem}[Hodge theoretic Torelli theorem, see {\cite{zheng2021orbifold}}]\label{torelli_cubic_fourfolds}
Let \(Y_1,Y_2\) be smooth cubic fourfolds. Suppose \(f:\Homology^4(Y_2,\mathbb{Z})\xrightarrow[]{\cong}\Homology^4(Y_1,\mathbb{Z})\) is an isometry of polarized Hodge structures, then there exists a unique isomorphism \(\phi:Y_1\xrightarrow{\cong} Y_2\) such that \(\phi^*=f\). In particular, we have an isomorphism \[\Aut(Y)\cong \bO_{Hdg}(\Homology^4(Y,\mathbb{Z}),\eta_Y)\] where \(\bO_{Hdg}(H^4(Y,\mathbb{Z}),\eta_Y)\) denotes the group of Hodge isometries of \(\Homology^4(Y,\mathbb{Z})\) fixing the class \(\eta_Y\).
\end{theorem}

For this reason it is natural to study isometries of \(\Homology^4_p(Y,\mathbb{Z})\cong\F\).  Similarly to \autoref{inv_contenuto_in_NS}, if \(G\subseteq\Aut(Y)\) is a group of non-symplectic automorphisms then we have the invariant lattice contained in the algebraic lattice, i.e. \(\F^G\subseteq A_p(Y)\), and the transcendental lattice contained in the coinvariant one, i.e. \(\T(Y) \subseteq \F_G\).
We say that a cubic fourfold \(Y\) endowed with a group of non-symplectic automorphisms \(G\subseteq\Aut(Y)\)  is \textit{general} if one of the above inclusions is an equality.

\subsection{Laza--Saccà--Voisin manifolds}\label{Laza--Saccà--Voisin manifolds}
In this section we recall the construction of two geometric models of an ihs manifold of \(\ogten\) type due to Laza--Saccà--Voisin \cite{laza2017hyper,voisin2016hyper}. 
Consider \(Y\subset\mathbb{P}^5\) a smooth cubic fourfold. The dual projective space \((\mathbb{P}^5)^\vee\) parametrizes hyperplane sections of \(Y\), namely \(Y_H=Y\cap H \subset Y\), and \(U\subset(\mathbb{P}^5)^\vee \) is the open set of the smooth ones. Denote by \[ \Jac(Y_H)= \Homology^1(Y_H,\Omega^2_{Y_H})^\vee/\Homology_3(Y_H,\mathbb{Z})\]
the intermediate Jacobian of a smooth hyperplane section \(Y_H\), which is a principally polarized abelian fivefold.
Over \(U\) consider the Donagi--Markman fibration \[ \pi_U:J_U(Y)\rightarrow U\]
whose fiber over the smooth hyperplane section \(Y_H\) consists of the intermediate Jacobian \(\Jac(Y_H)\). It is proved in \cite{donagi1996spectral} that \(J_U(Y)\) is quasi-projective and it admits a symplectic form \(\sigma_U\) for which \(\pi_U\) is a Lagrangian fibration.

Following \cite{voisin2016hyper} there is another Lagrangian fibration \[ \pi^t_U:J^t_U(Y)\rightarrow U\]
whose fiber over \(Y_H\) is given by the torsor \(\Jac^1(Y_H)\) parametrizing degree \(1\) cycles.
\begin{theorem}[see \cite{laza2017hyper,sacca2020birational}, and \cite{voisin2016hyper} for the twisted case]
Let \(Y\) be a smooth cubic fourfold. There exist smooth projective compactifications \(J(Y)\) and \(J^t(Y)\) of \(J_U(Y)\) and \(J^t_U(Y)\) respectively, with projective flat morphisms \(\pi:J(Y)\rightarrow \mathbb{P}^5\) and \(\pi^t:J^t(Y)\rightarrow \mathbb{P}^5\), extending \(\pi_U\) and \( \pi^t_U\) respectively. The manifolds \(J(Y)\) and \(J^t(Y)\) are smooth ihs manifolds of \(\ogten\) type.
\end{theorem}

The compactification \(J(Y)\) is called the \textit{LSV manifold} associated to \(Y\), while \(J^t(Y)\) is called the \textit{twisted LSV manifold} associated to \(Y\).
There is an effective relative theta divisor \(\Theta\subset J(Y)\) obtained as the closure of the union of theta divisors of the smooth fibers, it has the property that \(q_{J(Y)}(\Theta)=-2\). There is another class \(L=\pi^*\mathcal{O}_{\mathbb{P}^5}(1)\), that together with \(\Theta\) span a hyperbolic lattice \(\U_Y:=\langle L,\Theta\rangle\subset \NS(J(Y))\).  For a very general cubic fourfold \(Y\) one has \(\NS(J(Y))=\U_Y\), in particular the family can not be locally complete since there are always two algebraic classes in the LSV manifolds.
 Similarly, in the twisted case there are classes \(L^t,\Theta^t\in\NS(J^t(Y))\) spanning a lattice \(\U^t_Y:=\langle L^t,\Theta^t\rangle\cong\U(3)\) and for a general cubic fourfold we have \(\NS(J^t(Y))=\U^t_Y\).

\section{Classification of non-symplectic automorphisms of OG10}\label{Classification of non-symplectic automorphisms of OG10}

In this section we consider the abstract lattice \(\bL:=\E_8(-1)^{\oplus 2}\oplus \U^{\oplus 3}\oplus\A_2(-1)\) isometric to the second integral cohomology lattice of an ihs manifold of \(\ogten\) type. Recall that there is a unique embedding \(\bL\hookrightarrow\bLambda\) in the unimodular lattice \(\bLambda:=\E_8(-1)^{\oplus 2}\oplus \U^{\oplus 5}\), and the orthogonal complement is the rank two lattice \(\bL^\perp\cong\A_2\). The lattice \(\bL\) is \(3\)-elementary and the discriminant group is the cyclic group \(A_{\bL} \cong \mathbb{Z}/3\mathbb{Z}\). We classify non-symplectic automorphisms of prime order \(p\) by listing the possible invariant and coinvariant lattices of their induced action on the second integral cohomology \(\bL\). This is achieved giving the classification of isometries of the unimodular lattice \(\bLambda\). In the following proposition we prove that any non-symplectic isometry of prime order of \(\bL\) is induced by a non-symplectic automorphism of an ihs manifold of \(\ogten\) type.

\begin{proposition}\label{non sympl autom are effective}
Let \(G \subset \bO(\bL)\) be a group of prime order \(p\) generated by a non-symplectic isometry, then there exists a marked pair \((X,\eta)\) of \(\ogten\) type such that \(G\subset\Aut(X)\) is a group of nonsymplectic automorphisms of prime order \(p\). 
\end{proposition}
\begin{proof} A generator \(\varphi\) of \(G\) is non-symplectic and hence one can endow \(\bL_\mathbb{C}=\bL \otimes_{\mathbb{Z}} \mathbb{C}\) with a weight-two Hodge structure such that \(\bL^G=\bL^{1,1}_{\mathbb{C}}\cap \bL\). By the surjectivity of the period map there exists a manifold \(X\) of \(\ogten\) type and a marking \(\Homology^2(X,\mathbb{Z})\cong \bL\) which is an isomorphism of Hodge structures. By construction \(G\) consists of Hodge isometries, moreover since all the algebraic classes are fixed then the positive cone is fixed, and so the Kähler cone is. This implies that \(\sgn(\bL^G)=(1,\rk(\bL^G)-1)\) and \(\sgn(\bL_G)=(2,\rk(\bL_G)-2)\). In this case we know that \(\Mon^2(X)=\bO^+(\bL)\), and we want to prove that \(G\subset\bO^+(\bL)\). This is clear when \(p\not=2\) since in this case \(p\) is odd and \(\varphi^p=\id\). If \(p=2\) by \cite[Lemma 2.4]{grossi2022nonsymplectic} we have \(\spin(\varphi)=+1\), and \(\Sign(\bL_G)=(2,\rk(\bL_G)-2)\), so that \(\varphi\in\bO^+(\bL)\). Since \(G\subset \Mon^2_{Hdg}(X)\) and a Kähler class is preserved by \(G\) we conclude by the Hodge-theoretic Torelli theorem \cite[Theorem 1.3]{markman2011survey}, as the representation map on the second integral cohomology is injective for ihs manifolds of \(\ogten\) type. In particular \(X\) is projective by Huybrecht's projectivity criterion.
\end{proof}


In the following we relate isometries of the lattice \(\bL\) with isometries of the lattice \(\bLambda\), depending on the image of the group \(G\subset\bO(\bL)\) via the map \(\bO(\bL)\rightarrow \bO(A_{\bL})\). We fix the (unique) primitive embedding \(\bL\hookrightarrow\bLambda\), whose orthogonal complement is \(\bL^{\perp} \cong \A_2\).

\begin{lemma}\label{extension}
 If $\varphi\in \bO(\bL)$ is an isometry such that $\overline{\varphi}=\id\in\bO(A_{\bL})$ then it extends to and element $\Tilde{\varphi}\in \bO(\bLambda)$ acting trivially on $\bL^\perp\subset \bLambda$. If $\varphi\in \bO(\bL)$ is an isometry such that $\overline{\varphi}= -\id\in\bO(A_{\bL})$, then $\varphi$ extends to an isometry $\Tilde{\varphi}\in \bO(\bLambda)$ that acts on the rank \(2\) lattice $\bL^\perp\subset\bLambda$ permuting the generators.
\end{lemma}
\begin{proof}
    Let \(a,b\) be generators of \(\A_2(-1)\subset\bL\) and consider the generator $[\frac{a-b}{3}]=[\frac{a+2b}{3}]$ of $A_{\bL}\cong\mathbb{Z}/3\mathbb{Z}$.
    If \(\varphi \in \bO(\bL)\) is such that \(\overline{\varphi}=\id\) then \(\overline{\varphi}([\frac{a-b}{3}])=[\frac{a-b}{3}]\) hence $\varphi(a-b)=a-b+3w$ with $w\in\bL$. Let \(c,d\) be generators of \(\bL^\perp\cong\A_2\), its discriminant group is also \(\mathbb{Z}/3\mathbb{Z}\) and it is generated by \([\frac{c-d}{3}]\) with discriminant form given by \(q(\frac{c-d}{3})=2/3\). Notice that \(\bL \oplus \A_2\) has an overlattice isometric to \(\bLambda\) which is generated by \(\bL\), \(\frac{a-b+c-d}{3}\) and \(\frac{a+2b+c+2d}{3}\).
   We extend \(\varphi\) to \(\bL \oplus \A_2\) by defining \(\varphi(c)=c\) and \(\varphi(d)=d\) and we obtain an extension \(\tilde{\varphi}\) of \(\varphi\) on \(\bLambda\) as follows: $$ \Tilde{\varphi}( \frac{a-b+c-d}{3})=\frac{\varphi(a-b)+c-d}{3}$$ 
    and 
    $$ \Tilde{\varphi}( \frac{a+2b+c+2d}{3})=\frac{\varphi(a+2b)+c+2d}{3}.$$
    If \(\varphi \in \bO(\bL)\) is such that \(\overline{\varphi}=-\id\) then \(\overline{\varphi}([\frac{a-b}{3}])=[\frac{b-a}{3}]\) hence we extend \(\varphi\) to \(\bL \oplus \A_2\) by imposing \(\varphi(c)=d\) and \(\varphi(d)=c\) and we obtain an extension \(\tilde{\varphi}\) of \(\varphi\) on \(\bLambda\) as follows:
    $$ \Tilde{\varphi}( \frac{a-b+c-d}{3})=\frac{\varphi(a-b)+d-c}{3}$$ and
    $$ \Tilde{\varphi}( \frac{a+2b+c+2d}{3})=\frac{\varphi(a+2b)+d+2c}{3}.$$
\end{proof}

\begin{proposition}\label{signature of L_G and L^G trivial}
    Let \(G \subset \bO(\bL)\) be a subgroup of prime order \(p\) and consider its image \(\overline{G}\subset\bO(A_{\bL})\). Let \(c\) and \(d\) be the generators of \(\A_2=\bL^{\perp}\subset \bLambda\). \begin{itemize}
	    \item 	If \(|\overline{G}|=1\) there exists \(G' \subset \bO(\bLambda)\) a subgroup of order \(p\) such that \(G'\) restricts to \(G\) on \(\bL\) and \(\bL_G = \bLambda_{G'}\). In particular, we have \(\sgn(\bL_G)=\sgn(\bLambda_{G'})\) and  \(\sgn(\bL^G)=\sgn(\bLambda^{G'})-(2,0)\).
	    \item If \(|\overline{G}|=2\) there exists \(G' \subset \bO(\bLambda)\) a subgroup of order \(2\) such that \(G'\) restricts to \(G\) on \(\bL\) and \(\bL_G = (c-d)^{\perp} \subset \bLambda_{G'}\), \(\bL^G \cong (c+d)^{\perp} \subset \bLambda^{G'}\). In particular, we have \(\sgn(\bL_G)=\sgn(\bLambda_{G'})-(1,0)\) and \(\sgn(\bL^G)=\sgn(\bLambda^{G'})-(1,0)\).
	\end{itemize}
\end{proposition}
\begin{proof}
We apply \autoref{extension} to a generator of \(G\), noting that \(\bL_G \cong \bLambda_{G'}\).
\end{proof}

If \(G\subset \bO(\bL)\) is a group of non-symplectic isometries of prime order, we recall that \(\sgn(\bL_G)=(2,\rk(\bL_G)-2)\) by \autoref{non sympl autom are effective}. The two cases treated above lead to two possibilities: \(\sgn(\bLambda_G)=(2, \rk (\bLambda_G)-2)\) if the action of \(G\) is trivial on \(A_{\bL}\), and \(\sgn(\bLambda_G)=(3, \rk (\bLambda_G)-3)\) if the action is non-trivial on \(A_{\bL}\).

\begin{proposition}\label{involutions_lambda}
Let \(G \subset \bO(\bLambda)\) be a group of isometries of order \(2\). If \(\sgn(\bLambda_G)=(2, \rk (\bLambda_G)-2)\) then the pairs \((\bLambda^G, \bLambda_G)\) appear in \autoref{tab:Lambda}, if \(\sgn(\bLambda_G)=(3, \rk (\bLambda_G)-3)\) then the classification of \(\bLambda_G\) corresponds to the one of \(\bLambda^G\) in \autoref{tab:Lambda}.
\end{proposition}
\begin{proof}
   Since $G$ is cyclic of order $2$ and $\bLambda$ is unimodular, then $\bLambda^G$ and $\bLambda_G$ must be $2$-elementary lattices and their discriminant groups are anti-isometric by \autoref{prime_action}, in particular they have the same length. We use \cite[Theorem 3.6.2]{nikulin1980integral} to get all the possible isometry classes of such lattices by varying the signature, the length $a=l(\bLambda_G)=l(\bLambda^G)$ and the parameter $\delta=\delta(\bLambda_G)=\delta(\bLambda^G)$.
   Any pair of such lattices are invariant and coinvariant lattices for the isometry that acts trivially on the invariant lattice and as \(-\id\) on the coinvariant lattice. The lattices are all \(2\)-elementary then the isometry \(-\id\in\bO(\bLambda_G)\) acts trivially on the discriminant group \(A_{\bLambda_G}\).
\end{proof}
The following result is a direct consequence of \cite[Theorem 1.1]{prime_order_brandhorst}.

\begin{proposition}\label{order_p_lambda}
    Let \(G \subset \bO(\bLambda)\) be a group of isometries of prime order \(3\leq p\leq 23\). If \(\sgn(\bLambda_G)=(2, \rk (\bLambda_G)-2)\) then the pair \((\bLambda^G, \bLambda_G)\) appears in \autoref{tab:Lambda_p}.
\end{proposition}
\begin{proof}
If \(G\subset \bO(\bLambda)\) has prime order \(p\) then by \autoref{prime_action} the lattices \(\bLambda_G\) and \(\bLambda^G\) are \(p\)-elementary. The numerical criterion in \cite[Theorem 1.1]{prime_order_brandhorst} determines which pairs of \(p\)-elementary lattices \((\bLambda^G, \bLambda_G)\) can be invariant and coinvariant lattices with respect to the action of the group of isometries \(G\) on \(\bLambda\). Then we use \cite[Corollary 1.13.3, Corollary 1.13.5]{nikulin1980integral} and \cite[\S 1]{rudakov1981surfaces} to compute the isometry class of each lattice, and these pairs are summarized in \autoref{tab:Lambda_p}.
\end{proof}
\begin{remark}
We point out that there is only one conjugacy class for such isometries, apart from the case \(p=23\) where there are three conjugacy classes which are determined by the Steinitz 
class of \(\bLambda_G\). The reader can refer to \cite{prime_order_brandhorst} for more details.
\end{remark}


\begin{theorem}\label{Main theorem section 3}
Let \(G \subset \bO(\bL)\) be a non-symplectic
group of isometries of prime order \(p\). Then there exists a manifold \(X\) of \(\ogten\) type such that \(G \subset \Aut(X)\) if and only if \(2 \leq p \leq 23\) and the pairs \((\bL
^G,\bL_G)\) appear in \autoref{tab:Lid} and \autoref{tab:Lnotid} for \(p=2\), and in \autoref{tab:L_p} for \(p\geq 3\).
\end{theorem}
\begin{proof}
Let \(G\subset\Aut(X)\) be a subgroup of prime order \(p\), notice that by \autoref{p_leq_23} we have \(p\leq 23\), and consider the induced action \(G\subset \bO(\bL)\). According to \autoref{signature of L_G and L^G trivial}, one can extend it to an action $G'\subset\bO(\bLambda)$. The two different extensions lead to two possible cases.

In the first case we have $\bL^G=\A_2^{\perp_{ \bLambda^{G'}}}$
since the orthogonal complement of the embedding $\bL\hookrightarrow\bLambda$ is isometric to $\A_2$ and \(\bL_G=\bLambda_{G'}\). We compute the possible embeddings $\A_2\hookrightarrow\bLambda^{G'}$ and their orthogonal complements for any \(\bLambda^{G'}\) in \autoref{tab:Lambda} and \autoref{tab:Lambda_p}.

In the second case \(p=2\), we consider the unique primitive embedding \(\A_2\hookrightarrow \bLambda\) and observe that in this case \([2]=\langle a+b\rangle\) is \(G'\)-invariant since \(G'\) permutes \(a,b\) that are the generators of \(\A_2=\bL^\perp\subset\bLambda\). In this case, as in \autoref{involutions_lambda}, the lattice \(\bLambda^{G'}\) corresponds to the lattice \(\bLambda_{G'}\) in \autoref{tab:Lambda}. We compute the possible \(\bL^G=[2]^{\perp\bLambda^{G'}}\) for all primitive embeddings \([2]\hookrightarrow\bLambda^{G'}\), and we obtain \(\bL_G\) as the orthogonal complement of the primitive embedding \(\bL^G\hookrightarrow \bL\) when such an embedding exists. 
  
Orthogonal complements of the previous embeddings are uniquely determined up to isometry because of \cite[Proposition 1.14.2]{nikulin1980integral}. 

Viceversa, given a pair \((\bL^G,\bL_G)\), one can endow \(\bL\) with a Hodge structure that makes \(\bL^G\) and \(\bL_G\) the invariant and coinvariant lattices of a non-symplectic automorphism, then we conclude by \autoref{non sympl autom are effective}.
\end{proof}

\section{Non-symplectic automorphisms of order three of a cubic fourfold}\label{Non-symplectic automorphisms of order three of a cubic fourfold}

Finite order automorphisms of a cubic fourfold \(Y \subset \mathbb{P}^5\) are linear transformations of \(\mathbb{P}^5\) that restrict to \(Y\). 

According to \cite[Theorem 2.8]{Gonzales_Liendo}, there exist four families of non-symplectic automorphisms of order three of a cubic fourfold, and there are no non-symplectic automorphisms of prime order greater than three.
In this section we give a lattice theoretic classification of non-symplectic automorphisms of order three of a smooth cubic fourfold. 

The lattice-theoretic classification is carried on with the same techniques used for ihs manifolds of \(\ogten\) type. Definite lattices are often not uniquely determined by their genus, as we have seen in  \autoref{example}), and enumeration of definite lattices is demanding, but the rank of lattices we look for is known thanks to the construction of moduli spaces of cubic fourfolds with a group action given in \cite{yu2020moduli}.

\subsection{The algebraic and transcendental lattices}\label{Subsection alg and tras lattices}

In this section we compute the algebraic and the transcendental lattice of a general cubic fourfold that admits a non-symplectic automorphism of order three (see \autoref{order_three_cubiche_prim}). Here below we recall a classification result of prime order non-symplectic automorphisms of a cubic fourfold \(Y \subset \mathbb{P}^5\) in terms of the induced action of \(\mathbb{P}^5\) and of a general equation of \(Y\). In the following theorem we denote by \(\phi_i^j\) the \(j\)-th automorphism of order \(i\) following the numbering of the list that is given in \cite{yu2020moduli}.

\begin{theorem}[see \cite{Gonzales_Liendo}, and also \cite{fu2016classification,yu2020moduli}]\label{class_cubiche}
Let \(Y = \{F=0\} \subset \mathbb{P}^5\) be a smooth cubic fourfold with a non-symplectic automorphism \(\phi\in\Aut(Y)\) of prime order \(p\). After a linear change of coordinates that diagonalizes \(\phi\) we have \(\phi(x_0:\ldots :x_5)=(\xi^{\sigma_0} x_0:\ldots :\xi^{\sigma_5} x_5)\) and we denote by \((\sigma_0, \ldots, \sigma_5)\) such an action. If \(d\) denotes the dimension of the family of cubic fourfolds endowed with the automorphism \(\phi\), then we have the following possibilities:
\vspace{3pt}
\begin{itemize}\small{
        \item \(\phi_2^1\): \(p=2\), \(\sigma=(0,0,0,0,0,1)\), \(d=14\),
        \[F=L_3(x_0,\dots,x_4)+x_5^2 L_1(x_0,\dots,x_4),\]

        \item \(\phi_2^3\): \(p=2\), \(\sigma=(0,0,0,1,1,1)\), \(d=10\),
        \[F=L_3(x_0,x_1,x_2)+x_0 L_2(x_3,x_4,x_5)+x_1 M_2(x_3,x_4,x_5)+x_2 N_2(x_3,x_4,x_5),\]

        \item \(\phi_3^1\): \(p=3\), \(\sigma=(0,0,0,0,0,1)\), \(d=10\),
        \[F=L_3(x_0,\dots,x_4)+ x_5^3,\]

        \item \(\phi_3^2\): \(p=3\), \(\sigma=(0,0,0,0,1,1)\), \(d=4\),
        \[F=L_3(x_0,\dots,x_3)+ M_3(x_4,x_5),\]

        \item \(\phi_3^5\): \(p=3\), \(\sigma=(0,0,0,1,1,2)\), \(d=7\),
        \[F=L_3(x_0,x_1,x_2)+ M_3(x_3,x_4)+x_5^3+x_3 x_5 L_1(x_0,x_1,x_2)+x_4 x_5 M_1(x_0,x_1,x_2),\]

        \item \(\phi_3^7\): \(p=3\), \(\sigma=(0,0,1,1,2,2)\), \(d=6\)
\[F=x_2 L_2(x_0,x_1)+x_3 M_2(x_0,x_1)+x_4^2 L_1(x_0,x_1)+x_4 x_5 M_1(x_0,x_1)+x_5^2 N_1(x_0,x_1)+ x_4 N_2(x_2,x_3)+x_5 O_2(x_2,x_3)\]}
    \end{itemize}
where \(L_i,M_i,N_i\) and \(O_i\) are homogeneous polynomials of degree \(i\).
\end{theorem}
From now on, non-symplectic automorphisms of order three of a smooth cubic fourfolds are denoted by \(\phi_3^1, \phi_3^2, \phi_3^5, \phi_3^7\).

Recall that if \(Y\) is a smooth cubic fourfold and if \(G \subset \Aut(Y)\) is a finite group, there is an irreducible moduli space \(\mathcal{M}_G\) of cubic fourfolds with an action of \(G\), as constructed in \cite{yu2020moduli} via GIT. Let \(Y\in\mathcal{M}_G\) and let \(\xi\) be the character of the action of \(G\) on \(\Homology^{3,1}(Y)\), denote by \((\F\otimes \mathbb{C})_{\xi}\) the \(\xi\)-eigenspace for a fixed marking \(\gamma:\Homology^4_p(Y,\mathbb{Z})\rightarrow \F\).
\begin{theorem}[see \cite{yu2020moduli}]\label{yu_symmetric_domain}
    There is an isomorphism 
    \[\mathscr{P}_G :\mathcal{M}_G\xrightarrow{\cong} (\mathcal{D}\setminus \mathcal{H})/\Gamma \]
    where \(\mathcal{D}\) is the period domain associated with \((\F\otimes \mathbb{C})_{\xi}\), \(\mathcal{H}\) is a \(\Gamma\)-invariant hyperplane arrangement and \(\Gamma\) is an arithmetic group acting properly and discontinuously on \(\mathcal{D}\).
\end{theorem}

In the previous statement \(\mathcal{D}\) is a symmetric domain of type IV if \(\xi=\overline{\xi}\), and a complex hyperbolic ball otherwise.

The following result gives properties of invariant and coinvariant lattices of a prime order automorphism of a cubic fourfold.
\begin{lemma}\label{shape_inv_coinv_cubic}
Let \(Y\) be a cubic fourfold and let \(G\subset\Aut(Y)\) be a group of prime order \(p\). Then  
\begin{itemize}
\item \(A_{\F_G}\cong (\mathbb{Z}/p\mathbb{Z})^a\) and \(A_{\F^G}\cong (\mathbb{Z}/p\mathbb{Z})^{a}\oplus \mathbb{Z}/3\mathbb{Z}\) for some integer \(a \geq 0\), if \(p\not=3\).
\item \(A_{\F_G}\cong (\mathbb{Z}/p\mathbb{Z})^{a}\) and \(A_{\F^G}\cong (\mathbb{Z}/p\mathbb{Z})^{a\pm 1}\) for some integer \(0\leq a\leq \min(\rk \F_G,\rk \F^G)\), if \(p=3\).
\item \(\F_G\) is positive definite if \(G\) is symplectic
\item \(\F^G\) is positive definite if \(G\) is non-symplectic.
\end{itemize}
\end{lemma}
\begin{proof}
The action of \(G\) on the unimodular lattice \(\Homology^4(Y,\mathbb{Z})\) is trivial on \(\langle \eta_Y\rangle\cong [3]\) hence the action is trivial on the discriminant group \(A_{\langle\eta_Y\rangle}\cong A_{\F}\). Consider the primitive embedding \(\F \hookrightarrow \Homology^4(Y,\mathbb{Z})\). Due to the trivial action of \(G\) on \(A_{\F}\) we have an isometry \(\Homology^{4}(Y,\mathbb{Z})_G\cong \F_G\) and then \(\F_G\) is \(p\)-elementary by \autoref{prime_action}. The possible discriminant groups of \(\F^G\) are determined by \cite[Proposition 1.5.1]{nikulin1980integral}. For the last two statements we refer to \cite[Proposition 2.10]{marquand2023cubic}.
\end{proof}

\begin{lemma}\label{num_crit_long_roots}
Let \(S\) be an even lattice such that \(S \hookrightarrow \F\) is a primitive embedding. Denote by \(\N\) the smallest primitive lattice containing \(\bS\oplus\langle \eta_Y\rangle\subset \Homology^4(Y,\mathbb{Z})\), and assume that \(\bS\oplus \langle \eta_Y\rangle\subset \N\) has finite index embedding of order three. Then there exists \( v \in\bS\) with \(v^2=6\) and divisibility \(\divi(v,\F)=3\) if and only if there exists \( w\in \N\) such that \(w^2=1\).
\end{lemma}
\begin{proof}
Suppose there exists a vector \(v\in\bS\) such that \(v^2=6\) and \(\divi(v,\F)=3\), then \(w=\frac{v-\eta_Y}{3}\in\N\) is such that \(w^2=1\). Viceversa, take \(w\in\N\) such that \(w^2=1\). By the assumption the gluing subgroup of \(\bS\oplus\langle\eta_Y\rangle\subset\N\) is \(\mathbb{Z}/3\mathbb{Z}\) then there is \(v\in S\) with \(\divi(v,\F)=3\) such that \(w=\frac{v+\eta_Y}{3}\), it follows that \(v^2=6\).
\end{proof}

The following lemma gives a formula to compute the rank of \(\F^G\) and \(\F_G\) knowing the dimension of \(\mathcal{M}_G\).
\begin{lemma}\label{dimension_family}
Let \(Y\) be a cubic fourfold and let \(G\subset\Aut(Y)\) be a finite group of prime order \(p\). Let \(\xi\) be the associated character. Let \(\mathcal{M}_G\) be the moduli space of cubic fourfolds with an action of \(G\). The following holds:
\begin{itemize}
\item \(\rk\F^G=\dim\mathcal{M}_G+2\) if \(G\) is symplectic,
\item \(\rk\F_G=\dim\mathcal{M}_G+2\) if \(G\) is non-symplectic and \(p=2\),
\item \(\rk\F_G=2\dim\mathcal{M}_G+2\) if \(G\) is non-symplectic and \(p \geq 3\).
\end{itemize}
\end{lemma}

\begin{proof}
By \cite{yu2020moduli} the dimension of \(\mathcal{M}_G\) equals the dimension of the associated symmetric domain, which is given by \(\dim (\F\otimes \mathbb{C})_{\xi}-2\) if \(\xi=\overline{\xi}\) and by \(\dim (\F\otimes \mathbb{C})_{\xi}-1\) if \(\xi\not=\overline{\xi}\), where we denote by \((\F\otimes \mathbb{C})_{\xi}\) the \(\xi\)-eigenspace.
Suppose \(G\) is generated by a symplectic automorphism, then \(\xi=\overline{\xi} = 1\) and \(\rk \F^G= \dim (\F\otimes \mathbb{C})_1\). 
Let \(G\) be generated by a non-symplectic automorphism of order \(p\). If \(p=2\) we have \(\xi=\overline{\xi}\) and \(\rk \F_G= \dim (\F\otimes \mathbb{C})_{\xi}\). If \(p\geq3\) then \(\xi\not=\overline{\xi}\) and \( \rk \F_G =\dim (\F\otimes \mathbb{C})_{\xi}+\dim (\F\otimes \mathbb{C})_{\overline{\xi}}=2\dim (\F\otimes \mathbb{C})_{\xi} \). 
\end{proof}

Non-symplectic involutions on cubic fourfolds are studied by Marquand in \cite{marquand2023cubic}, and we use a similar approach to study non-symplectic automorphisms of order three.
We recall that if \(Y\) is general cubic fourfold and \(G\subset\Aut(Y)\) is a group of non-symplectic automorphisms then we have \(\T(Y) = \F_G\) and \(\F^G = A_p(Y)\).

\begin{theorem}\label{order_three_cubiche_prim}
Let \(G \subset\bO(\F)\) be a group of non-symplectic isometries of order three. Then there exists a cubic fourfold \(Y\) such that \(G\subset\Aut(Y)\) if and only if the pairs \((\F^G,\F_G)\) appears in \autoref{tab:order_three_non_sympl_cubic_prim}.  For such a general \(Y\) the algebraic lattice \(A(Y)\) appear in \autoref{tab:algebraic_transcendental_lattice_cubic_with_order3_action}, and the class of \(\eta_Y\) is the first basis vector of \(A(Y)\).
\end{theorem}
\begin{proof}
Using \autoref{dimension_family} we can determine the rank of \(\F^G\) and \(\F_G\), since the dimensions of \(\mathcal{M}_G\) are available at \cite[Theorem 6.1]{yu2020moduli}.
Consider the composition of the two natural primitive embeddings \(\F_G\hookrightarrow \F\hookrightarrow\Homology^4(Y,\mathbb{Z})\) and notice that \(\F_G\cong\Homology^4(Y,\mathbb{Z})_G\). 
Since \(\Homology^4(Y,\mathbb{Z})\) is unimodular, we can apply \cite[Theorem 1.1]{prime_order_brandhorst} to determine the list of lengths of \(\F_G\). Note that since \(\F_G\) is indefinite and \(3\)-elementary, to know the length is equivalent to determine its isometry class. To determine the length of \(\F_G\) we consider \(\F^G=A_p(Y)\), and by \autoref{shape_inv_coinv_cubic} we determine the lengths of \(\F^G\). We know by \autoref{period_cubic_fourfolds} that \(A_p(Y)\) does not contain short roots and long roots. We exclude isometry classes of lattices that contain short roots via the sphere packing argument used in \cite{2022symplecticrigidity} or by computer algebra (cf. OSCAR \cite{OSCAR}) for the remaining cases. By \autoref{num_crit_long_roots}, to check the existence of long roots in \(\F^G\) is equivalent to determine the gluing isometry between \(A_{\F^G}\) and \(A_{\langle\eta_Y\rangle}\). This is equivalent to determine the isometry class of \(A(Y)\). To do that by \autoref{num_crit_long_roots}  we compute the possible isometry classes of \(A(Y)\) and we discard the ones containing a vector of square one. Then for the remaining cases, we compute the orthogonal complement of vectors of square three in \(A(Y)\), that are the candidates for \(\eta_Y\), and we check if the orthogonal complement is isometric to \(A_p(Y)\). This happens only in one case, as we expect by \autoref{torelli_cubic_fourfolds}, and this gives the isometry class of \(A(Y)\) and the coordinates of \(\eta_Y\) in \(A(Y)\). The length of \(\T(Y)\) equals the one of \(A(Y)\) because they are orthogonal complements in a unimodular lattice and this allows us also to know the isometry class of \(\T(Y)\). The result of these computations is summarized in \autoref{tab:algebraic_transcendental_lattice_cubic_with_order3_action} and \autoref{tab:order_three_non_sympl_cubic_prim}, where the cases are listed in an increasing order for \(\rk(\F^G)\).
Given such a pair \((\F^G,\F_G)\), the existence of a cubic fourfold with those invariant and coinvariant lattices is guaranteed using \autoref{torelli_cubic_fourfolds} and the trivial action of \(G\) on \(A_{\F}\).
 \end{proof}

\begin{center}
\tiny{\begin{longtable}{lllllll}
\caption{Pairs \((\F^G, \F_G)\) of invariant and coinvariant lattices of a cubic fourfold with a non-symplectic automorphism of order three.}
\label{tab:order_three_non_sympl_cubic_prim} \\
\toprule
No. & \(\rk(\F^G)\) &  \(\F_G = \T(Y) \) & \(\F^G = A_p(Y)\) & \(\sgn(\F_G)\) & \(l(\F^G)\)  \\

		\midrule
		\endfirsthead
		
		\multicolumn{7}{c}%
		{\tablename\ \thetable{}, follows from previous page} \\
	 No. & \(\rk(\F^G)\) &  \(\F_G  = \T(Y)\) & \(\F^G = A_p(Y)\) & \(\sgn(\F_G)\) & \(l(\F^G)\)   \\
		
		\midrule
		\endhead
	
		\multicolumn{7}{c}{Continues on next page} \\
		\endfoot
		
		\bottomrule
		\endlastfoot
\(\phi^1_3\) & \(0\) & \( \U^{\oplus2} \oplus \E_8^{\oplus2} \oplus \A_2 \)& \(\{0\}\)& \((20,2)\)& \(0\)  \\ 

\(\phi^5_3\) & \(6\) & \(\U\oplus \U(3) \oplus \E_6\oplus \A_2^{\oplus3} \)& \(\begin{bmatrix}
 4  &  2  & -1  &  1  &  2  & -2\\
 2  &  4  &  1  &  2  &  1  & -1\\
-1 &   1   & 4  &  2  & -2  & -1\\
 1   & 2  &  2   & 4  & -1  & -2\\
 2   & 1  & -2 &  -1   & 4  & -1\\
-2  & -1  & -1  & -2   &-1   & 4
\end{bmatrix}\)& \((14,2)\) & \(5\)  \\

\(\phi^7_3\) & \(8\) & \( \U\oplus\U(3)\oplus \A_2^{\oplus 5} \)& \(\begin{bmatrix}
 6  &  3  & 3  &  3  &  3  & 3 & -3 & 3\\
 3  &  6  &  0  &  0  &  0  & 0 & -3 & 0\\
3 &   0   & 6  &  0  & 3  &0 & 0 & 3\\
 3   & 0  &  0   & 6  & 0  & 3 & 0 & 3\\
 3   & 0   & 3 & 0   & 6  & 0 & 0 & 3\\
3  & 0  & 0  & 3   &0   & 6 & -3 & 0\\
-3 & -3 & 0 & 0 & 0 & -3 & 6 & 0\\
3 & 0 & 3 & 3 & 3& 0 & 0 & 6
\end{bmatrix}\)& \((12,2)\) & \(8\) \\
        
\(\phi^2_3\) & \(12
\) & \( \U \oplus \U(3) \oplus \A_2^{\oplus3} \)& \(\begin{bmatrix}
4 &   1  & -2 &  -2   & 1 &   2  & -2 &  -1  & -2 &  -2  & -2 &   2\\
  
 1  &  4  & -2  & -1   & 1  &  0  &  0  &  1  & -2   & 1  &  1  &  2\\
 
-2   &-2  &  4 &   2  & -2  &  0  &  0  &  1  &  1  &  1  &  1  & -1\\

-2   &-1  &  2 &   4  &  0  & -2 &   1  &  2  &  0  &  2  &  0 &  -2\\
 1   & 1   &-2 &   0  &  4  & -1 &  -1  & -1  & -1  &  1  &  0  & -1\\
 2    &0  &  0 &  -2  & -1  &  4  & -2  &  0  &  0  & -1  &  0  &  2\\
-2    &0   & 0 &   1  & -1  & -2  &  4  &  1  &  0  &  0   & 1  &  0\\
-1    &1  &  1 &   2  & -1  &  0 &   1  &  4  &  0  &  2  &  1  &  0\\
-2  & -2 &   1 &   0  & -1  &  0   & 0  &  0   & 4  &  0 &   0  & -2\\
-2   & 1   & 1 &   2  &  1  & -1  &  0  &  2  &  0  &  4  &  2  & -1\\
-2   & 1  &  1 &   0  &  0  &  0 &   1  &  1  &  0  &  2  &  4  &  0\\
 2   & 2 &  -1  & -2 &  -1  &  2 &   0  &  0  & -2  & -1 &   0  &  4
\end{bmatrix}\)& \((8,2)\) & \(6\)  \\
    
	\midrule
    	\end{longtable}}
\end{center}
\begin{center}
\tiny{\begin{longtable}{lllllll}
		
		\caption{Pairs \((A(Y),T(Y))\) for a general cubic fourfold with a non-symplectic automorphism of order three.}
	\label{tab:algebraic_transcendental_lattice_cubic_with_order3_action} \\
		
		\toprule
	 No. & \(\rk(A(Y))\) &  \(\T(Y)\) & \(A(Y)\) & \(\sgn(\T(Y))\) & \(l(A(Y))\)   \\

		\midrule
		\endfirsthead
		
		\multicolumn{7}{c}%
		{\tablename\ \thetable{}, follows from previous page} \\
	 No. & \(\rk(A(Y))\) &  \(\T(Y)\) & \(A(Y)\) & \(\sgn(\T(Y))\) & \(l(A(Y))\)   \\
		
		\midrule
		\endhead
	
		\multicolumn{7}{c}{Continues on next page} \\
		\endfoot
		
		\bottomrule
		\endlastfoot
\(\phi^1_3\) & \(1\) & \( \U^{\oplus2} \oplus \E_8^{\oplus2} \oplus \A_2 \)& \([3]\)& \((20,2)\)& \(1\)   \\ 

\(\phi^5_3\) & \(7\) & \( \U\oplus \U(3) \oplus \E_6\oplus \A_2^{\oplus3} \) & \(\begin{bmatrix}
3 & 0 & 0 & 0 & 0 & 0 & 0 \\
 0&4  &  2  & -1  &  1  &  2  & -2\\
0& 2  &  4  &  1  &  2  &  1  & -1\\
0&-1 &   1   & 4  &  2  & -2  & -1\\
0& 1   & 2  &  2   & 4  & -1  & -2\\
0& 2   & 1  & -2 &  -1   & 4  & -1\\
0&-2  & -1  & -1  & -2   &-1   & 4
\end{bmatrix}\) & \((14,2)\) & \(6\)   \\

\(\phi^7_3\) & \(9\) & \( \U \oplus \U(3) \oplus \A_2^{\oplus5} \)& \(\begin{bmatrix}

  3  & 1 &  1  & 1  & 1 &  1 &  1 &  1& 1\\
  1 &  3 &  0 &   0 &   0 &  0 &  0  &  0& 0\\
  1 &  0 &  3 &   0 &   0 &  0 &  0 &   0& 0\\
 1 &  0 &  0 &   3 &   0 &  0 &  0 &   0& 0\\
 1 &  0 &  0 &   0  &  3 &  0  & 0 &   0& 0\\
  1 &  0 &  0 &   0  &  0  & 3  & 0 &   0& 0\\
 1 &  0  & 0 &   0  &  0 &  0  & 3 &   0&0\\
1 &  0  & 0 &   0  &  0 &  0  & 0 &   3&0\\
1 & 0 & 0 & 0 & 0 & 0 & 0 & 0 & 3

\end{bmatrix}\)& \((12,2)\) & \(7\)   \\

\(\phi^2_3\) & \(13\) & \( \U \oplus \U(3) \oplus \A_2^{\oplus3} \)& \(\begin{bmatrix}
 3  &  -1 &  -1 &   1 &  -1 &   1 &  -1 &   1  &  0  &  1 &   0 &   1 &   0\\
-1  &  3 &   1 &  -1  &  1 &  -1&   -1 &  -1&   -1 &   0  &  1 &  -1 &   1\\
-1  &1   & 3   &-1   &-1   &-1  & -1   &-1  & -1   & 0    &1   & 0   & 2\\
 1  & -1 &  -1 &   3 &  -1 &   1&   -1 &   1&    0  &  1   &-1 &   0  & -1\\
-1  &  1 &  -1 &  -1 &   3 &  -1&    1 &   0&    1  & -1  &  1 &   0  &  0\\
 1  & -1 &  -1 &   1 &  -1 &   3&    1 &   0 &  -1  &  0   &-2 &   0  & -1\\
-1  & -1   &-1 &  -1&    1&    1&    3 &   0&    1  & -1  & -1 &   0 &  -1\\
 1  & -1   &-1 &   1 &   0 &   0&    0  &  3 &   2  &  1   & 1  &  0   & 0\\
 0   -&1   &-1 &   0 &   1 &  -1&    1 &   2 &   4  &  1   & 1  &  1  & -1\\
 1    &0   & 0  &  1  & -1 &   0 &  -1  &  1 &   1  &  3   & 1  &  1   &-1\\
 0    &1   & 1  & -1 &   1 &  -2&   -1 &   1 &   1  &  1   & 4  &  1  &  1\\
 1   &-1    &0 &   0 &   0 &   0 &   0  &  0 &   1   & 1    &1   & 3   & 0\\
 0   & 1    &2  & -1 &   0 &  -1&   -1 &   0 &  -1  & -1  &  1 &   0  &  4

\end{bmatrix}\)& \((8,2)\) & \(5\)   \\
    
	\midrule
    	\end{longtable}}
\end{center}

\subsection{Geometry of cubic fourfolds with a non-symplectic automorphism of order three}\label{Geometry of cubic fourfolds with a non-symplectic automorphism of order three}

If \(Y\) is a cubic fourfold, the algebraic lattice \(A(Y)\) encodes geometric information about the cubic and its rationality, since knowing the algebraic lattice allows to determine on which Hassett divisors the cubic lies. 
The purpose of this subsection is to describe the generators of the algebraic lattices \(A(Y)\) in terms of the geometry of a general cubic fourfold that admits a non-symplectic automorphism of order three.

We recall that, referring to Hassett \cite[\S 1.2]{Hassett_rationality_questions}, the existence of two disjoint planes in a smooth cubic fourfold \(Y\) allows to define a rational map \(\mathbb{P}^2\times \mathbb{P}^2 \dashrightarrow Y\) which takes \((p_1,p_2) \in \mathbb{P}^2\times \mathbb{P}^2\) and associates the unique point on \(Y\) given by \(l \cap Y\), where \(l\) is the line trough \(p_1\) and \(p_2\) (if \(l\) is not contained in \(Y\)). The indeterminacy locus of this map is a K3 surface given by a complete intersection of hypersurfaces of bidegrees \((1,2)\) and \((2,1)\), which is a K3 surface associated to \(Y\).

\subsection{Cubic fourfolds with automorphism \(\phi_3^1\) or \(\phi^5_3\)}
A general cubic fourfold \(Y\) with the non-symplectic automorphism of order three \(\phi_3^1\) or \(\phi^5_3\) (see notation in \autoref{class_cubiche}) has a primitive algebraic lattice \(A_p(Y)\) computed in \autoref{tab:order_three_non_sympl_cubic_prim}. 

We show that these cubic fourfolds have no associated K3 surface and contain no planes.
\begin{lemma}
Let \(Y\) be a general cubic fourfold with an automorphism of order three of type \(\phi^1_3\) or \(\phi^5_3\), then \(Y\) does not have an associated K3 surface.
\end{lemma}
\begin{proof}
Suppose the cubic has an associated K3 surface, then there exists a primitive embedding of \(\T(Y)(-1)\) in the K3 lattice \(\U^{\oplus3}\oplus\E_8(-1)^{\oplus2}\). A direct computation shows that this is not possible for transcendental lattices \(\T(Y)\) in \autoref{tab:algebraic_transcendental_lattice_cubic_with_order3_action} corresponding to \(\phi^1_3\) and \(\phi^5_3\). 
\end{proof}

\begin{remark}
According to \cite{huybrechts2017k3}, there is also a notion of twisted K3 surface associated to a cubic fourfold. We observe that if \(Y\) is a general cubic fourfold with the automorphism \(\phi_3^5\), then \(Y\) has an associated twisted K3 surface. Namely \(Y\in\mathcal{C}_{24}\) and \(24\) satisfies condition (ii) of \cite[Proposition 6.23]{huybrechts2023geometry}.
A general cubic fourfold with an automorphism of type \(\phi^1_3\) or \(\phi^5_3\) is conjecturally irrational. Moreover, a general cubic fourfold with the automorphism \(\phi^1_3\) does not lie on any Hassett divisor.
\end{remark}

\begin{lemma}\label{hasset_div_not_2}
 Let \(Y\) be a general cubic fourfold with an automorphism of order three of type \(\phi^5_3\), then \(Y\not\in \mathcal{C}_d\) for \(d\equiv2 \ (6)\).
\end{lemma}
\begin{proof} Any primitive lattice \(K_d\) containing \(\eta_Y\) is of the form \(\langle\eta_Y,a \eta_Y+v\rangle\) for \(0\not=v\in A_p(Y)\) and \(a\in\mathbb{Z}\). We can suppose \(K_d=\langle\eta_Y,v\rangle\) after applying a linear transformation, hence \(d=3k\) in virtue of \autoref{order_three_cubiche_prim}, where \(k=v^2\) and \(k\) is an even number.
\end{proof}

The divisor \(\mathcal{C}_8\) parametrizes cubic fourfolds containing a plane. 
In particular, a general cubic fourfold with an automorphism of type \(\phi^5_3\) does not contain a plane since \(Y\not\in\mathcal{C}_8\). However, as proved in the following propositions, such a cubic belongs to the Hassett divisor \(\mathcal{C}_{12}\), which is the closure of the locus of cubic fourfolds containing a rational cubic scroll.

\begin{proposition}\label{curves on fano variety}
Let \(Y\) be the general cubic fourfold with automorphism \(\phi_3^5\), then the Fano variety of lines \(F(Y)\) contains \(54\) classes of rational curves.
\end{proposition}
\begin{proof}
 By the numerical description of the extremal rays of the Mori cone for hyperkähler manifolds of \(K3^{[2]}\) type in \cite[Proposition 2.12]{Mongardi_Mori}, we know that divisors of square \(-2\) and divisors of square \(-10\) and divisibility \(2\) in \(\Homology^2(F(Y),\mathbb{Z})\) correspond to extremal rays of the Mori cone, in particular they are classes of rational curves on \(F(Y)\). By \cite{Beauville_Donagi_Droites},
 there is an isomorphism of integral Hodge structures
 \(\Homology^{4}_{p}(Y,\mathbb{Z})(-1)\cong\Homology^2_{p}(F(Y),\mathbb{Z})\). We know that there are no short roots in \(A_p(Y)\), then we have  vectors of square \(-10\) in \(\Homology^{1,1}(F(Y),\mathbb{Z})\) and divisibility \(2\) in \(\Homology^2(F(Y),\mathbb{Z})\). 
 Recall that the polarization \(H\) of \(F(Y)\) has square \(6\) and divisibility \(2\) in \(\Homology^2(F(Y),\mathbb{Z})\). The lattice \(A_p(Y)(-1)\cong\Homology^{1,1}(F(Y),\mathbb{Z})\) is \(3\)-elementary, hence vectors \(v\in\Homology^{1,1}(F(Y),\mathbb{Z})\) of divisibility \(2\) in \(\Homology^{2}(F(Y),\mathbb{Z})\) are of the form \(v=H+2a\) with \(a\in \Homology^{1,1}_p(F(Y),\mathbb{Z})\). To get \(v^2=-10\) we need \(4a^2=-16\) and then \(a^2=-4\). 
 By the classification in \autoref{tab:order_three_non_sympl_cubic_prim} there are exactly \(54\) vectors of square \(4\) in \(A_p(Y)\) for \(Y\) a general cubic fourfold with an automorphism of type \(\phi_3^5\).
\end{proof}

\begin{proposition}
    
\label{theorem set of generators for phi_3^5}
Let \(Y\) be the general cubic fourfold with a automorphism of type \(\phi_3^5\), then \(Y\) contains \(27\) families of cubic scrolls \(\{T_i,T^\vee_i\}^{27}_{i=1}\) such that \([T_i]+[T^\vee_i]=2\eta_Y\). Moreover the algebraic lattice \(A(Y)\) is generated by the classes \([T_i]\) for \(i=1,\dots,27\).
\end{proposition}
\begin{proof}
Every class of a rational curve on \(F(Y)\) corresponds to the class of a rational ruled surface on the cubic fourfold \(Y\). Since by construction every such a surface is contained in an hyperplane section of \(Y\), then the cubic surface is a rational cubic scroll. By \autoref{curves on fano variety} we have \(54\) rational cubic scrolls, and by \cite[Example 7.16]{Hassett_Tschinkel_Rational}  on a fixed cubic fourfold these scrolls are parametrized by two distinct copies of \(\mathbb{P}^{2}\). Given a cubic scroll \([T_i]\), there is a residual scroll (the dual cubic scroll) \([T_i^{\vee}]\) which is obtained by intersecting a linear hyperplane and a quadratic hypersurface containing \([T_i]\), as in \cite{hassett2000special}  
(each one of them correspond to a distinct \((-10)\) class, as we found).
One can check (using computer algebra \cite{OSCAR}) that \(A_p(Y)\) contains exactly \(54\) vectors of square \(4\) that generate the entire lattice, moreover the classes \(\alpha_i:=[T_i]-\eta_Y\) and \(\alpha^\vee_i:=[T_i^\vee]-\eta_Y\) have square \(4\).
\end{proof}

\subsection{Cubic fourfolds with automorphism \(\phi^7_3\)}
A general cubic fourfold \(Y\) with the non-symplectic automorphism of order three \(\phi_3^7\) has a primitive algebraic lattice \(A_p(Y)\) computed in \autoref{tab:order_three_non_sympl_cubic_prim}.
We prove that such a cubic fourfold has an associated K3 surface and it is rational. Moreover we show that the algebraic lattice of such a cubic fourfold is generated by classes of planes. 
\begin{lemma}\label{Phi_3^7 are rational}
Let \(Y\) be a general cubic fourfold with an automorphism of type \(\phi^7_3\), then \(Y\) has an associated K3 surface and \(Y\) is rational.
\end{lemma}
\begin{proof}
Consider a general cubic fourfold \(Y\) with automorphism \(\phi^7_3\), then it is easy to see that the lattice \(\T(Y)(-1)\) in \autoref{tab:algebraic_transcendental_lattice_cubic_with_order3_action} corresponding to \(\phi^7_3\) admits a primitive embedding in a K3 lattice, hence \(Y\) has an associated K3 surface.
Moreover if \(Y\) admits an automorphism \(\phi^7_3\) then \(Y\in \mathcal{C}_{14}\), as \(K_{14}=\langle\eta_Y,m_1+m_2\rangle\) gives a labeling if \((\eta_Y,m_1,\dots,m_8)\) is a basis for the matrix of \(A(Y)\) in \autoref{tab:algebraic_transcendental_lattice_cubic_with_order3_action}. It is well known that any cubic fourfold on \(\mathcal{C}_{14}\) is rational \cite{Beauville_Donagi_Droites}.
\end{proof}

\begin{proposition}\label{theorem set of generators for phi_3^7}
  Let \(Y\) be a general cubic fourfold with automorphism \(\phi_3^7\). Then the cubic fourfold contains exactly nine disjoint planes \(F_1, \ldots, F_9\) and a basis of \(A(Y)\) is given by \(\{\eta_Y, [F_1], \ldots, [F_8]\}\).
\end{proposition}

\begin{proof} 
If a plane is contained in \(Y\) it has to be invariant for the action of \(\phi_3^7\) on \(\mathbb{P}^5\).
If a plane is invariant for the action of \(\phi_3^7\) it has equation 
\[F_{\{a,b,c,d,e,f\}}\colonequals \{ax_0=bx_1,cx_2=dx_3,ex_4=fx_5\}.\]
We want to study the intersection \(Y\cap F_{\{a,b,c,d,e,f\}}\). The equation of a cubic fourfold \(Y\) with an action of \(\phi_3^7\) is given in \autoref{class_cubiche} and it is straightforward to see that 
\(Y \cap F_{\{a,b,c,d,e,f\}}\) coincides with 
\[\{
x_0^2 x_2 P_{\{2,1,0\}}(a,b,c,d,e,f)+ x_4^2x_0 P_{\{1,0,2\}}(a,b,c,d,e,f)+x_2^2 x_4 P_{\{0,2,1\}}(a,b,c,d,e,f)=0 \},\]
where \(P_{\{i,j,k\}}\) are polynomials of multi-degrees \((i,j,k)\) in \(\mathbb{P}^1_{[a:b]} \times \mathbb{P}^1_{[c:d]} \times \mathbb{P}^1_{[e:f]}\) . From this description, we deduce that the planes contained in \(Y\) correspond to the set
\begin{equation*}\label{equation_set}
    S:=\{([a:b],[c:d],[e:f]) \mbox{ such that } P_{\{2,1,0\}}=P_{\{1,0,2\}}=P_{\{0,2,1\}}=0\}\subset \mathbb{P}^1_{[a:b]} \times \mathbb{P}^1_{[c:d]} \times \mathbb{P}^1_{[e:f]}.
\end{equation*}
Consider the projections \(p_{a,b},p_{c,d},p_{e,f}\) from \(\mathbb{P}^1_{[a:b]} \times \mathbb{P}^1_{[c:d]} \times \mathbb{P}^1_{[e:f]}\) to its factors,
and denote by \(f_1, f_2, f_3\) the fibers of these three projections. The set \(S\) coincides with the intersection of the following three divisors:
\begin{equation} 
\begin{split}
 2f_1 +f_2, \\
f_1+2f_3, \\
2f_2+f_3,
\end{split}
\end{equation}
then \(S\) is a finite set of degree \((2f_1 +f_2)
(f_1+2f_3)(2f_2+f_3)=9\). This means that the intersection of these three divisors gives exactly nine points in \(\mathbb{P}^1_{[a:b]} \times \mathbb{P}^1_{[c:d]} \times \mathbb{P}^1_{[e:f]}\), that correspond to the nine planes that we denote by \(F_1,\dots,F_9\). Observe that the lattice generated by the classes \(\{\frac{1}{3}\sum^9_{i=1}[F_i],[F_1],\dots,[F_8]\}\) has intersection matrix as in \autoref{tab:algebraic_transcendental_lattice_cubic_with_order3_action} and it is a saturated sublattice of \(A(Y)\), it follows that \(\eta_Y=\frac{1}{3}\sum^9_{i=1}[F_i]\) and \(A(Y)=\langle\eta_Y, [F_1], \ldots, [F_8]\rangle \).
In particular, by construction the planes are disjoint as claimed. 
\end{proof}

If \(Y\) is a cubic fourfold with an automorphism \(\phi_3^7\) then it contains at least two disjoint planes as explained in \autoref{theorem set of generators for phi_3^7} and the associated K3 surface that we find in \autoref{Phi_3^7 are rational} corresponds to the indeterminacy locus of the birational map \(\mathbb{P}^2 \times \mathbb{P}^2 \dashrightarrow Y\) introduced above.

\subsection{Cubic fourfolds with automorphism \(\phi^2_3\)}

A cubic fourfold \(Y\) with the non-symplectic automorphism of order three of type \(\phi_3^2\) has a primitive algebraic lattice \(A_p(Y)\) computed in \autoref{tab:order_three_non_sympl_cubic_prim}. 

In this section we explore the geometry of a cubic fourfold with such an action referring to \cite[\S 2]{Laza_Perl_Zheng}. In particular, we prove that a cubic fourfold \(Y\) with such an action contains \(81\) invariant planes that are related to the existence of three Eckardt points on \(Y\).

\begin{lemma}\label{Phi_3^2 are rational}
     Let \(Y\) be a general cubic fourfold with automorphism \(\phi_3^2\), then \(Y\) has an associated K3 surface and \(Y\) is rational.
\end{lemma}
\begin{proof}
The proof is analogous to that of \autoref{Phi_3^7 are rational}, where the labeling is given by \(K_{14}=\langle \eta_Y,-m_6+m_7\rangle\) where \((\eta_Y,m_1,\dots,m_{12})\) is a basis for the matrix of \(A(Y)\) in \autoref{tab:algebraic_transcendental_lattice_cubic_with_order3_action}. 
\end{proof}

Here we recall the definition of an Eckardt point for a smooth cubic fourfold given in \cite[Definition 1.5]{Laza_Perl_Zheng}.
\begin{definition}
Let \(Y\) be a smooth cubic fourfold. We say that \(p\in Y\) is an \textit{Eckardt point} if \(p\) has multiplicity \(3\) in \(T_pY\cap Y\), equivalently if \(T_pY\cap Y\) is a cone with vertex \(p\) over a cubic surface.  
\end{definition}

We prove that a cubic fourfold with the automorphism \(\phi_3^2\) contains three Eckardt points.

\begin{proposition}
    \label{theorem 81 planes for Phi_3^2}
Let \(Y\) be a general cubic fourfold with an order three non-symplectic automorphism of type \(\phi_3^2\). Then \(Y\) contains exactly \(81\) (invariant) planes, associated to three fixed Eckardt points \(P_1,P_2,P_3\in Y\). 

\end{proposition}
\begin{proof}
A general cubic fourfold \(Y\) with an action of automorphism \(\phi_3^2\) is described by the following equation
 \[L_3(x_0,\dots,x_3)+ M_3(x_4,x_5)=0,\]
where \(L_3\) and \(M_3\) are two homogeneous polynomials of degree \(3\) in \(x_0, \dots, x_3\) and \(x_4,x_5\) respectively.
If \(P=[0:0:0:0:p:q]\) is a zero of the polynomial \(M_{3}(x_4,x_5)\), then we can write \[M_{3}(x_4,x_5)=(qx_4-px_5)M_{2}(x_4,x_5).\] Note that the equation of the tangent plane \(T_PY\) is \(\{qx_4-px_5=0\}\).
Then the intersection \(T_PY \cap Y\) is a cone with vertex \(P\) over the invariant cubic surface \(S=Y\cap\{x_4=x_5=0\}\), hence \(P\) is an Eckardt point. There are three such Eckardt points \(P_1,P_2,P_3\) corresponding to the points \(M_3=0\).
By the argument above each one of them gives an invariant Eckardt point \(P_i=[0:0:0:0:a_i:b_i]\in Y\) for \(i=1,2,3\). 
The cone over the invariant surface with vertex any of the Eckardt points is still invariant and contained in \(Y\). The cubic surface \(S\) contains exactly \(27\) lines, then any Eckardt point determines \(27\) invariant planes passing through the lines of the surface. If \(F\subset Y\) is a plane, then by \cite{voisin1986theoreme}, its cohomology class is such that \([F]^2=3\) and \([F]\cdot \eta_Y=1\), moreover any class with these numerical properties is represented by a unique plane, see \cite{voisin1986theoreme}. Using computer algebra we check that there are exactly \(81\) such classes in \(A(Y)\). 
\end{proof}

\begin{corollary}\label{involution for a cubic with phi_3^2}
Let \(Y\) be a cubic with an automorphism of type \(\phi_3^2\), then \(\mathbb{Z}/3\mathbb{Z}\times D_4\subset \Aut(Y)\), where \(D_4\) denotes the dihedral group.
\end{corollary}
\begin{proof}

The automorphism of order three is given by \(\phi_3^2\) and, according to \cite{Laza_Perl_Zheng}, for any Eckardt point there is an associated involution given by a hyperplane reflection. From the matrix description of the automorphisms, choosing an appropriate set of coordinates, the reflections generate the dihedral group and commute with the automorphism \(\phi_3^2\).
\end{proof}

\begin{proposition}\label{theorem set of generators for phi_3^2}
Let \(Y\) be a general cubic fourfold with automorphism of type \(\phi_3^2\). In the notation of the proof of \autoref{theorem 81 planes for Phi_3^2}, consider the \(6\) disjoint lines on the cubic surface \(S\) contained in \(Y\), and consider the classes of the planes \( [F_{ij}]\) for \(i \in \{1,2,3\}\) and for \(j \in \{1, \ldots, 6\} \) that are the unique planes passing through the Eckardt point \(P_i\) and one of the disjoint lines in \(S\). Then the algebraic lattice \(A(Y)\) has a basis given by the classes of the following planes: 
    \begin{itemize}
     \item \([F_{1,j}]\) for all \(j \in \{1, \ldots, 6\}\) 
     \item \([F_{2,k}]\) for \(k \in \{1, \ldots, 5\}\) 
     \item \([F_{1,0}]\) and \([F_{2,0}]\)
    \end{itemize}
where \(F_{1,0}\) and \(F_{2,0}\) are classes of the the cones of the pullback of a general line via \(S=\Bl_6\mathbb{P}^2\rightarrow \mathbb{P}^2\) with vertex the Eckardt points \(P_1\) and \(P_2\) respectively.
\end{proposition}
\begin{proof}
The intersection numbers of the classes corresponding to planes in the cone with vertex an Eckardt point are described in \cite[Lemma 2.4]{Laza_Perl_Zheng}. The cases where the planes pass thought different Eckardt points can be easily deduced. 
The classes \([F_{1,0}]\) and \([F_{2,0}]\) correspond to a union of planes \(P_0 +P_0'+P_0''\) and \(Q_0+Q_0'+Q_0''\) respectively, where \(P_0\) shares a line with \(P_0'\) and \(P_0''\) and the intersection \(P_0' \cap P_0''\) is a point (and the same behaviour for \(Q_0, Q_0', Q_0''\)) (cf. \cite{reid1997chapters}).   
The rank \(13\) sublattice of \(A(Y)\) generated by \(\{[F_{1,j}],[F_{2,k}],[F_{1,0}],[F_{2,0}]\}_{j,k}\) coincides with \(A(Y)\) since the intersection matrix has the same determinant of the one of \(A(Y)\) in \autoref{tab:algebraic_transcendental_lattice_cubic_with_order3_action}.
\end{proof}

If \(Y\) is a cubic fourfold with an automorphism \(\phi_3^2\) then it contains at least two disjoint planes as explained in \autoref{theorem set of generators for phi_3^2} and the associated K3 surface that we find in \autoref{Phi_3^2 are rational} corresponds to the indeterminacy locus of the birational map \(\mathbb{P}^2 \times \mathbb{P}^2 \dashrightarrow Y\) introduced above.

\begin{remark}
According to \cite[Proposition 3.3]{Cools_Coppens} for a hypersurface of degree \(d=3\) in \(\mathbb{P}^5\) the Eckardt points of the hypersurface have to be at most \(d=3\) points on a line \(l\) not contained in \(Y\). This configuration is actually verified for a general cubic fourfold with an order three non-symplectic automorphism \(\phi_3^2\).
\end{remark}
Note that in \cite[Theorem 1.8]{Laza_maximally} and in \cite{Laza_Perl_Zheng} the authors study the minimal algebraic lattice that a cubic fourfold \(Y\) needs to have in order to admit an Eckardt point, and they find that a cubic fourfold \(Y\) with a non-symplectic involution \(\phi_2^1\) has an Eckardt point. In this situation they prove that \(A_p(Y) \cong \E_6(2)\) and the cubic fourfold has no associated K3 surface by \cite[Theorem 1.2]{marquand2023cubic}. In the case of a cubic fourfold with an automorphism of order three \(\phi_3^2\) the cubic fourfold has more algebraic classes, it admits also the involution \(\phi_2^1\) as we showed in \autoref{involution for a cubic with phi_3^2}, but it is rational and has an associated K3 surface, as we prove in \autoref{Phi_3^2 are rational}.
\begin{remark}\label{embedding E6}
Note that by lattice-theoretic considerations we can detect that a cubic fourfold with automorphism \(\phi_3^2\) contains an Eckardt point. Namely let \(Y\) be a general cubic fourfold with an order three non-symplectic automorphism \(\phi_3^2\), then it is easy to see that there exists a primitive embedding \(\E_6(2) \hookrightarrow A_p(Y)\),
and this is enough to conclude that \(Y\) contains an Eckardt point
by \cite[Proposition 2.8]{Laza_Perl_Zheng}.
\end{remark}

\section{Induced action on Laza-Sacca-Voisin manifolds}\label{Automorphisms on Laza-Sacca-Voisin manifolds of OG10-type}
In this section we study non-symplectic automorphisms of ihs manifolds of \(\ogten\) type that are constructed as Laza--Saccà--Voisin manifolds and we investigate when these automorphisms are induced by non-symplectic automorphisms of the cubic fourfold.

The following result relates the Hodge structure of the middle cohomology of a smooth cubic fourfold to the Hodge structure of the second cohomology of the associated twisted LSV manifold.


\begin{proposition}\label{hodge_isometry_twist}
Let \(Y\) be a smooth cubic fourfold, and let \(J^t(Y)\) be the associated twisted LSV manifold,  then there is a Hodge isometry 
\[\Homology_p^4(Y,\mathbb{Z})(-1)\xrightarrow[]{\cong} (\U_Y^t)^\perp\subset \Homology^2(J^t(Y),\mathbb{Z}).\]   
\end{proposition}

\begin{proof}
The twisted LSV manifold is birational to the Li--Pertusi--Zhao manifold \(\widetilde{\Moduli}_\sigma(2(\lambda_1+\lambda_2),\mathcal{A}_Y)\) by \cite[Theorem 1.3]{li2022elliptic}. Consider the symplectic resolution \(\widetilde{\Moduli}_\sigma(2(\lambda_1+\lambda_2),\mathcal{A}_Y)\) of the moduli space of Bridgeland \(\sigma\)-semi stable objects on the Kuznetsov component \(\mathcal{A}_Y\) of the cubic fourfold \(Y\).
We know by \cite[Example 2.13]{giovenzana2022period} that there is a Hodge embedding \(\Homology_p^4(Y,\mathbb{Z})(-1)\hookrightarrow\Homology^2(\widetilde{\Moduli}_\sigma(2(\lambda_1+\lambda_2),\mathcal{A}_Y),\mathbb{Z})\) with orthogonal complement of type \((1,1)\) and isometric to \(\U(3)\). The moduli space \(\widetilde{\Moduli}_\sigma(2(\lambda_1+\lambda_2),\mathcal{A}_Y)\) is birational to \(J^t(Y)\) by \cite[Theorem 1.3]{li2022elliptic}. For a general cubic fourfold \(Y\) the algebraic lattice of \(\widetilde{\Moduli}_\sigma(2(\lambda_1+\lambda_2),\mathcal{A}_Y)\) is isometric to \(\U(3)\) and the algebraic lattice of \(J^t(Y)\) is isometric to \(\U_Y^t\), hence  composing the Hodge isometries we get the following Hodge isometry
\[\Homology_p^4(Y,\mathbb{Z})(-1)\xrightarrow[]{\cong} (\U_Y^t)^\perp\subset \Homology^2(J^t(Y),\mathbb{Z}).\]
\end{proof}

We give a numerical criterion for an ihs manifold of \(\ogten\) type to be birational to a twisted LSV manifold \(J^t(Y)\). 


\begin{proposition}\label{birat_twisted_LSV}
Let \(X\) be an ihs manifold of \(\ogten\) type. 
There exists a smooth cubic fourfold \(Y\) such that \(X\) is birational to \(J^t(Y)\) if and only if there is a primitive embedding \(\U(3)\hookrightarrow\NS(X)\hookrightarrow \bL\) such that the composition \(\U(3)\hookrightarrow\bL\) has embedding subgroup \(\mathbb{Z}/3\mathbb{Z}\) and \(\U(3)^{\perp_{ \NS(X)}}\cap \mathcal{W}^{pex}_{\ogten}(X)=\emptyset\).
\end{proposition}

\begin{proof}
If \(X\) and \(J^t(Y)\) are birational, then there is an Hodge isometry \(\Homology^2(J^t(Y),\mathbb{Z})\cong\Homology^2(X,\mathbb{Z})\), so the embedding \(\U_Y^t\subset\NS(J^t(Y))\) induces the embedding of \(\U(3)\) in \(\NS(X)\). We know by \autoref{hodge_isometry_twist} that the lattice \((\U_Y^t)^\perp\subset \Homology^2(J(Y),\mathbb{Z})\) is Hodge-isometric to \(\Homology_p^4(Y,\mathbb{Z})(-1)\). The description of the image of the period map of cubic fourfold \autoref{period_cubic_fourfolds} ensures that there are no long or short roots in \(\Homology_p^4(Y,\mathbb{Z})\), which coincide with classes of prime exceptional divisors of ihs manifolds of \(\ogten\) type up to a sign, as we point out in \autoref{Automorphisms of ihs manifolds of OG10 type}.

Viceversa, if there are no short or long roots in \(\U(3)^\perp\subset\NS(X)(-1)\) then by \autoref{period_cubic_fourfolds} we know that \(\U(3)^\perp\subset\Homology^2(X,\mathbb{Z})(-1)\) is Hodge isometric to \(\Homology_p^4(Y,\mathbb{Z})\) for a cubic fourfold \(Y\), which is also Hodge isometric to \((\U_Y^t)^\perp\subset\Homology^2(J^t(Y),\mathbb{Z})(-1)\). Extending the Hodge isometry to the entire lattice \(\Homology^2(X,\mathbb{Z})\cong\Homology^2(J^t(Y),\mathbb{Z})\) via \cite[Corollary 1.5.2]{nikulin1980integral}, we conclude that \(X\) and \(J^t(Y)\) are birational.
\end{proof}

\begin{remark}
     Combining \cite[Theorem 1.3]{li2022elliptic} with \cite[Theorem 4.3]{giovenzana2022period} we obtain that \(J(Y)\) is birational to \(J^t(Y)\) if and only if the cubic fourfold \(Y\) admits a \(d\)-labeling with \(d\equiv 2 (6)\). As an application of the results in \autoref{Non-symplectic automorphisms of order three of a cubic fourfold}, we obtain that \(J(Y)\) and \(J^t(Y)\) are birational if \(Y\) is general with an automorphism \(\phi_3^7\) or \(\phi_3^2\) but they are not birational if \(Y\) is general with an automorphism \(\phi_3^1\) or \(\phi_3^5\). We also point out that by \cite{marquand2023cubic} it follows that \(J(Y)\) and \(J^t(Y)\) are birational if \(Y\) is general with one of the non-symplectic involutions \(\phi_2^1\) or \(\phi_2^3\). We collect this information in \autoref{tab:induced_actions_LSV}.
\end{remark}
We study the induced action on the second integral cohomology of birational transformations of LSV manifolds induced by automorphisms of the cubic fourfold. 
Recall that an automorphism of a cubic fourfold \(Y\) induces a birational transformation of the LSV manifold \(J(Y)\) and the same holds for the twisted case. 
\begin{remark}
An automorphism of the cubic fourfold \(Y\) induces a birational transformation of the twisted LSV manifold \( J^t(Y)\). In fact, let \(u:\mathcal{V}_U\rightarrow U\subset \mathbb{P}^5 \) be the family of smooth hyperplane sections of \(Y\), recall that Deligne-Belinson cohomology gives an exact sequence of sheaves of groups 
    \[0\rightarrow J_U(Y)\rightarrow \Homology_{\mathcal{D}}^4(\mathcal{V}_U,\mathbb{Z}(2))\xrightarrow{c} R^4 u_* \mathbb{Z}\rightarrow 0\]
 where the sheaf \(R^4 u_* \mathbb{Z}\) is canonically isomorphic to \(\mathbb{Z}\). By definition \(J_U^t=c^{-1}(1)\), functoriality gives an automorphism of \(J_U^t\).
\end{remark}

Let \(\phi\in\Aut(Y)\), we denote by \(\widetilde{\phi}\in\Bir(J(Y))\) and by \(\widetilde{\phi}^t\in\Bir(J^t(Y))\) the induced birational transformations.

\begin{lemma}[see {\cite[Lemma 7.1]{mongardi2022birational}} or {\cite[Lemma 3.2]{sacca2020birational}}]\label{non-symplectic on Y induces non-symplectic on X}
 Let \(Y\) be a cubic fourfold. An automorphism \(\phi\in \Aut(Y)\) is symplectic if and only if \(\widetilde{\phi} \in \Bir(J(Y))\) is symplectic, if and only if \(\widetilde{\phi}^t \in \Bir(J^t(Y))\) is symplectic.
\end{lemma}

\begin{proposition}\label{inv_to_inv}
Let \(Y\) be a general cubic fourfold with a non-symplectic automorphism \(\phi\in\Aut(Y)\) of finite order and let \(\widetilde{\phi}\in\Bir(J(Y))\) and \(\widetilde{\phi}^t\in\Bir(J^t(Y))\) be the induced birational transformations on the manifolds \(J(Y)\) and \(J^t(Y)\). Then we have  \[ \NS(J(Y))=  \Homology^2(J(Y),\mathbb{Z})^{\widetilde{\phi}}, \]\[ \NS(J^t(Y))=  \Homology^2(J^t(Y),\mathbb{Z})^{\widetilde{\phi}^t}.\] Moreover, we have an isometry \[(\Homology_p^4(Y,\mathbb{Z}))^\phi(-1)\cong (\U_Y^{t\perp})^{\widetilde{\phi}^t}\subset \Homology^2(J^t(Y),\mathbb{Z})^{\widetilde{\phi}^t}.\]
\end{proposition}

\begin{proof}We prove \(\NS(J(Y))=  \Homology^2(J(Y),\mathbb{Z})^{\widetilde{\phi}}\), the twisted case is analogous.

By \cite[Section 3.1]{sacca2020birational} every birational transformation \(\widetilde{\phi}\in \Bir(J(Y))\) of a LSV manifold which is induced by an automorphism \(\phi \in \Aut(Y)\) of the cubic fourfold \(Y\) fixes the two generators of \(\U_Y\). We consider the isometry \(\widetilde{\phi}\) restricted on \(\U_Y^{\perp}\), which is still a Hodge isometry since the two generators in \(\U_Y\) are of \((1,1)\) type.
According to \cite[Lemma 7.1]{mongardi2022birational} there is an isogeny of Hodge structures 
\[\alpha \colon \Homology_p^4(Y,\mathbb{Z})(-1) \rightarrow \U_Y^\perp\subset \Homology^2(J(Y),\mathbb{Z}),\]
where \(\alpha\) is the restriction of the morphism \([Z]_*\circ q^*:\Homology^4(Y,\mathbb{Z})\rightarrow \Homology^2(J(Y),\mathbb{Z})\). Here \(q:\mathcal{U}_Y\rightarrow Y\) is the inclusion of linear sections and \([Z]_*(x)=\pi_{1*}(\pi_2^*x.Z)\) where \(Z\in\CH^2(J(Y)\times_{\mathbb{P}^5}\mathcal{U}_Y)_\mathbb{Q}\) is the closure of a distinguished cycle \(Z_U\in\CH^2(J_U(Y)\times_{\mathbb{P}^5}\mathcal{U}_U)_\mathbb{Q}\), and \(\pi_1,\pi_2\) are the respective projections.
For an integer \(k\), it is well defined the cycle \(\phi^k(Z_U)\in\CH^2(J_U(Y)\times_{\mathbb{P}^5}\mathcal{U}_U)_\mathbb{Q}\), consider its closure \(\overline{\phi^k(Z_U)}\in\CH^2(J(Y)\times_{\mathbb{P}^5}\mathcal{U}_Y)_\mathbb{Q}\). Replacing \(Z\) with 
\[\widetilde{Z}=\frac{1}{\sqrt{\ord(\phi)}}\sum_{k=0}^{\ord(\phi)} \overline{\phi^k(Z_U)}\] in the above definition, since \(\phi(\widetilde{Z})=\widetilde{Z}\), one gets a \((\phi, \widetilde{\phi})\)-equivariant isogeny
of Hodge structures \[\widetilde{\alpha} \colon \Homology_p^4(Y,\mathbb{Z})(-1) \rightarrow \U_Y^\perp\subset \Homology^2(J(Y),\mathbb{Z}).\]  By the equivariance of \(\widetilde{\alpha}\) it follows that \(\Homology^{4}_p(Y,\mathbb{Z})^{\phi}(-N) \subseteq (\U_Y^{\perp})^{\widetilde{\phi}}\) for some integer \(N>0\), since the cubic fourfold is general we have \(\Homology^{4}_p(Y,\mathbb{Z})^{\phi} \cong A_p(Y)\) and there are finite index embeddings \[ A_p(Y)(-N) \cong \Homology^{4}_p(Y,\mathbb{Z})^{\phi}(-N) \subseteq (\U_Y^{\perp})^{\widetilde{\phi}} \subseteq (\U_Y^{\perp})^{1,1},\]
we conclude that \((\U_Y^{\perp})^{\widetilde{\phi}}= (\U_Y^{\perp})^{1,1}\) by observing that the last embedding is primitive and the lattices have the same rank.

For the twisted case, there is an isometry \(A_p(Y)(-1) \cong (\U_Y^{t\perp})^{1,1}\) by \autoref{hodge_isometry_twist}, and since the automorphism is general on  \(Y\) then \((\Homology_p^4(Y,\mathbb{Z}))^\phi(-1)=A_p(Y)(-1)\cong (\U_Y^{t\perp})^{1,1}\). By the first part of the statement we conclude that \((\Homology_p^4(Y,\mathbb{Z}))^\phi(-1)=A_p(Y)(-1)\cong (\U_Y^{t\perp})^{1,1}=((\U_Y^{t\perp})^{\widetilde{\phi}^t}\).
\end{proof}

We can recover a cubic fourfold with an automorphism from a certain automorphism of a manifold of \(\ogten\) type as it is proved in the following result.
\begin{proposition}\label{recover_cubic}
Let \(X\) be a manifold of \(\ogten\) type with a marking \(\Homology^2(X,\mathbb{Z})\cong\bL\) and let \(f\in\Bir(X)\) be a general non-symplectic birational transformation of prime order. Then \(f\) is induced by an automorphism of a cubic fourfold if and only if it acts trivially on the discriminant group \(A_{\bL}\), and there is a primitive embedding \(\U(3)\hookrightarrow\NS(X)\hookrightarrow \bL\) such that the composition \(\U(3)\hookrightarrow\bL\) has embedding subgroup \(\mathbb{Z}/3\mathbb{Z}\) and \(\U(3)^{\perp_{\NS(X)}}\cap\mathcal{W}^{Pex}_{\ogten}=\emptyset\).
  
\end{proposition}
\begin{proof} 
By generality assumption we have \(\bL^{f} = \NS(X)\). If \(f\) is induced by an automorphism of a cubic fourfold the statement follows from \autoref{inv_to_inv} and \autoref{torelli_cubic_fourfolds}.

Viceversa, consider the lattice \(\N:=\U(3)^{\perp_{\bL(-1)}}\) and observe that by hypothesis this is isometric to the lattice \(\U^{\oplus2}\oplus\E_8^{\oplus 2}\oplus \A_2\). Since \(\U(3)^{\perp_{\NS(X)}}\cap\mathcal{W}^{Pex}_{\ogten}=\emptyset\) and prime exceptional divisors of an \(\ogten\) type manifold correspond to short and long roots up to the sign, by \autoref{period_cubic_fourfolds} there exits a smooth cubic fourfold \(Y\) such that \(\Homology_p^4(Y,\mathbb{Z})\cong\N\) is a Hodge isometry. 
   The restriction of the isometry \(f\) to \(\N\) extends to an isometry of \(\Homology^4(Y,\mathbb{Z})\) that fixes the class \(\langle \eta_Y\rangle=\Homology^4_p(Y,\mathbb{Z})^\perp\subset\Homology^4(Y,\mathbb{Z})\) if and only if \(f\) acts trivially on the discriminant group \(A_{\bL(-1)}\cong A_{\N}\), as in our assumption. We conclude by  \autoref{inv_to_inv} and \autoref{torelli_cubic_fourfolds}.
\end{proof}

\begin{lemma}
    
\label{regular_action}
Let \(Y\) be a general cubic fourfold with a non-symplectic automorphism \(\phi\in\Aut(Y)\) of finite order. Then the induced birational transformations \(\widetilde{\phi}\in\Bir(J(Y))\) and \(\widetilde{\phi}^t\in\Bir(J^t(Y))\) are automorphisms.
\end{lemma}
\begin{proof}
    From \autoref{inv_to_inv} we know that from generality assumption \(A_p(Y)=\Homology^4_p(Y,\mathbb{Z})^\phi\) then \(\NS(J(Y))=\Homology^2(J(Y),\mathbb{Z}))^{\widetilde{\phi}}\) and \(\NS(J^t(Y))=\Homology^2(J^t(Y),\mathbb{Z}))^{\widetilde{\phi}^t}\). By the Hodge theoretic Torelli theorem for ihs manifolds \cite[Theorem 1.3]{markman2011survey} we conclude that the transformations are regular.
\end{proof}

\begin{remark}
    This shows that the converse of \cite[Proposition 3.11]{sacca2020birational} does not hold. Namely, some birational transformations induced by automorphisms of a cubic fourfold extend to regular automorphisms even if the fibers of the Lagrangian fibration are reducible. Reducible fibers arise for cubic fourfolds containing planes or cubic scrolls, see \cite{marquand2023defect}.
\end{remark}
\begin{theorem}\label{induced_by_cubic}
Let \(Y\) be a general cubic fourfold with a non-symplectic automorphism of prime order. Then the invariant and coinvariant lattices of the induced action on \(J^t(Y)\) is described in \autoref{tab:induced_actions_LSV}.
\end{theorem}

\begin{proof}
We have a classification of possible invariant and coinvariant sublattices of \(\Homology_p^4(Y,\mathbb{Z})\) for non-symplectic automorphisms of a general \(Y\). This is the content of \cite[Theorem 1.1]{marquand2023cubic} for the case of involutions, and of \autoref{order_three_cubiche_prim} for the case of automorphisms of order three. Using \autoref{inv_to_inv} and \autoref{tab:Lid} we conclude the statement.
\end{proof}

\begin{remark}
    As \autoref{tab:induced_actions_LSV} shows, for a general cubic fourfold \(Y\) with an automorphism \(\phi_3^1\) and \(\phi_3^5\) the manifolds \(J(Y)\) and \(J^t(Y)\) are not birational. The automorphism \(\phi_3^1\) fixes only the class \(\eta_Y\), then by \autoref{inv_to_inv} the induced action on the LSV manifold \(J(Y)\) has invariant lattice given by \(\U\) and coinvariant lattice given by \(\U^{\oplus2}\oplus\E_8(-1)^{\oplus2}\oplus \A_2(-1)\). If \(Y\) is a general cubic fourfold with an automorphism \(\phi_3^5\), the invariant lattice of the induced automorphism on the LSV manifold \(J(Y)\) has signature \((1,7)\), and it corresponds to one of the pairs from 33 to 39 in \autoref{tab:L_p}. We can exclude cases \(37\) and \(39\) because they do not admit a primitive embedding of \(\U\).
\end{remark}
\begin{theorem}\label{regular_autom}
Let \(X\) be an ihs manifold of \(\ogten\) type and
let \(G= \langle f \rangle \in \Aut(X)\) be a group of automorphisms generated by a non-symplectic automorphism of prime order. If the pair \((\bL^G,\bL_G)\) appears in \autoref{tab:induced_actions_LSV}, then there exists a cubic fourfold \(Y\) with an automorphism \(\phi\in\Aut(Y)\) such that \(X\) is birational to \(J^t(Y)\) and the action of \(f\) on \(X\) coincide with the action of \(\widetilde{\phi}^t\) on \(J^t(Y)\), up to birational transformations.
\end{theorem}
\begin{proof}
    Suppose \(G\subset\Aut(X)\) is a group of automorphisms as in \autoref{tab:induced_actions_LSV}, then the hypothesis of \autoref{birat_twisted_LSV} are satisfied hence \(X\) is birational to a twisted LSV manifold. By \autoref{recover_cubic} we conclude the proof. 
\end{proof}

\appendix
\section{Tables of invariant and coinvariant lattices of non-symplectic automorphisms of OG10}
Here we collect tables of (isometry classes) of invariant and coinvariant lattices of the action of \(G\) on \(\bLambda\) and \(\bL\), i.e. pairs \((\bL^G,\bL_G)\) and \((\bLambda^{G},\bLambda_{G})\) of a group \(G\) of prime order isometries, with given signatures.
\subsection{\(p=2\)}
\begin{center}
\tiny{\begin{longtable}{lllllll}
		
		\caption{Pairs \((\bLambda^{G},\bLambda_{G})\) for \({G} \subset \bO(\bLambda)\) of prime order \(p=2\) and \(\sgn(\bLambda_{G})=(2, \rk(\bLambda_{G})-2)\), i.e. trivial action on the discriminant group.}
	\label{tab:Lambda} \\
		
		\toprule
	 No. & \(\rk(\bLambda^{{G}})\) &  \(\bLambda_{G}\) & \(\bLambda^{G}\) & \(\sgn(\bLambda_{G})\) & \(a\) & \(\delta\)  \\

		\midrule
		\endfirsthead
		
		\multicolumn{7}{c}%
		{\tablename\ \thetable{}, follows from previous page} \\
		\midrule
	No. & \(\rk(\bLambda^{{G}})\) &  \(\bLambda_{G}\) & \(\bLambda^{G}\) & \(\sgn(\bLambda_{G})\) & \(a\) & \(\delta\)  \\
		\midrule
		\endhead
	
		\multicolumn{7}{c}{Continues on next page} \\
		\endfoot
		
		\bottomrule
		\endlastfoot

		\(1\) & \(4\) & \(\E_8(-1)^{\oplus2} \oplus \U^{\oplus2} \oplus [-2]^{\oplus2}\)& \(\U\oplus [ 2 ] ^{\oplus 2 }\)& \((2,20)\) & \(2\) & \(1\) \\
  
         \(2\) & \(4\) & \(\E_8(-1)^{\oplus2} \oplus\U\oplus [2] \oplus [-2]^{\oplus 3}\)& \([ 2 ] ^{\oplus 3 }\oplus [-2]\)& \((2,20)\) & \(4\) & \(1\) \\
		\(3\) & \(5\) & \(\E_8(-1)^{\oplus2} \oplus \U^{\oplus2} \oplus [-2]\) & \(\U^{\oplus 2} \oplus [ 2 ] \) &  \((2,19)\) & \(1\) & \(1\) \\
  
  	\(4\) & \(5\) & \(\E_8(-1)^{\oplus2} \oplus \U\oplus [2]\oplus [-2]^{\oplus 2} \) & \(\U\oplus [2]^{\oplus 2} \oplus [-2 ] \) &  \((2,19)\) & \(3\) & \(1\) \\

   \(5\) & \(5\) & \(\E_8(-1)^{\oplus2} \oplus  [2]^{\oplus 2}\oplus [-2]^{\oplus 3} \) & \( [2]^{\oplus 3} \oplus [-2 ]^{\oplus 2} \) &  \((2,19)\) & \(5\) & \(1\) \\
   
		\(6\) & \(6\)  & \(\E_8(-1)^{\oplus2} \oplus \U^{\oplus2}\) & \(\U^{\oplus3}\)&  \((2,18)\) & \(0\) & \(0\) \\
 \(7\) & \(6\)  & \(\E_8(-1)^{\oplus2} \oplus \U\oplus\U(2)\) & \(\U^{\oplus2}\oplus \U(2)\)&  \((2,18)\) & \(2\) & \(0\) \\
 
  	\(8\) & \(6\)  & \(\E_8(-1)^{\oplus2} \oplus \U\oplus[2] \oplus[-2]\) & \(\U^{\oplus 2} \oplus[2] \oplus[-2]\)  & \((2,18)\) & \(2\) & \(1\) \\

\(9\) & \(6\)  & \(\E_8(-1)^{\oplus 2} \oplus \U(2)^{\oplus2}\) & \(\U\oplus\U(2)^{\oplus2}\) &  \((2,18)\) & \(4\) & \(0\) \\

		\(10\) & \(6\)  & \(\E_8(-1)^{\oplus 2} \oplus [ -2 ] ^{\oplus 2 } \oplus [ 2 ]^{\oplus2}\) & \(\U\oplus [2 ] ^{\oplus 2 } \oplus [-2 ]^{\oplus2}\) &  \((2,18)\) & \(4\) & \(1\) \\

        \(11\) & \(6\)  & \(\E_8(-1) \oplus \U\oplus\D_4(-1)^{\oplus2}\oplus \U(2)\) & \(\U(2)^{\oplus3}\) &  \((2,18)\) & \(6\) & \(0\) \\
        
        \(12\) & \(6\)  & \(\E_8(-1) \oplus \U\oplus\D_4(-1)^{\oplus2}\oplus [2]\oplus[-2]\) & \([2]^{\oplus3} \oplus[-2]^{\oplus3}\) &  \((2,18)\) & \(6\) & \(1\) \\
		
		\(13\) & \(7\) & \(\E_{8}(-1)^{\oplus2} \oplus \U\oplus [ 2 ]\) & \(\U^{\oplus3} \oplus [-2]\) &  \((2,17)\) & \(1\) & \(1\) \\

		\(14\) & \(7\)  & \(\E_{8}(-1)^{\oplus2} \oplus [ 2 ] ^{\oplus 2} \oplus [ -2 ] \) & \(\U^{\oplus2} \oplus [ 2 ] \oplus [ -2 ]^{\oplus 2}\) &  \((2,17)\) & \(3\) & \(1\) \\

		\(15\) & \(7\) & \(	\E_{8}(-1) \oplus \U^{\oplus 2} \oplus \D_4(-1) \oplus [-2]^{\oplus3} \) & \(\U\oplus [ 2 ]^{\oplus2} \oplus [-2]^{\oplus3}\) &  \((2,17)\) & \(5\) & \(1\) \\

  \(16\) & \(7\) & \(	\E_{8}(-1) \oplus \U\oplus \D_4(-1) \oplus[2]\oplus [-2]^{\oplus4} \) & \( [ 2 ]^{\oplus3} \oplus [-2]^{\oplus4}\) &  \((2,17)\) & \(7\) & \(1\) \\

		\(17\) & \(8\) & 
  \(\E_8(-1)^{\oplus 2}\oplus [2]^{\oplus 2} \)& \(\U^{\oplus3} \oplus [-2]^{\oplus 2}\)&  \((2,16)\) & \(2\) & \(1\) \\

		\(18\) & \(8\) & \(\E_{8}(-1) \oplus \U^{\oplus 2}\oplus \D_4(-1) \oplus [-2]^{\oplus2}\) & \(\U^{\oplus 2} \oplus [ 2 ] \oplus [ -2 ]^{\oplus3}\) &  \((2,16)\) & \(4\) & \(1\) \\

		\( 19\) & \(8\) & \(\E_{8}(-1) \oplus \D_4(-1)^{\oplus 2}\oplus [2]^{\oplus 2} \) & \(\U\oplus [2]^{\oplus 2} \oplus [-2]^{\oplus 4} \) &  \((2,16)\) & \(6\) & \(1\) \\

        	\( 20\) & \(8\) & \(\E_{8}(-1) \oplus \U\oplus  [2] \oplus [-2]^{\oplus 7} \) & \( [2]^{\oplus 3} \oplus [-2]^{\oplus 5} \) &  \((2,16)\) & \(8\) & \(1\) \\

		\(21\) & \(9\) & \(\E_{8}(-1) \oplus \U^{\oplus 2} \oplus \D_4(-1)\oplus [-2] \)& \(\U^{\oplus 3} \oplus [-2]^{\oplus 3} \) &  \((2,15)\) & \(3\) & \(1\) \\

		\(22\) & \(9\) & \(\E_{8}(-1) \oplus \U\oplus \D_4(-1)\oplus [2 ] \oplus [ -2 ]^{\oplus2}\) & \(\U^{\oplus 2} \oplus [-2]^{\oplus 4} \oplus [2]\) &  \((2,15)\) & \(5\) & \(1\) \\
		
	    \(23\) & \(9\) & \(\E_{8}(-1) \oplus \D_4(-1)\oplus [2 ]^{\oplus 2} \oplus [ -2 ]^{\oplus3}\) & \(\U\oplus [ 2 ]^{\oplus2} \oplus [-2]^{\oplus 5}\) &  \((2,15)\) & \(7\) & \(1\) \\
     
         \(24\) & \(9\) & \(\E_{8}(-1) \oplus  [2 ]^{\oplus 2} \oplus [ -2 ]^{\oplus7}\) & \( [ 2 ]^{\oplus3} \oplus [-2]^{\oplus 6}\) &  \((2,15)\) & \(9\) & \(1\) \\

\(25\) & \(10\) & \( \E_8(-1)\oplus \U^{\oplus 2}\oplus \D_4(-1) \) & \(\U^{\oplus 3} \oplus \D_4(-1)\) &  \((2,14)\) & \(2\) & \(0\) \\

          \(25\) & \(10\) & \( \E_8(-1)\oplus \U\oplus \U(2)\oplus \D_4(-1) \) & \(\U^{\oplus 2}\oplus \U(2) \oplus \D_4(-1)\) &  \((2,14)\) & \(4\) & \(0\) \\

        \(26\) & \(10\) & \(\E_8(-1)\oplus \U\oplus\D_4(-1)\oplus [2]\oplus [-2] \) & \(\U^{\oplus 3} \oplus [-2]^{\oplus 4}\) &  \((2,14)\) & \(4\) & \(1\) \\

      \(27\) & \(10\) & \( \U^{\oplus 2} \oplus \D_4(-1)^{\oplus 3} \) & \(\U\oplus\U(2)^{\oplus 2}\oplus \D_4(-1)\) &  \((2,14)\) & \(6\) & \(0\) \\

       \(28\) & \(10\) & \(\E_8(-1) \oplus \D_4(-1)\oplus  [2]^{\oplus 2} \oplus [-2]^{\oplus 2} \) & \(\U^{\oplus 2}\oplus [2]  \oplus [-2]^{\oplus 5}\) &  \((2,14)\) & \(6\) & \(1\) \\

         \(29\) & \(10\) & \(\U\oplus\U(2)\oplus\D_4(-1)^{\oplus 3} \) & \( \U(2)^{\oplus 3}  \oplus \D_4(-1)\) &  \((2,14)\) & \(8\) & \(0\) \\

        \(30\) & \(10\) & \(\U^{\oplus 2}\oplus \D_4(-1)^{\oplus 2} \oplus [-2]^{\oplus 4} \) & \(\U\oplus [2]^{\oplus 2}  \oplus [-2]^{\oplus 6}\) &  \((2,14)\) & \(8\) & \(1\) \\
        
        \(31\) & \(10\) & \(\U\oplus \D_4(-1)^{\oplus 2} \oplus [2]\oplus[-2]^{\oplus 5} \) & \( [2]^{\oplus 3}  \oplus [-2]^{\oplus 7}\) &  \((2,14)\) & \(10\) & \(1\) \\
        
        \(32\) & \(11\) & \( \E_8(-1) \oplus \U^{\oplus2} \oplus [-2]^{\oplus 3} \) & \(\E_8(-1) \oplus [2]^{\oplus3}\) &  \((2,13)\) & \(3\) & \(1\) \\

        \(33\) & \(11\) & \( \E_8(-1) \oplus \D_4(-1)\oplus  [2]^{\oplus 2} \oplus[-2]\) & \(\U^{\oplus 3}\oplus [-2]^{\oplus 5}\) &  \((2,13)\) & \(5\) & \(1\) \\

        \(34\) & \(11\) & \( \E_8(-1)  \oplus  [2]^{\oplus 2} \oplus [-2]^{\oplus 5} \) & \(\U^{\oplus 2}\oplus [2]  \oplus [-2]^{\oplus 6}\) &  \((2,13)\) & \(7\) & \(1\) \\

        \(35\) & \(11\) & \( \U^{\oplus 2} \oplus\D_4(-1) \oplus [-2]^{\oplus 7} \) & \(\U\oplus [2]^{\oplus2}  \oplus [-2]^{\oplus 7}\) &  \((2,13)\) & \(9\) & \(1\) \\
        
        \(36\) & \(11\) & \( \U\oplus\D_4(-1)\oplus[2] \oplus [-2]^{\oplus 8} \) & \( [2]^{\oplus3}  \oplus [-2]^{\oplus 8}\) &  \((2,13)\) & \(11\) & \(1\) \\

        \(37\) & \(12\) & \( \E_8(-1)\oplus \U^{\oplus2}\oplus  [-2]^{\oplus2} \) & \(\E_8(-1) \oplus \U\oplus [2]^{\oplus2}\) &  \((2,12)\) & \(2\) & \(1\) \\

        \(38\) & \(12\) & \( \E_8(-1)\oplus \U\oplus [2]\oplus [-2]^{\oplus 3}\) & \(\E_8(-1) \oplus [2]^{\oplus3}\oplus [-2]\) &  \((2,12)\) & \(4\) & \(1\) \\

        \( 39\) & \(12\) & \(\E_8(-1) \oplus [2]^{\oplus2} \oplus  [-2]^{\oplus 4} \) & \(\U^{\oplus 3}\oplus [-2]^{\oplus 6}\) &  \((2,12)\) & \(6\) & \(1\) \\

         \( 40\) & \(12\) & \( \U^{\oplus2} \oplus \D_4(-1) \oplus [-2]^{\oplus 6} \) & \(\U^{\oplus 2}\oplus [2]  \oplus [-2]^{\oplus 7}\) &  \((2,12)\) & \(8\) & \(1\) \\

       \( 41\) & \(12\) & \( \U\oplus \D_4(-1) \oplus [2]\oplus [-2]^{\oplus 7}  \) & \(\U\oplus [2]^{\oplus2}  \oplus [-2]^{\oplus 8}\) &  \((2,12)\) & \(10\) & \(1\) \\
       
        \( 42\) & \(12\) & \( \U\oplus [2] \oplus  [-2]^{\oplus 11}  \) & \( [2]^{\oplus3}  \oplus [-2]^{\oplus 9}\) &  \((2,12)\) & \(12\) & \(1\) \\

        \( 43 \) & \(13\) & \( \E_8(-1)\oplus \U^{\oplus2}\oplus  [-2] \) & \(\E_8(-1) \oplus \U^{\oplus2}\oplus [2]\) &  \((2,11)\) & \(1\) & \(1\) \\

        \( 44\) & \(13\) & \( \E_8(-1)\oplus\U\oplus [2] \oplus [-2]^{\oplus2} \) & \(\E_8(-1) \oplus \U\oplus [2]^{\oplus2}\oplus [-2]\) &  \((2,11)\) & \(3\) & \(1\) \\

       \(45 \) & \(13\) & \( \E_8(-1)\oplus [2]^{\oplus 2}\oplus [-2]^{\oplus 3} \) & \(\E_8(-1) \oplus [2]^{\oplus3}\oplus [-2]^{\oplus 2}\) &  \((2,11)\) & \(5\) & \(1\) \\

        \(  46\) & \(13\) & \( \U^{\oplus2} \oplus\D_4(-1)  \oplus [-2]^{\oplus5}  \) & \(\U^{\oplus 3}\oplus [-2]^{\oplus 7}\) &  \((2,11)\) & \(7\) & \(1\) \\

        \(  47\) & \(13\) & \( \U\oplus \D_4(-1) \oplus  [2] \oplus [-2]^{\oplus6} \) & \(\U^{\oplus 2}\oplus [2]  \oplus [-2]^{\oplus 8}\) &  \((2,11)\) & \(9\) & \(1\) \\

        \(48 \) & \(13\) & \( \D_4(-1) \oplus  [2]^{\oplus 2} \oplus [-2]^{\oplus7} \) & \(\U\oplus [2]^{\oplus2}  \oplus [-2]^{\oplus 9}\) &  \((2,11)\) & \(11\) & \(1\) \\
        
         \(49 \) & \(13\) & \(  [2]^{\oplus 2} \oplus [-2]^{\oplus 11} \) & \([2]^{\oplus3}  \oplus [-2]^{\oplus 10}\) &  \((2,11)\) & \(13\) & \(1\) \\

         \( 50 \) & \(14\) & \( \E_8(-1)\oplus \U^{\oplus 2}  \) & \(\E_8(-1)\oplus \U^{\oplus 3}\) &  \((2,10)\) & \(0\) & \(0\) \\

           \( 51 \) & \(14\) & \( \E_8(-1)\oplus \U\oplus\U(2)  \) & \(\E_8(-1)\oplus \U^{\oplus 2}\oplus\U(2)\) &  \((2,10)\) & \(2\) & \(0\) \\

           \( 52 \) & \(14\) & \(  \E_8(-1)\oplus \U\oplus[2]\oplus [-2] \) & \(\E_8(-1)\oplus \U^{\oplus 2}\oplus[2]\oplus [-2]\) &  \((2,10)\) & \(2\) & \(1\) \\

           \( 53 \) & \(14\) & \( \E_8(-1)\oplus\U(2)^{\oplus 2}  \) & \(\E_8(-1)\oplus \U\oplus\U(2)^{\oplus 2}\) &  \((2,10)\) & \(4\) & \(0\) \\

           \( 54\) & \(14\) & \( \E_8(-1)\oplus [2]^{\oplus 2}\oplus [-2]^{\oplus 2}  \) & \(\E_8(-1)\oplus \U\oplus[2]^{\oplus 2}\oplus [-2]^{\oplus 2}\) &  \((2,10)\) & \(4\) & \(1\) \\

            \( 55 \) & \(14\) & \( \U\oplus \U(2)\oplus\D_4(-1)^{\oplus 2}  \) & \(\E_8(-1)\oplus \U(2)^{\oplus 3}\) &  \((2,10)\) & \(6\) & \(0\) \\
            
           \( 56 \) & \(14\) & \(  \U^{\oplus 2} \oplus \D_4(-1) \oplus [-2]^{\oplus 4}  \) & \(\E_8(-1)\oplus [2]^{\oplus 3}\oplus [-2]^{\oplus 3}\) &  \((2,10)\) & \(6\) & \(1\) \\

           \( 57 \) & \(14\) & \( \U(2)^{\oplus 2}\oplus \D_4(-1)^{\oplus 2}  \) & \(\D_4(-1)^{\oplus 2}\oplus\U\oplus \U(2)^{\oplus 2}\) &  \((2,10)\) & \(8\) & \(0\) \\

         \( 58 \) & \(14\) & \( \D_4(-1)^{\oplus 2}\oplus [2]^{\oplus2} \oplus  [-2]^{\oplus 2} \) & \(\U^{\oplus 3}\oplus [-2]^{\oplus 8}\) &  \((2,10)\) & \(8\) & \(1\) \\

         \( 59 \) & \(14\) & \(  \U\oplus \U(2) \oplus \E_8(-2)\) & \(\D_4(-1)^{\oplus 2}\oplus \U(2)^{\oplus 3}\) &  \((2,10)\) & \(10\) & \(0\) \\

        \( 60 \) & \(14\) & \(  \U\oplus[2]\oplus [-2]^{\oplus 9} \) & \(\D_4(-1)^{\oplus 2}\oplus [2]^{\oplus 3}\oplus [-2]^{\oplus 3}\) &  \((2,10)\) & \(10\) & \(1\) \\

         \( 61 \) & \(14\) & \(  \E_8(-2)\oplus \U(2)^{\oplus 2}\) & \( \E_8(-2)\oplus \U\oplus \U(2)^{\oplus 2}\) &  \((2,10)\) & \(12\) & \(0\) \\

          \(62  \) & \(14\) & \(  [2]^{\oplus 2}\oplus [-2]^{\oplus 10}\) & \( \U\oplus [2]^{\oplus 2}\oplus [-2]^{\oplus 10}\) &  \((2,10)\) & \(12\) & \(1\) \\

          \(  63\) & \(15\) & \( \E_8(-1) \oplus \U\oplus [2] \) & \(\E_8(-1)\oplus \U^{\oplus 3}\oplus [-2] \) &  \((2,9)\) & \(1\) & \(1\) \\

          \(64\) & \(15\) & \( \E_8(-1) \oplus [2]^{\oplus 2 }\oplus [-2] \) & \(\E_8(-1)\oplus \U^{\oplus 2}\oplus [2]\oplus [-2]^{\oplus 2} \) &  \((2,9)\) & \(3\) & \(1\) \\

          \( 65 \) & \(15\) & \( \U^{\oplus 2}\oplus \D_4(-1) \oplus [-2]^{\oplus 3} \) & \(\E_8(-1)\oplus \U\oplus [2]^{\oplus 2}\oplus [-2]^{\oplus 3} \) &  \((2,9)\) & \(5\) & \(1\) \\

          \(  66\) & \(15\) & \( \U\oplus \D_4(-1)  \oplus [2]\oplus [-2]^{\oplus 4} \) & \(\E_8(-1)\oplus  [2]^{\oplus 3}\oplus [-2]^{\oplus 4} \) &  \((2,9)\) & \(7\) & \(1\) \\

          \( 67\) & \(15\) & \( \D_4(-1) \oplus [2]^{\oplus 2 }\oplus [-2]^{\oplus 5} \) & \(\U^{\oplus 3}\oplus [-2]^{\oplus 9} \) &  \((2,9)\) & \(9\) & \(1\) \\

           \(68  \) & \(15\) & \(  [2]^{\oplus 2} \oplus [-2]^{\oplus 9 } \) & \(\U^{\oplus 2}\oplus [2]^{\oplus 2}\oplus [-2]^{\oplus 9} \) &  \((2,9)\) & \(11\) & \(1\) \\

 \( 69 \) & \(16\) & \(\E_8(-1)\oplus [2]^{\oplus 2} \) & \( \E_8(-1)\oplus \U^{\oplus 3}\oplus[-2]^{\oplus 2}\) &  \((2,8)\) & \( 2\) & \(1\) \\

              \( 70\) & \(16\) & \(\U^{\oplus 2}\oplus\D_4(-1)\oplus[-2]^{\oplus 2} \) & \( \E_8(-1)\oplus \D_4(-1)\oplus \U\oplus[2]^{\oplus 2}\) &  \((2,8)\) & \( 4\) & \(1\) \\

            \( 71 \) & \(16\) & \( \U^{\oplus2}\oplus [-2]^{\oplus 6}\) & \( \E_8(-1)\oplus \U\oplus [2]^{\oplus 2}\oplus[-2]^{\oplus 4}\) &  \((2,8)\) & \( 6\) & \(1\) \\

            \( 72 \) & \(16\) & \( \U\oplus[2]\oplus [-2]^{\oplus 7}\) & \( \E_8(-1)\oplus [2]^{\oplus 3}\oplus[-2]^{\oplus 5}\) &  \((2,8)\) & \( 8\) & \(1\) \\

             \( 73 \) & \(16\) & \( [2]^{\oplus 2}\oplus [-2]^{\oplus 8}\) & \( \U^{\oplus 3}\oplus [-2]^{\oplus 10}\) &  \((2,8)\) & \( 10\) & \(1\) \\

             \( 74 \) & \(17\) & \( \U^{\oplus2} \oplus \D_{4}(-1) \oplus [-2] \) & \(\E_8(-1) \oplus \U^{\oplus3} \oplus [-2]^{\oplus 3} \) &  \((2,7)\) & \( 3\) & \(1\) \\
             
              \( 75\) & \(17\) & \( \U^{\oplus2} \oplus [-2]^{\oplus 5} \) & \(\E_8(-1) \oplus \U^{\oplus2} \oplus [2] \oplus [-2]^{\oplus 4} \) &  \((2,7)\) & \( 5\) & \(1\) \\

              \( 76\) & \(17\) & \( \U\oplus [2] \oplus [-2]^{\oplus 6} \) & \(\E_8(-1) \oplus \U\oplus [2]^{\oplus 2} \oplus [-2]^{\oplus 5} \) &  \((2,7)\) & \( 7\) & \(1\) \\

              \( 77 \) & \(17\) & \( [2]^{\oplus2} \oplus [-2]^{\oplus 7} \) & \(\E_8(-1) \oplus [2]^{\oplus 3} \oplus [-2]^{\oplus 6} \) &  \((2,7)\) & \( 9\) & \(1\) \\
            
              \(  78\) & \(18\) & \( \U^{\oplus2} \oplus \D_{4}(-1) \) & \(\E_8(-1) \oplus   \U^{\oplus3}\oplus \D_4(-1) \) &  \((2,6)\) & \( 2\) & \(0\) \\

               \(  79\) & \(18\) & \( \U\oplus \U(2)\oplus \D_{4}(-1) \) & \(\E_8(-1) \oplus \U^{\oplus2}\oplus \U(2)\oplus \D_4(-1)  \) &  \((2,6)\) & \( 4\) & \(0\) \\

               \( 80 \) & \(18\) & \( \U\oplus \D_{4}(-1)\oplus [2]\oplus[-2] \) & \(\E_8(-1) \oplus \U^{\oplus2}\oplus \D_4(-1) \oplus [2]\oplus [-2] \) &  \((2,6)\) & \( 4\) & \(1\) \\

               \( 81 \) & \(18\) & \( \U(2)^{\oplus 2}\oplus \D_{4}(-1) \) & \(\E_8(-1) \oplus \U\oplus \U(2)^{\oplus 2}\oplus \D_4(-1)  \) &  \((2,6)\) & \( 6\) & \(0\) \\

                \( 82 \) & \(18\) & \( \U \oplus[2]\oplus [-2]^{\oplus5}\) & \(\E_8(-1) \oplus \U\oplus \D_4(-1) \oplus [2]^{\oplus 2}\oplus [-2]^{\oplus 2} \) &  \((2,6)\) & \( 6\) & \(1\) \\

                \( 83 \) & \(18\) & \( [2]^{\oplus 2}\oplus [-2]^{\oplus 6} \) & \(\E_8(-1) \oplus \D_4(-1) \oplus [2]^{\oplus 3}\oplus [-2]^{\oplus 3} \) &  \((2,6)\) & \( 8\) & \(1\) \\

                \( 84 \) & \(19\) & \( \U^{\oplus 2}\oplus [-2]^{\oplus 3} \) & \( \E_8(-1)^{\oplus 2}\oplus [2]^{\oplus 3} \) &  \((2,5)\) & \( 3\) & \(1\) \\

                \( 85 \) & \(19\) & \( \U\oplus[2]\oplus [-2]^{\oplus 4} \) & \( \E_8(-1)\oplus \U^{\oplus2 }\oplus\D_4(-1)\oplus [2]\oplus[-2]^{\oplus 2} \) &  \((2,5)\) & \( 5\) & \(1\) \\

                \( 86 \) & \(19\) & \( [2]^{\oplus 2}\oplus [-2]^{\oplus 5} \) & \( \E_8(-1)\oplus \U\oplus \D_4(-1)\oplus [2]^{\oplus 2}\oplus[-2]^{\oplus 3} \) &  \((2,5)\) & \( 7\) & \(1\) \\

                \( 87 \) & \(20\) & \( \U^{\oplus 2}\oplus [-2]^{\oplus 2} \) & \( \E_8(-1)^{\oplus 2}\oplus \U\oplus [2]^{\oplus 2} \) &  \((2,4)\) & \( 2\) & \(1\) \\

                \( 88 \) & \(20\) & \( \U\oplus [2]\oplus [-2]^{\oplus 3} \) & \( \E_8(-1)^{\oplus 2}\oplus  [2]^{\oplus 3}\oplus [-2] \) &  \((2,4)\) & \( 4\) & \(1\) \\

                \( 89 \) & \(20\) & \( [2]^{\oplus 2}\oplus [-2]^{\oplus 4} \) & \( \E_8(-1)\oplus \U^{\oplus 2}\oplus \D_4(-1)\oplus [2]\oplus [-2]^{\oplus 3}\) &  \((2,4)\) & \( 6\) & \(1\) \\

                \( 90 \) & \(21\) & \( \U^{\oplus 2}\oplus [-2] \) & \( \E_8(-1)^{\oplus 2}\oplus \U^{\oplus 2}\oplus [2] \) &  \((2,3)\) & \( 1\) & \(1\) \\

                \( 91 \) & \(21\) & \( \U\oplus [2]\oplus [-2]^{\oplus 2} \) & \( \E_8(-1)^{\oplus 2}\oplus \U\oplus [2]^{\oplus 2}\oplus [-2] \) &  \((2,3)\) & \( 3\) & \(1\) \\

              \(  92\) & \(21\) & \(  [2]^{\oplus 2}\oplus [-2]^{\oplus 3} \) & \( \E_8(-1)^{\oplus 2}\oplus  [2]^{\oplus 3}\oplus [-2]^{\oplus 2} \) &  \((2,3)\) & \( 5\) & \(1\) \\

              \( 93 \) & \(22\) & \(  \U^{\oplus 2} \) & \( \E_8(-1)^{\oplus 2}\oplus \U^{\oplus 3} \) &  \((2,2)\) & \( 0\) & \(0\) \\

              \(94  \) & \(22\) & \(  \U\oplus\U(2) \) & \( \E_8(-1)^{\oplus 2}\oplus \U^{\oplus 2}\oplus \U(2) \) &  \((2,2)\) & \( 2\) & \(0\) \\

              \( 95 \) & \(22\) & \(  \U\oplus[2]\oplus [-2] \) & \( \E_8(-1)^{\oplus 2}\oplus \U^{\oplus 2}\oplus [2]\oplus [-2] \) &  \((2,2)\) & \( 2\) & \(1\) \\

                \( 96\) & \(22\) & \(  \U(2)^{\oplus 2} \) & \( \E_8(-1)^{\oplus 2}\oplus \U\oplus \U(2)^{\oplus 2} \) &  \((2,2)\) & \( 4\) & \(0\) \\

                \( 97 \) & \(22\) & \(  [2]^{\oplus 2}\oplus [-2]^{\oplus 2} \) & \( \E_8(-1)^{\oplus 2}\oplus \U\oplus [2]^{\oplus 2}\oplus [-2]^{\oplus 2} \) &  \((2,2)\) & \( 4\) & \(1\) \\

               \( 98 \) & \(23\) & \(  \U\oplus[2]  \) & \( \E_8(-1)^{\oplus 2}\oplus \U^{\oplus 3} \oplus[-2]\) &  \((2,1)\) & \( 1\) & \(1\) \\

               \( 99 \) & \(23\) & \(  [2]^{\oplus 2} \oplus [-2] \) & \( \E_8(-1)^{\oplus 2}\oplus \U^{\oplus 2}\oplus [2]\oplus[-2]^{\oplus 2}\) &  \((2,1)\) & \( 3\) & \(1\) \\

               \( 100 \) & \(24\) & \(  [2]^{\oplus 2}  \) & \( \E_8(-1)^{\oplus 2}\oplus \U^{\oplus 3} \oplus[-2]^{\oplus 2}\) &  \((2,0)\) & \( 2\) & \(1\) \\

	\midrule
    	\end{longtable}}
\end{center}

\begin{center}
\tiny{	\begin{longtable}{lllllll}
		
		\caption{Pairs \((\bL^G,\bL_G)\) for \(G \subset \bO(\bL)\) of prime order \(p=2\) and \(\sgn(\bL_G)=\sgn(\bLambda_{G})\), i.e. trivial action on the discriminant group.}
	\label{tab:Lid} \\
		
		\toprule
	 No. & \(\rk(\bLambda^{{G}})\) &  \(\bL_G\) & \(\bL^G\) & \(\sgn(\bL_G)\) & \(a(\bL_G)\) & \(\delta(\bL_G)\)  \\

		\midrule
		\endfirsthead
		
		\multicolumn{7}{c}%
		{\tablename\ \thetable{}, follows from previous page} \\
		\midrule
	No. & \(\rk(\bLambda^{{G}})\) &  \(\bL_G\) & \(\bL^G\) & \(\sgn(\bL_G)\) & \(a\) & \(\delta\)  \\
		\midrule
		\endhead
	
		\multicolumn{7}{c}{Continues on next page} \\
		\endfoot
		
		\bottomrule
		\endlastfoot

		\(1\) & \(4\) & \(\E_8(-1)^{\oplus2} \oplus \U^{\oplus2} \oplus [-2]^{\oplus2}\)& \( [ 2 ]\oplus [-6]\)& \((2,20)\) & \(2\) & \(1\) \\
		\(2\) & \(5\) & \(\E_8(-1)^{\oplus2} \oplus \U^{\oplus2} \oplus [-2]\) & \(\A_2(-1)\oplus [ 2 ] \) &  \((2,19)\) & \(1\) & \(1\) \\
  
  	\(3\) & \(5\) & \(\E_8(-1)^{\oplus2} \oplus \U\oplus [2]\oplus [-2]^{\oplus 2} \) & \( [2] \oplus [-2]\oplus [-6] \) &  \((2,19)\) & \(3\) & \(1\) \\
   
		\(4\) & \(6\)  & \(\E_8(-1)^{\oplus2} \oplus \U^{\oplus2}\) & \(\U\oplus \A_2(-1)\)&  \((2,18)\) & \(0\) & \(0\) \\
  
 \(5\) & \(6\)  & \(\E_8(-1)^{\oplus2} \oplus \U\oplus\U(2)\) & \( \U(2)\oplus \A_2(-1)\)&  \((2,18)\) & \(2\) & \(0\) \\
  	\(6\) & \(6\)  & \(\E_8(-1)^{\oplus2} \oplus \U\oplus[2] \oplus[-2]\) & \(\A_2(-1) \oplus[2] \oplus[-2]\)  & \((2,18)\) & \(2\) & \(1\) \\

		\(7\) & \(6\)  & \(\E_8(-1) \oplus [ -2 ] ^{\oplus 2 } \oplus [ 2 ]^{\oplus2}\) & \([2 ]  \oplus [-2 ]^{\oplus2}\oplus [-6]\) &  \((2,18)\) & \(4\) & \(1\) \\
		
		\(8\) & \(7\) & \(\E_{8}(-1)^{\oplus2} \oplus \U\oplus [ 2 ]\) & \(\U\oplus\A_2(-1) \oplus [-2]\) &  \((2,17)\) & \(1\) & \(1\) \\

		\(9\) & \(7\)  & \(\E_{8}(-1)^{\oplus2} \oplus [ 2 ] ^{\oplus 2} \oplus [ -2 ] \) & \( \A_2(-1) \oplus [ 2 ] \oplus [ -2 ]^{\oplus 2}\) &  \((2,17)\) & \(3\) & \(1\) \\

		\(10\) & \(7\) & \(	\E_{8}(-1) \oplus \U^{\oplus 2} \oplus \D_4(-1) \oplus [-2]^{\oplus3} \) & \( [ 2 ] \oplus [-2]^{\oplus3}\oplus [-6]\) &  \((2,17)\) & \(5\) & \(1\) \\

		\(11\) & \(8\) & 
  \(\E_8(-1)^{\oplus 2}\oplus [2]^{\oplus 2} \)& \(\U\oplus \A_2(-1) \oplus [-2]^{\oplus 2}\)&  \((2,16)\) & \(2\) & \(1\) \\

		\(12\) & \(8\) & \(\E_{8}(-1) \oplus \U^{\oplus 2}\oplus \D_4(-1) \oplus [-2]^{\oplus2}\) & \(\A_2(-1) \oplus [ 2 ] \oplus [ -2 ]^{\oplus3}\) &  \((2,16)\) & \(4\) & \(1\) \\

		\( 13\) & \(8\) & \(\E_{8}(-1) \oplus \D_4(-1)^{\oplus 2}\oplus [2]^{\oplus 2} \) & \( [2] \oplus [-2]^{\oplus 4}\oplus[-6] \) &  \((2,16)\) & \(6\) & \(1\) \\

		\(14\) & \(9\) & \(\E_{8}(-1) \oplus \U^{\oplus 2} \oplus \D_4(-1)\oplus [-2] \)& \(\U\oplus\A_2(-1) \oplus [-2]^{\oplus 3} \) &  \((2,15)\) & \(3\) & \(1\) \\

		\(15\) & \(9\) & \(\E_{8}(-1) \oplus \U\oplus \D_4(-1)\oplus [2 ] \oplus [ -2 ]^{\oplus2}\) & \(\A_2(-1) \oplus [-2]^{\oplus 4} \oplus [2]\) &  \((2,15)\) & \(5\) & \(1\) \\
		
	    \(16\) & \(9\) & \(\E_{8}(-1) \oplus \D_4(-1)\oplus [2 ]^{\oplus 2} \oplus [ -2 ]^{\oplus3}\) & \( [ 2 ]\oplus [-2]^{\oplus 5}\oplus[-6]\) &  \((2,15)\) & \(7\) & \(1\) \\

\(17\) & \(10\) & \( \E_8(-1)\oplus \U^{\oplus 2}\oplus \D_4(-1) \) & \( \U \oplus \A_2(-1)\oplus \D_4(-1)\) &  \((2,14)\) & \(2\) & \(0\) \\

          \(17\) & \(10\) & \( \E_8(-1)\oplus \U\oplus \U(2)\oplus \D_4(-1) \) & \( \U(2) \oplus \A_2(-1)\oplus \D_4(-1)\) &  \((2,14)\) & \(4\) & \(0\) \\

        \(18\) & \(10\) & \(\E_8(-1)\oplus \U\oplus\D_4(-1)\oplus [2]\oplus [-2] \) & \(\U\oplus\A_2(-1) \oplus [-2]^{\oplus 4}\) &  \((2,14)\) & \(4\) & \(1\) \\

      \(19\) & \(10\) & \( \U^{\oplus 2} \oplus \D_4(-1)^{\oplus 3} \) & \(\U\oplus\E_6(-2)\) &  \((2,14)\) & \(6\) & \(0\) \\

       \(20\) & \(10\) & \(\E_8(-1) \oplus \D_4(-1)\oplus  [2]^{\oplus 2} \oplus [-2]^{\oplus 2} \) & \(\A_2(-1)\oplus [2]  \oplus [-2]^{\oplus 5}\) &  \((2,14)\) & \(6\) & \(1\) \\

     \(21\) & \(10\) & \(\U\oplus \U(2)\oplus \D_4(-1)^{\oplus 3}  \) & \( \U(2)\oplus \E_6(-2)\) &  \((2,14)\) & \(8\) & \(0\) \\

        \(21\) & \(10\) & \(\U^{\oplus 2}\oplus \D_4(-1)^{\oplus 2} \oplus [-2]^{\oplus 4} \) & \( [2] \oplus [-2]^{\oplus 6}\oplus [-6]\) &  \((2,14)\) & \(8\) & \(1\) \\
        
        \(22\) & \(11\) & \( \E_8(-1) \oplus \U^{\oplus2} \oplus [-2]^{\oplus 3} \) & \(\D_6(-1)\oplus\A_2(-1) \oplus [2]\) &  \((2,13)\) & \(3\) & \(1\) \\

        \(23\) & \(11\) & \( \E_8(-1) \oplus \D_4(-1)\oplus  [2]^{\oplus 2} \oplus[-2]\) & \(\U\oplus\A_2(-1)\oplus [-2]^{\oplus 5}\) &  \((2,13)\) & \(5\) & \(1\) \\

        \(24\) & \(11\) & \( \E_8(-1)  \oplus  [2]^{\oplus 2} \oplus [-2]^{\oplus 5} \) & \(\A_2(-1)\oplus [2]  \oplus [-2]^{\oplus 6}\) &  \((2,13)\) & \(7\) & \(1\) \\

        \(25\) & \(11\) & \( \U^{\oplus 2} \oplus\D_4(-1) \oplus [-2]^{\oplus 7} \) & \( [2]  \oplus [-2]^{\oplus 7}\oplus [-6]\) &  \((2,13)\) & \(9\) & \(1\) \\

        \( 26\) & \(12\) & \( \E_8(-1)\oplus \U^{\oplus2}\oplus  [-2]^{\oplus2} \) & \(\E_8(-1) \oplus  [2]\oplus[-6]\) &  \((2,12)\) & \(2\) & \(1\) \\

        \(27\) & \(12\) & \( \E_8(-1)\oplus \U\oplus [2]\oplus [-2]^{\oplus 3}\) & \(\D_6(-1)\oplus\A_2(-1) \oplus [2]\oplus [-2]\) &  \((2,12)\) & \(4\) & \(1\) \\

        \( 28\) & \(12\) & \(\E_8(-1) \oplus [2]^{\oplus2} \oplus  [-2]^{\oplus 4} \) & \(\U\oplus\A_2(-1)\oplus [-2]^{\oplus 6}\) &  \((2,12)\) & \(6\) & \(1\) \\

         \( 29\) & \(12\) & \( \U^{\oplus2} \oplus \D_4(-1) \oplus [-2]^{\oplus 6} \) & \(\A_2(-1)\oplus [2]  \oplus [-2]^{\oplus 7}\) &  \((2,12)\) & \(8\) & \(1\) \\

       \( 30\) & \(12\) & \( \U\oplus \D_4(-1) \oplus [2]\oplus [-2]^{\oplus 7}  \) & \( [2]  \oplus [-2]^{\oplus 8}\oplus[-6]\) &  \((2,12)\) & \(10\) & \(1\) \\

        \( 31 \) & \(13\) & \( \E_8(-1)\oplus \U^{\oplus2}\oplus  [-2] \) & \(\E_8(-1) \oplus \A_2(-1)\oplus [2]\) &  \((2,11)\) & \(1\) & \(1\) \\

        \( 32\) & \(13\) & \( \E_8(-1)\oplus [2] \oplus [-2]^{\oplus2} \) & \(\E_8(-1) \oplus [2]\oplus [-2]\oplus[-6]\) &  \((2,11)\) & \(3\) & \(1\) \\

       \(33 \) & \(13\) & \( \E_8(-1)\oplus [2]^{\oplus 2}\oplus [-2]^{\oplus 3} \) & \(\D_6(-1)\oplus\A_2(-1) \oplus [2]\oplus [-2]^{\oplus 2}\) &  \((2,11)\) & \(5\) & \(1\) \\

        \(  34\) & \(13\) & \( \U^{\oplus2} \oplus\D_4(-1)  \oplus [-2]^{\oplus5}  \) & \(\U\oplus\A_2(-1)\oplus [-2]^{\oplus 7}\) &  \((2,11)\) & \(7\) & \(1\) \\

        \(  35\) & \(13\) & \( \U\oplus \D_4(-1) \oplus  [2] \oplus [-2]^{\oplus6} \) & \(\U \oplus [-2]^{\oplus 8}\oplus [-6]\) &  \((2,11)\) & \(9\) & \(1\) \\

        \(36 \) & \(13\) & \( \D_4(-1) \oplus  [2]^{\oplus 2} \oplus [-2]^{\oplus7} \) & \([2]  \oplus [-2]^{\oplus 9}\oplus [-6]\) &  \((2,11)\) & \(11\) & \(1\) \\

         \( 37 \) & \(14\) & \( \E_8(-1)\oplus \U^{\oplus 2}  \) & \(\E_8(-1)\oplus \U\oplus\A_2(-1)\) &  \((2,10)\) & \(0\) & \(0\) \\

           \( 38 \) & \(14\) & \( \E_8(-1)\oplus \U\oplus\U(2)  \) & \(\E_8(-1)\oplus \U(2)\oplus \A_2(-1)\) &  \((2,10)\) & \(2\) & \(0\) \\

           \( 39 \) & \(14\) & \(  \E_8(-1)\oplus \U\oplus[2]\oplus [-2] \) & \(\E_8(-1)\oplus \A_2(-1)\oplus[2]\oplus [-2]\) &  \((2,10)\) & \(2\) & \(1\) \\

           \( 40 \) & \(14\) & \( \E_8(-1)\oplus \U(2)^{\oplus 2}  \) & \( \U\oplus \D_4(-1)^{\oplus 2} \oplus \A_2(-1)\) &  \((2,10)\) & \(4\) & \(0\) \\

           \( 41 \) & \(14\) & \( \E_8(-1)\oplus [2]^{\oplus 2}\oplus [-2]^{\oplus 2}  \) & \(\E_8(-1)\oplus [2]\oplus [-2]^{\oplus 2}\oplus[-6]\) &  \((2,10)\) & \(4\) & \(1\) \\

 \( 42 \) & \(14\) & \( \U\oplus \U(2) \oplus \D_4(-1)^{\oplus 2}  \) & \(\U(2)\oplus\D_4(-1)^{\oplus 2}\oplus \A_2(-1)\) &  \((2,10)\) & \(6\) & \(0\) \\

           \( 42 \) & \(14\) & \( \U^{\oplus 2} \oplus \D_4(-1) \oplus [-2]^{\oplus 4}  \) & \(\D_6(-1)\oplus\A_2(-1)\oplus [2]\oplus [-2]^{\oplus 3}\) &  \((2,10)\) & \(6\) & \(1\) \\

           \( 43 \) & \(14\) & \( \U(2)^{\oplus 2}\oplus \D_4(-1)^{\oplus 2}  \) & \(\D_4(-1)\oplus\U\oplus \E_6(-2)\) &  \((2,10)\) & \(8\) & \(0\) \\

         \( 44 \) & \(14\) & \( \D_4(-1)^{\oplus 2}\oplus [2]^{\oplus2} \oplus  [-2]^{\oplus 2} \) & \(\U\oplus\A_2(-1)\oplus [-2]^{\oplus 8}\) &  \((2,10)\) & \(8\) & \(1\) \\

         \( 45 \) & \(14\) & \(  \U\oplus \U(2) \oplus \E_8(-2) \) & \(\D_4(-1)\oplus \U(2)\oplus\E_6(-2)\) &  \((2,10)\) & \(10\) & \(0\) \\
         
\( 46 \) & \(14\) & \(  \U\oplus[2]\oplus [-2]^{\oplus 9} \) & \([2]\oplus [-2] \oplus \E_6(-2) \oplus \D_4(-1) = \U\oplus \M \) &  \((2,10)\) & \(10\) & \(1\) \\

          \(47  \) & \(14\) & \(  \U(2)^{\oplus 2}\oplus [-2]^{\oplus 8}\) & \( \U(2) \oplus [-2]^{\oplus 9}\oplus [-6]\) &  \((2,10)\) & \(12\) & \(1\) \\

          \(  48\) & \(15\) & \( \E_8(-1) \oplus \U\oplus [2] \) & \(\E_8(-1)\oplus \U\oplus\A_2(-1)\oplus [-2] \) &  \((2,9)\) & \(1\) & \(1\) \\

          \(49  \) & \(15\) & \( \E_8(-1) \oplus [2]^{\oplus 2 }\oplus [-2] \) & \(\E_8(-1)\oplus \A_2(-1)\oplus [2]\oplus [-2]^{\oplus 2} \) &  \((2,9)\) & \(3\) & \(1\) \\

          \( 50 \) & \(15\) & \( \U^{\oplus 2}\oplus \D_4(-1) \oplus [-2]^{\oplus 3} \) & \(\E_8(-1)\oplus [2]\oplus [-2]^{\oplus 3} \oplus[-6]\) &  \((2,9)\) & \(5\) & \(1\) \\

          \(  51\) & \(15\) & \( \U\oplus \D_4(-1)  \oplus [2]\oplus [-2]^{\oplus 4} \) & \(\D_6(-1)\oplus\A_2(-1)\oplus  [2]\oplus [-2]^{\oplus 4} \) &  \((2,9)\) & \(7\) & \(1\) \\

          \( 52 \) & \(15\) & \( \D_4(-1) \oplus [2]^{\oplus 2 }\oplus [-2]^{\oplus 5} \) & \(\U\oplus\A_2(-1)\oplus [-2]^{\oplus 9} \) &  \((2,9)\) & \(9\) & \(1\) \\

           \(53  \) & \(15\) & \(  [2]^{\oplus 2} \oplus [-2]^{\oplus 9 } \) & \(\A_2(-1)\oplus [2]^{\oplus 2}\oplus [-2]^{\oplus 9} \) &  \((2,9)\) & \(11\) & \(1\) \\

 \( 54 \) & \(16\) & \( \E_8(-1) \oplus  [2]^{\oplus 2}  \) & \( \E_8(-1)\oplus \U\oplus \A_2(-1)\oplus[-2]^{\oplus 2}\) &  \((2,8)\) & \( 2\) & \(1\) \\

              \( 54 \) & \(16\) & \( \U^{\oplus2} \oplus \D_4(-1) \oplus [-2]^{\oplus 2}  \) & \( \E_8(-1)\oplus \D_4(-1)\oplus [2]\oplus[-6]\) &  \((2,8)\) & \( 4\) & \(1\) \\

            \( 55 \) & \(16\) & \( \U^{\oplus2}\oplus [-2]^{\oplus 6}\) & \( \D_6(-1)\oplus\A_2(-1)\oplus \U\oplus [-2]^{\oplus 4}\) &  \((2,8)\) & \( 6\) & \(1\) \\

            \( 56 \) & \(16\) & \( \U\oplus[2]\oplus [-2]^{\oplus 7}\) & \( \D_6(-1)\oplus\A_2(-1)\oplus [2]\oplus[-2]^{\oplus 5}\) &  \((2,8)\) & \( 8\) & \(1\) \\

             \( 57 \) & \(16\) & \( [2]^{\oplus 2}\oplus [-2]^{\oplus 8}\) & \( \U\oplus\A_2(-1)\oplus [-2]^{\oplus 10}\) &  \((2,8)\) & \( 10\) & \(1\) \\

             \( 58 \) & \(17\) & \( \U^{\oplus2} \oplus \D_{4}(-1) \oplus [-2] \) & \(\E_8(-1) \oplus \U\oplus\A_2(-1) \oplus [-2]^{\oplus 3} \) &  \((2,7)\) & \( 3\) & \(1\) \\
             
              \( 59 \) & \(17\) & \( \U^{\oplus2} \oplus [-2]^{\oplus 5} \) & \(\E_8(-1) \oplus \A_2(-1) \oplus [2] \oplus [-2]^{\oplus 4} \) &  \((2,7)\) & \( 5\) & \(1\) \\

              \( 60 \) & \(17\) & \( \U\oplus [2]^{\oplus 2} \oplus [-2]^{\oplus 5} \) & \(\E_8(-1) \oplus  [2]\oplus [-2]^{\oplus 5} \oplus[-6]\) &  \((2,7)\) & \( 7\) & \(1\) \\

              \(61  \) & \(17\) & \( [2]^{\oplus2} \oplus [-2]^{\oplus 7} \) & \(\D_6(-1)\oplus\A_2(-1) \oplus [2]\oplus [-2]^{\oplus 6} \) &  \((2,7)\) & \( 9\) & \(1\) \\
            
              \(62  \) & \(18\) & \( \U^{\oplus2} \oplus \D_{4}(-1) \) & \(\E_8(-1) \oplus   \U\oplus\A_2(-1)\oplus \D_4(-1) \) &  \((2,6)\) & \( 2\) & \(0\) \\

               \(  63\) & \(18\) & \( \U\oplus \U(2)\oplus \D_{4}(-1) \) & \(\E_8(-1) \oplus \A_2(-1)\oplus \U(2)\oplus \D_4(-1)  \) &  \((2,6)\) & \( 4\) & \(0\) \\

               \( 64 \) & \(18\) & \( \U\oplus \D_{4}(-1)\oplus [2]\oplus[-2] \) & \(\E_8(-1) \oplus \A_2(-1)\oplus \D_4(-1) \oplus [2]\oplus [-2] \) &  \((2,6)\) & \( 4\) & \(1\) \\

               \( 65 \) & \(18\) & \( \U(2)^{\oplus 2}\oplus \D_{4}(-1) \) & \(\E_8(-1) \oplus \U\oplus \E_6(-2)  \) &  \((2,6)\) & \( 6\) & \(0\) \\

                \( 66 \) & \(18\) & \( \D_{4}(-1) \oplus [2]^{\oplus2}\oplus [-2]^{\oplus2}\) & \(\E_8(-1) \oplus  \D_4(-1) \oplus [2]\oplus [-2]^{\oplus 2}\oplus [-6] \) &  \((2,6)\) & \( 6\) & \(1\) \\

                \(  67\) & \(18\) & \( [2]^{\oplus 2}\oplus [-2]^{\oplus 6} \) & \(\D_4(-1)  \oplus \D_6(-1) \oplus [2] \oplus [-2]^{\oplus3}\)&  \((2,6)\) & \( 8\) & \(1\) \\

                \(68  \) & \(19\) & \( \U^{\oplus 2}\oplus [-2]^{\oplus 3} \) & \( \E_8(-1)\oplus \D_6(-1)\oplus\A_2(-1)\oplus [2]\) &  \((2,5)\) & \( 3\) & \(1\) \\

                \( 69 \) & \(19\) & \( \U\oplus[2]\oplus [-2]^{\oplus 4} \) & \( \E_8(-1)\oplus \A_2(-1)\oplus\D_4(-1)\oplus [2]\oplus[-2]^{\oplus 2} \) &  \((2,5)\) & \( 5\) & \(1\) \\

                \( 70 \) & \(19\) & \( [2]^{\oplus 2}\oplus [-2]^{\oplus 5} \) & \( \E_8(-1)\oplus  \D_4(-1)\oplus [2]\oplus[-2]^{\oplus 3} \oplus [-6]\) &  \((2,5)\) & \( 7\) & \(1\) \\

                \( 71 \) & \(20\) & \( \U^{\oplus 2}\oplus [-2]^{\oplus 2} \) & \( \E_8(-1)^{\oplus 2}\oplus [2] \oplus [-6]\) &  \((2,4)\) & \( 2\) & \(1\) \\

                \( 72 \) & \(20\) & \( \U\oplus [2]\oplus [-2]^{\oplus 3} \) & \( \E_8(-1)\oplus\D_6(-1)\oplus\A_2(-1) \oplus  [2]\oplus [-2] \) &  \((2,4)\) & \( 4\) & \(1\) \\

                \( 73 \) & \(20\) & \( [2]^{\oplus 2}\oplus [-2]^{\oplus 4} \) & \( \E_8(-1)\oplus \A_2(-1)\oplus \D_4(-1)\oplus [2]\oplus [-2]^{\oplus 3}\) &  \((2,4)\) & \( 6\) & \(1\) \\

                \(  74\) & \(21\) & \( \U^{\oplus 2}\oplus [-2] \) & \( \E_8(-1)^{\oplus 2}\oplus \A_2(-1)\oplus [2] \) &  \((2,3)\) & \( 1\) & \(1\) \\

                \(  75\) & \(21\) & \( \U\oplus [2]\oplus [-2]^{\oplus 2} \) & \( \E_8(-1)^{\oplus 2}\oplus [2]\oplus [-2]\oplus[-6] \) &  \((2,3)\) & \( 3\) & \(1\) \\

              \(76  \) & \(21\) & \(  [2]^{\oplus 2}\oplus [-2]^{\oplus 3} \) & \( \E_8(-1)\oplus\D_6(-1)\oplus\A_2(-1) \oplus  [2]\oplus [-2]^{\oplus 2}\) &  \((2,3)\) & \( 5\) & \(1\) \\

              \(  77\) & \(22\) & \(  \U^{\oplus 2} \) & \( \E_8(-1)^{\oplus 2}\oplus \U\oplus\A_2(-1) \) &  \((2,2)\) & \( 0\) & \(0\) \\

              \( 78 \) & \(22\) & \(  \U\oplus\U(2) \) & \( \E_8(-1)^{\oplus 2}\oplus \A_2(-1)\oplus \U(2) \) &  \((2,2)\) & \( 2\) & \(0\) \\

              \(  79\) & \(22\) & \(  \U\oplus[2]\oplus [-2] \) & \( \E_8(-1)^{\oplus 2}\oplus \A_2(-1)\oplus [2]\oplus [-2] \) &  \((2,2)\) & \( 2\) & \(1\) \\

                \( 80 \) & \(22\) & \(  \U(2)^{\oplus 2} \) & \( \E_8(-1)\oplus \U\oplus \D_4(-1)^{\oplus2} \oplus \A_2(-1)\) &  \((2,2)\) & \( 4\) & \(0\) \\

                \( 81 \) & \(22\) & \(  [2]^{\oplus 2}\oplus [-2]^{\oplus 2} \) & \( \E_8(-1)^{\oplus 2}\oplus [2]\oplus [-2]^{\oplus 2}\oplus[-6] \) &  \((2,2)\) & \( 4\) & \(1\) \\

               \( 82 \) & \(23\) & \(  \U\oplus[2]  \) & \( \E_8(-1)^{\oplus 2}\oplus \U\oplus\A_2(-1)\oplus[-2]\) &  \((2,1)\) & \( 1\) & \(1\) \\

               \( 83 \) & \(23\) & \(  [2]^{\oplus 2} \oplus [-2] \) & \( \E_8(-1)^{\oplus 2}\oplus \A_2(-1)\oplus [2]\oplus[-2]^{\oplus 2}\) &  \((2,1)\) & \( 3\) & \(1\) \\

               \(  84\) & \(24\) & \(  [2]^{\oplus 2}  \) & \( \E_8(-1)^{\oplus 2}\oplus \U\oplus\A_2(-1)  \oplus[-2]^{\oplus 2}\) &  \((2,0)\) & \( 2\) & \(1\) \\

		\midrule
    	\end{longtable}}
\end{center}

 \begin{center}
\tiny{\begin{longtable}{lllllll}
		
		\caption{Pairs \((\bL^G,\bL_G)\) for \(G \subset \bO(\bL)\) of prime order \(p=2\) and \(\sgn(\bL_G)=\sgn(\bLambda_{G})-(1,0)\), i.e. non-trivial action on the discriminant group.}
	\label{tab:Lnotid} \\
		
		\toprule
	 No. & \(\rk(\bL_{G})\) &  \(\bL_G=(\bL^G)^{\perp \bL}\) & \(\bL^G=[2]^{\perp \bLambda^{G}}\) & \(\sgn(\bL^G)\) & \(a(\bL^G)\) & \(\delta(\bL^G)\) \\

		\midrule
		\endfirsthead
		
		\multicolumn{7}{c}%
		{\tablename\ \thetable{}, follows from previous page} \\
		\midrule
	No. & \(\rk(\bL_{G})\) &  \(\bL_G=(\bL^G)^{\perp \bL}\) & \(\bL^G=[2]^{\perp \bLambda^{G}}\) & \(\sgn(\bL^G)\) & \(a(\bL^G)\) & \(\delta(\bL^G)\)  \\
		\midrule
		\endhead
	
		\multicolumn{7}{c}{Continues on next page} \\
		\endfoot
		
		\bottomrule
		\endlastfoot

         \(1\) & \(3\) & \([2]^{\oplus2}\oplus[-6]\)& \(\E_8(-1)^{\oplus2} \oplus\U\oplus [-2]^{\oplus 3}\)& \((1,20)\) & \(3\) & \(1\) \\

  	\(2\) & \(4\) & \([2]^{\oplus2}\oplus\A_2(-1) \) & \(\E_8(-1)^{\oplus2} \oplus \U\oplus [-2]^{\oplus 2} \) &  \((1,19)\) & \(2\) & \(1\) \\

   \(3\) & \(4\) & \([2]^{\oplus2}\oplus[-2]\oplus[-6] \) & \(\E_8(-1)^{\oplus2} \oplus [2]\oplus [-2]^{\oplus 3} \) &  \((1,19)\) & \(3\) & \(1\) \\

		\(4\) & \(5\)  & \(\U\oplus[2]\oplus\A_2(-1)\) & \(\E_8(-1)^{\oplus2} \oplus \U\oplus[-2]\)&  \((1,18)\) & \(1\) & \(1\) \\

    \(6 \) & \(5\)  & \([2]^{\oplus2}\oplus[-2]\oplus\A_2(-1)\) & \(\E_8(-1)^{\oplus2} \oplus[2] \oplus[-2]^{\oplus 2}\)  & \((1,18)\) & \(3\) & \(1\) \\

        \(7\) & \(5\)  & \([2]^{\oplus2}\oplus[-2]^{\oplus 2}\oplus[-6]\) & \(\E_8(-1) \oplus \U\oplus\D_4(-1)^{\oplus2}\oplus[-2]\) &  \((1,18)\) & \(5\) & \(1\) \\

		\(8\) & \(6\) & \(\U^{\oplus2}\oplus\A_2(-1)\) & \(\E_{8}(-1)^{\oplus2} \oplus \U\) &  \((1,17)\) & \(0\) & \(0\) \\

\(9\) & \(6\)  & \(\U\oplus\U(2)\oplus\A_2(-1) \) & \(\E_{8}(-1)^{\oplus2} \oplus\U(2)\) &  \((1,17)\) & \(2\) & \(0\) \\

  	\( 11\) & \(6\) & \(\U\oplus[2]\oplus[-2]\oplus\A_2(-1)\) & \(\E_{8}(-1)^{\oplus2} \oplus [ 2 ]\oplus[-2]\) &  \((1,17)\) & \(2\) & \(1\) \\

  \(12\) & \(6\) & \(	\U(2)^{\oplus2}\oplus\A_2(-1) \) & \(\E_{8}(-1) \oplus \U\oplus \D_4(-1)^{\oplus2}\) &  \((1,17)\) & \(4\) & \(0\) \\

  	\(13 \) & \(6\)  & \([2]^{\oplus2}\oplus[-2]^{\oplus2}\oplus\A_2(-1) \) & \(\E_{8}(-1)\oplus\U\oplus\D_6(-1) \oplus [ -2 ]^{\oplus 2}\) &  \((1,17)\) & \(4\) & \(1\) \\

       \(14\) & \(6\) & \(	[2]^{\oplus2}\oplus [-2]^{\oplus3}\oplus[-6] \) & \(\E_{8}(-1) \oplus \U\oplus \D_4(-1) \oplus [-2]^{\oplus4}\) &  \((1,17)\) & \(6\) & \(1\) \\

		\(15\) & \(7\) & 
  \(\U^{\oplus2}\oplus[-2]\oplus\A_2(-1) \)& \(\E_8(-1)^{\oplus 2}\oplus [2]\)&  \((1,16)\) & \(1\) & \(1\) \\

         \(16 \) & \(7\) & \(\U\oplus[2]\oplus[-2]^{\oplus2}\oplus\A_2(-1)\) & \(\E_{8}(-1) \oplus \U\oplus \D_6(-1) \oplus [-2]\) &  \((1,16)\) & \(3\) & \(1\) \\

		\( 17\) & \(7\) & \(\D_4(-1)\oplus[2]^{\oplus2}\oplus[-6]\) & \(\E_{8}(-1) \oplus \U\oplus \D_4(-1) \oplus [-2]^{\oplus3}\) &  \((1,16)\) & \(5\) & \(1\) \\

        	\( 18\) & \(7\) & \([2]^{\oplus2}\oplus [-2]^{\oplus4}\oplus[-6] \) & \( \E_{8}(-1) \oplus \U\oplus [-2]^{\oplus 7} \) &  \((1,16)\) & \(7\) & \(1\) \\

		\(19\) & \(8\) & \(\U^{\oplus2}\oplus [-2]^{\oplus2}\oplus\A_2(-1) \)& \(\E_{8}(-1) \oplus \U\oplus \D_6(-1) \) &  \((1,15)\) & \(2\) & \(0\) \\

		\(20\) & \(8\) & \(\D_4(-1)\oplus [2]^{\oplus2}\oplus\A_2(-1)\) & \(\E_{8}(-1) \oplus \U\oplus \D_4(-1) \oplus [ -2 ]^{\oplus2}\) &  \((1,15)\) & \(4\) & \(1\) \\

	    \(21\) & \(8\) & \(\D_4(-1)\oplus [2]^{\oplus2}\oplus[-2]\oplus[-6]\) & \(\E_{8}(-1) \oplus \D_4(-1)\oplus [2 ] \oplus [ -2 ]^{\oplus3}\) &  \((1,15)\) & \(6\) & \(1\) \\

         \( 22 \) & \(8\) & \([2]^{\oplus 2}\oplus [-2]^{\oplus5}\oplus[-6]\) & \(\E_{8}(-1) \oplus [2 ] \oplus [ -2 ]^{\oplus7}\) &  \((1,15)\) & \(8\) & \(1\) \\

      \(23\) & \(9\) & \( \U\oplus\D_4(-1)\oplus[2]\oplus\A_2(-1)\) & \(\E_8(-1)\oplus \U\oplus\D_4(-1)\oplus [-2]\) &  \((1,14)\) & \(3\) & \(1\) \\

       \(24\) & \(9\) & \(\D_4(-1)\oplus[2]^{\oplus 2}\oplus[-2]\oplus\A_2(-1) \) & \(\E_8(-1) \oplus \D_4(-1)\oplus  [2] \oplus [-2]^{\oplus 2}\) &  \((1,14)\) & \(5\) & \(1\) \\

       \( 25\) & \(9\) & \([2]^{\oplus2}\oplus[-2]^{\oplus 5}\oplus\A_2(-1)\) & \(\E_8(-1) \oplus [2] \oplus [-2]^{\oplus 6}\) &  \((1,14)\) & \(7\) & \(1\) \\

        \(26\) & \(9\) & \([2]^{\oplus 2}\oplus [-2]^{\oplus6}\oplus[-6] \) & \(\U\oplus \D_4(-1)^{\oplus 2} \oplus [-2]^{\oplus 5} \) &  \((1,14)\) & \(9\) & \(1\) \\

\(27\) & \(10\) & \( \U^{\oplus 2}\oplus \D_4(-1)\oplus\A_2(-1) \) & \(\E_8(-1) \oplus \D_4(-1)\oplus  \U\) &  \((1,13)\) & \(2\) & \(0\) \\

  \(27\) & \(10\) & \( \U\oplus\U(2) \oplus \D_4(-1)\oplus\A_2(-1) \) & \(\E_8(-1) \oplus \D_4(-1)\oplus  \U(2)\) &  \((1,13)\) & \(4\) & \(0\) \\

        \(28\) & \(10\) & \( \U\oplus[2]\oplus[-2] \oplus \D_4(-1)\oplus\A_2(-1) \) & \(\E_8(-1) \oplus \D_4(-1)\oplus  [2]\oplus[-2]\) &  \((1,13)\) & \(4\) & \(1\) \\

\(29\) & \(10\) & \( \U(2)^{\oplus2}\oplus\D_4(-1)\oplus\A_2(-1) \) & \(\U\oplus\D_4(-1)^{\oplus 3}\) &  \((1,13)\) & \(6\) & \(0\) \\
     
        \(30\) & \(10\) & \( \D_6(-1) \oplus [2]^{\oplus3}\oplus[-6] \) & \( \E_8(-1)  \oplus  [2] \oplus [-2]^{\oplus 5} \) &  \((1,13)\) & \(6\) & \(1\) \\

\(31\) & \(10\) & \( \U\oplus \E_6(-2) \) & \(\U(2)\oplus\D_4(-1)^{\oplus 3} \) &  \((1,13)\) & \(8\) & \(0\) \\
        
        \(31\) & \(10\) & \( [2]^{\oplus2}\oplus [-2]^{\oplus 6}\oplus\A_2(-1) \) & \(\U\oplus\D_4(-1)^{\oplus 2} \oplus [-2]^{\oplus 4}\) &  \((1,13)\) & \(8\) & \(1\) \\

           \(32\) & \(10\) & \( \U(2)^{\oplus 2}\oplus \E_6(-2) \) & \(\U\oplus \E_8(-2)\oplus\D_4(-1)\) &  \((1,13)\) & \(10\) & \(0\) \\

        \(33\) & \(10\) & \( [2]^{\oplus2}\oplus [-2]^{\oplus 7}\oplus[-6] \) & \(\U\oplus\D_4(-1) \oplus [-2]^{\oplus 8}\) &  \((1,13)\) & \(10\) & \(1\) \\

        \(34\) & \(11\) & \( \E_8(-1)\oplus [2]^{\oplus2}\oplus  [-6]\) & \(\E_8(-1)\oplus \U\oplus [-2]^{\oplus 3}\) &  \((1,12)\) & \(3\) & \(1\) \\

        \(35\) & \(11\) & \(\U{\oplus2}\oplus [-2]^{\oplus6}\oplus\A_2(-1) \) & \(\E_8(-1) \oplus [2] \oplus  [-2]^{\oplus 4}\) &  \((1,12)\) & \(5\) & \(1\) \\

         \( 36\) & \(11\) & \( \D_4(-1)^{\oplus2}\oplus [2]^{\oplus2}\oplus  [-6] \) & \(\U\oplus \D_4(-1)^{\oplus2} \oplus [-2]^{\oplus 3} \) &  \((1,12)\) & \(7\) & \(1\) \\

       \( 37\) & \(12\) & \( \D_4(-1)\oplus [2]^{\oplus2}\oplus[-2]^{\oplus4}\oplus  [-6]  \) & \( \U\oplus \D_4(-1) \oplus [-2]^{\oplus 7}\) &  \((1,12)\) & \(9\) & \(1\) \\

        \( 38\) & \(11\) & \(  [2]^{\oplus2}\oplus[-2]^{\oplus8}\oplus  [-6]  \) & \( \U\oplus  [-2]^{\oplus 11}\) &  \((1,12)\) & \(11\) & \(1\) \\

        \( 39\) & \(12\) & \( \E_8(-1)\oplus [2]^{\oplus2} \oplus\A_2(-1) \) & \(\E_8(-1)\oplus\U\oplus [-2]^{\oplus2}\) &  \((1,11)\) & \(2\) & \(1\) \\

       \(40 \) & \(12\) & \( \E_8(-1)\oplus [2]^{\oplus2} \oplus[2]\oplus[-2]\) & \(\E_8(-1)\oplus [2]\oplus [-2]^{\oplus 3}\) &  \((1,11)\) & \(4\) & \(1\) \\
 
        \(  41\) & \(12\) & \( \D_4(-1)^{\oplus2}  \oplus [2]^{\oplus2}\oplus \A_2(-1)  \) & \(\U\oplus\D_4(-1)^{\oplus 2}  \oplus [-2]^{\oplus2} \) &  \((1,11)\) & \(6\) & \(1\) \\

        \(  42\) & \(12\) & \( \D_4(-1)  \oplus [2]^{\oplus2}\oplus[-2]^{\oplus4}\oplus \A_2(-1)  \) & \(\U\oplus\D_4(-1)  \oplus [-2]^{\oplus6} \) &  \((1,11)\) & \(8\) & \(1\) \\

        \(43 \) & \(12\) & \( \D_4(-1)  \oplus [2]^{\oplus2}\oplus[-2]^{\oplus5}\oplus[-6] \) & \(\D_4(-1) \oplus  [2] \oplus [-2]^{\oplus7}\) &  \((1,11)\) & \(10\) & \(1\) \\

         \(44 \) & \(12\) & \(  [2]^{\oplus 2} \oplus [-2]^{\oplus 9}\oplus[-6] \) & \([2]  \oplus [-2]^{\oplus 11}\) &  \((1,11)\) & \(12\) & \(1\) \\

         \( 45 \) & \(13\) & \( \E_8(-1)  \oplus \U\oplus[2]\oplus \A_2(-1)  \) & \(\E_8(-1)\oplus \U\oplus [-2]\) &  \((1,10)\) & \(1\) & \(1\) \\

            \(  46\) & \(13\) & \( \E_8(-1)  \oplus \oplus[2]\oplus[-2]^{\oplus 2}\A_2(-1)   \) & \(\E_8(-1)\oplus[2]\oplus[-2]^{\oplus 2}\) &  \((1,10)\) & \(3\) & \(1\) \\

     \( 47 \) & \(13\) & \( \U\oplus\D_4(-1)^{\oplus 2}\oplus[2] \oplus\A_2(-1)   \) & \(  \U\oplus \D_4(-1)^{\oplus 2} \oplus [-2]\) &  \((1,10)\) & \(5\) & \(1\) \\

           \( 48 \) & \(13\) & \( \U\oplus\D_4(-1)\oplus[2]\oplus[-2]^{\oplus4} \oplus\A_2(-1)  \) & \( \U\oplus \D_4(-1) \oplus [-2]^{\oplus 5}
        \) &  \((1,10)\) & \(7\) & \(1\) \\

 \( 49 \) & \(13\) & \(  \U\oplus[2]\oplus[-2]^{\oplus8} \oplus\A_2(-1) \) & \(\U\oplus [-2]^{\oplus 9}\) &  \((1,10)\) & \(9\) & \(1\) \\

          \(50  \) & \(13\) & \(  [2]^{\oplus2}\oplus[-2]^{\oplus9} \oplus\A_2(-1)\) & \( [2]\oplus [-2]^{\oplus 10}\) &  \((1,10)\) & \(11\) & \(1\) \\

          \(  51\) & \(14\) & \( \E_8(-1) \oplus \U^{\oplus2}\oplus \A_2(-1) \) & \(\E_8(-1) \oplus \U\) &  \((1,9)\) & \(0\) & \(0\) \\

   \(52  \) & \(14\) & \( \E_8(-1) \oplus \U\oplus\U(2)\oplus \A_2(-1)  \) & \(\E_8(-1) \oplus \U(2) \) &  \((1,9)\) & \(2\) & \(0\) \\

          \(53  \) & \(14\) & \( \E_8(-1) \oplus \U^{\oplus2}\oplus \A_2(-1)  \) & \(\E_8(-1) \oplus [2]\oplus [-2] \) &  \((1,9)\) & \(2\) & \(1\) \\

 \( 54\) & \(14\) & \( \E_8(-1) \oplus \U(2)^{\oplus2}\oplus\A_2(-1) \) & \(\U\oplus \D_4(-1)^{\oplus 2}  \) &  \((1,9)\) & \(4\) & \(0\) \\

          \( 54\) & \(14\) & \( \E_8(-1) \oplus \U\oplus[2]\oplus[-2]^{\oplus2}\oplus[-6]  \) & \(\U\oplus \D_6(-1)\oplus [-2]^{\oplus 2}  \) &  \((1,9)\) & \(4\) & \(1\) \\

   \(  55\) & \(14\) & \( \U\oplus\U(2)\oplus \D_4(-1)^{ \oplus2} \oplus\A_2(-1) \) & \( \U(2)\oplus \D_4(-1)^ {\oplus2}\) &  \((1,9)\) & \(6\) & \(0\) \\
   
          \( 56 \) & \(14\) & \( \U^{\oplus2}\oplus \D_4(-1) \oplus [-2]^{\oplus 4}\oplus\A_2(-1) \) & \(\U\oplus \D_4(-1) \oplus [-2]^{\oplus 4}\) &  \((1,9)\) & \(6\) & \(1\) \\

          \( 57 \) & \(14\) & \( \U^{\oplus2}\oplus \E_8(-2) \oplus\A_2(-1) \) & \(\U\oplus\E_8(-2)\) &  \((1,9)\) & \(8\) & \(0\) \\

          \( 58 \) & \(14\) & \( \U\oplus \D_4(-1) \oplus[2]\oplus [-2]^{\oplus 5}\oplus\A_2(-1) \) & \(\D_4(-1) \oplus [2]\oplus [-2]^{\oplus 5}\) &  \((1,9)\) & \(8\) & \(1\) \\

 \(59  \) & \(14\) & \(  \U\oplus \U(2)\oplus \E_8(-2)\oplus\A_2(-1) \) & \(\E_8(-2) \oplus \U(2)  \) &  \((1,9)\) & \(10\) & \(0\) \\

           \(60  \) & \(14\) & \(  \U\oplus[2] \oplus [-2]^{\oplus 9}\oplus\A_2(-1) \) & \([2] \oplus [-2]^{\oplus 9 } \) &  \((1,9)\) & \(10\) & \(1\) \\

            \( 61 \) & \(15\) & \(\E_8(-1)\oplus\U^{\oplus2} \oplus [-2]\oplus\A_2(-1)  \) & \( \E_8(-1)\oplus [2]\) &  \((1,8)\) & \( 1\) & \(1\) \\

 \( 62 \) & \(15\) & \(\E_8(-1)\oplus\U\oplus \D_6(-1)\oplus[2]^{\oplus2} \oplus [-6] \) & \(\U\oplus\D_6(-1)\oplus[-6]\) &  \((1,8)\) & \( 3\) & \(1\) \\

              \( 62 \) & \(15\) & \(\E_8(-1)\oplus\D_4(-1)\oplus[2]^{\oplus2} \oplus [-6] \) & \(\U\oplus\D_4(-1)\oplus[-2]^{\oplus 3}\) &  \((1,8)\) & \( 5\) & \(1\) \\

            \( 63 \) & \(15\) & \( \U\oplus\D_6(-1)\oplus[2]\oplus[-2]^{\oplus4} \oplus\A_2(-1) \) & \(\U\oplus [-2]^{\oplus 7}\) &  \((1,8)\) & \( 7\) & \(1\) \\

            \(64  \) & \(15\) & \( \D_6(-1)\oplus[2]^{\oplus2}\oplus[-2]^{\oplus5} \oplus\A_2(-1)\) & \(  [2]\oplus [-2]^{\oplus 8}\) &  \((1,8)\) & \( 9\) & \(1\) \\

\( 65 \) & \(16\) & \( \E_8(-1)\oplus\U\oplus \D_6(-1)\oplus\A_2(-1) \) & \( \U\oplus \D_{6}(-1)  \) &  \((1,7)\) & \( 2\) & \(1\) \\

             \( 65 \) & \(16\) & \( \E_8(-1)\oplus\D_4(-1)\oplus[2]^{\oplus2}\oplus\A_2(-1) \) & \( \U\oplus \D_{4}(-1) \oplus [-2]^{\oplus 2} \) &  \((1,7)\) & \( 4\) & \(1\) \\

              \( 66 \) & \(16\) & \( \E_8(-1)\oplus[2]^{\oplus2}\oplus[-2]^{\oplus4}\oplus\A_2(-1) \) & \( \U\oplus [-2]^{\oplus 6} \) &  \((1,7)\) & \( 6\) & \(1\) \\

               \( 67 \) & \(16\) & \(\E_8(-1)\oplus[2]^{\oplus2}\oplus[-2]^{\oplus5}\oplus[-6] \) & \( [2] \oplus [-2]^{\oplus 7} \) &  \((1,7)\) & \( 8\) & \(1\) \\

              \(  68\) & \(17\) & \(\E_8(-1)\oplus\U\oplus\D_4(-1)\oplus[2]\oplus\A_2(-1) \) & \(\U\oplus \D_{4}(-1)\oplus[-2] \) &  \((1,6)\) & \( 3\) & \(1\) \\

                \( 69 \) & \(17\) & \( \E_8(-1)\oplus\U\oplus[2]\oplus[-2]^{\oplus4}\oplus\A_2(-1)\) & \(\U\oplus [-2]^{\oplus 5}\) &  \((1,6)\) & \( 5\) & \(1\) \\

                \( 70 \) & \(17\) & \( \E_8(-1)\oplus[2]^{\oplus2}\oplus[-2]^{\oplus5}\oplus\A_2(-1)\) & \([2] \oplus [-2]^{\oplus 6}\) &  \((1,6)\) & \( 7\) & \(1\) \\

            \( 71 \) & \(18\) & \( \E_8(-1)\oplus\U\oplus\U(2)\oplus\D_4(-1)\oplus\A_2(-1) \) & \( \U(2)\oplus \D_4(-1) \) &  \((1,5)\) & \( 4\) & \(0\) \\

                \( 71 \) & \(18\) & \( \E_8(-1)\oplus\U^{\oplus2}\oplus[-2]^{\oplus4}\oplus\A_2(-1) \) & \( \U\oplus [-2]^{\oplus 4} \) &  \((1,5)\) & \( 4\) & \(1\) \\

                \( 72 \) & \(18\) & \( \E_8(-1)\oplus\U\oplus[2]\oplus[-2]^{\oplus5}\oplus\A_2(-1)\) & \( [2]\oplus [-2]^{\oplus 5}\) &  \((1,5)\) & \( 6\) & \(1\) \\

                \( 73 \) & \(19\) & \( \E_8(-1)^{\oplus2}\oplus[2]^{\oplus2}\oplus[-6] \) & \( \U\oplus [-2]^{\oplus 3} \) &  \((1,4)\) & \( 3\) & \(1\) \\

                \(  74\) & \(19\) & \( \E_8(-1)\oplus\U^{\oplus2}\oplus[-2]^{\oplus5}\oplus\A_2(-1) \) & \(  [2]\oplus [-2]^{\oplus 4} \) &  \((1,4)\) & \( 5\) & \(1\) \\

                \( 75 \) & \(20\) & \( \E_8(-1)^{\oplus2}\oplus[2]^{\oplus2}\oplus\A_2(-1) \) & \( \U\oplus [-2]^{\oplus 2}\) &  \((1,3)\) & \( 2\) & \(1\) \\

              \(  76\) & \(20\) & \(  \E_8(-1)^{\oplus2}\oplus[2]^{\oplus2}\oplus[-2]\oplus[-6]  \) & \( [2]\oplus [-2]^{\oplus 3}  \) &  \((1,3)\) & \( 4\) & \(1\) \\

              \( 77 \) & \(21\) & \(  \E_8(-1)^{\oplus2}\oplus\U\oplus[2]\oplus\A_2(-1)  \) & \( \U\oplus[-2] \) &  \((1,2)\) & \( 1\) & \(1\) \\

            \( 78 \) & \(21\) & \(  \E_8(-1)^{\oplus2}\oplus[2]^{\oplus2}\oplus[-2]\oplus\A_2(-1)  \) & \([2]\oplus [-2]^{\oplus2}  \) &  \((1,2)\) & \( 3\) & \(1\) \\

               \( 79 \) & \(22\) & \(  \E_8(-1)^{\oplus2}\oplus\U^{\oplus2}\oplus\A_2(-1)   \) & \( \U\) &  \((1,1)\) & \( 0\) & \(0\) \\

 \( 80\) & \(22\) & \(  \E_8(-1)^{\oplus2}\oplus\U\oplus\U(2)\oplus\A_2(-1)  \) & \( \U(2)\) &  \((1,1)\) & \( 2\) & \(0\) \\
               
               \( 80\) & \(22\) & \(  \E_8(-1)^{\oplus2}\oplus\U\oplus[2]\oplus[-2]\oplus\A_2(-1)  \) & \( [2]\oplus [-2]\) &  \((1,1)\) & \( 2\) & \(1\) \\

               \( 81 \) & \(23\) & \(  \E_8(-1)^{\oplus2}\oplus\U^{\oplus2}\oplus[2]\oplus\A_2(-1)  \) & \( [2]\)  &  \((1,0)\) & \( 1\) & \(1\) \\

	\midrule
    	\end{longtable}}
\end{center}

\subsection{\(p\geq 3 \)}
\begin{center}
\tiny{\begin{longtable}{lllllll}
		
		\caption{Pairs \((\bLambda^{G},\bLambda_{G})\) for \({G} \subset \bO(\bLambda)\) of prime order \(p \geq 3\) and \(\sgn(\bLambda_{G})=(2, \rk(\bLambda_{G})-2)\).}
	\label{tab:Lambda_p} \\
		
		\toprule
	 No. & \(\rk(\bLambda^{{G}})\) &  \(\bLambda_{G}\) & \(\bLambda^{G}\) & \(\sgn(\bLambda_{G})\) & \(a\) & \(p\)  \\

		\midrule
		\endfirsthead
		
		\multicolumn{7}{c}%
		{\tablename\ \thetable{}, follows from previous page} \\
		\midrule
	No. & \(\rk(\bLambda^{{G}})\) &  \(\bLambda_{G}\) & \(\bLambda^{G}\) & \(\sgn(\bLambda_{G})\) & \(a\) & \(p\)  \\
		\midrule
		\endhead
	
		\multicolumn{7}{c}{Continues on next page} \\
		\endfoot
		
		\bottomrule
		\endlastfoot

		\(1\) & \(24\) & \(\A_2\)& \(\U^{\oplus 3}\oplus \E_8(-1)^{\oplus 2}\oplus \A_2(-1)\)& \((2,0)\) & \(1\) & \(3\) \\
  
         \(2\) & \(22\) & \(\U^{\oplus2} \)& \(\U^{\oplus 3}\oplus \E_8(-1)^{\oplus 2}\)& \((2,2)\) & \(0\) & \(3\) \\

		 \(3\) & \(22\) & \(\U\oplus\U(3) \)& \(\U^{\oplus 2}\oplus\U(3)\oplus \E_8(-1)^{\oplus 2}\)& \((2,2)\) & \(2\) & \(3\) \\

   \(4\) & \(20\) & \( \U^{\oplus2}\oplus\A_2(-1) \)& \(\U^{\oplus 3}\oplus \E_8(-1)\oplus \E_6(-1)\)& \((2,4)\) & \(1\) & \(3\) \\
   
		\(5\) & \(20\) & \( \U\oplus\U(3)\oplus\A_2(-1) \)& \(\U^{\oplus 3}\oplus \E_8(-1)\oplus \A_2(-1)^{\oplus 3}\)& \((2,4)\) & \(3\)  & \(3\) \\

   \(6\) & \(18\) & \( \U^{\oplus 2}\oplus\A_2(-1)^{\oplus2}\)& \(\U^{\oplus 3}\oplus \E_6(-1)^{\oplus 2}\)& \((2,6)\) & \(2\) & \(3\) \\

\(7\) & \(18\) & \( \U\oplus \U(3)\oplus\A_2(-1)^{\oplus2} \)& \(\U^{\oplus 2}\oplus\U(3)\oplus \E_6(-1)^{\oplus 2}\)& \((2,6)\) & \(4\) & \(3\) \\

    \(8\) & \(16\) & \( \U^{\oplus2}\oplus\E_6(-1) \)& \(\U^{\oplus 3}\oplus\E_8(-1)\oplus \A_2(-1)\)& \((2,8)\) & \(1\) & \(3\) \\

 \(9\) & \(16\) & \( \U \oplus \U(3) \oplus \E_6(-1) \)& \(\U^{\oplus 3}\oplus\E_6(-1)\oplus \A_2(-1)^{\oplus 2}\)& \((2,8)\) & \(3\) & \(3\) \\

  \(10\) & \(16\) & \( \U\oplus\U(3)\oplus\A_2(-1)^{\oplus 3} \)& \(\U^{\oplus 2}\oplus \U(3)\oplus\E_6(-1)\oplus \A_2(-1)^{\oplus 2}\)& \((2,8)\) & \(5\) & \(3\) \\

 \(11\) & \(14\) & \( \U^{\oplus2}\oplus\E_8(-1) \)& \(\U^{\oplus3}\oplus\E_8(-1)\)& \((2,10)\) & \(0\) & \(3\) \\

        \(12\) & \(14\) & \( \U\oplus\U(3)\oplus\E_8(-1) \)& \(\U^{\oplus2}\oplus\U(3)\oplus\E_8(-1)\)& \((2,10)\) & \(2\) & \(3\) \\
        
   \(13\) & \(14\) & \( \U^{\oplus2}\oplus\A_2(-1)^{\oplus 4} \)& \(\U^{\oplus3}\oplus\A_2(-1)^{\oplus 4} \)& \((2,10)\) & \(4\) & \(3\) \\
		
	  \(14\) & \(14\) & \( \U\oplus\U(3)\oplus\A_2(-1)^{\oplus 4} \)& \(\U^{\oplus2}\oplus\U(3)\oplus\A_2(-1)^{\oplus 4} \)& \((2,10)\) & \(6\) & \(3\) \\

	  \(15\) & \(12\) & \( \U^{\oplus 2}\oplus \E_8(-1)\oplus\A_2(-1) \)& \(\U^{\oplus3}\oplus\E_6(-1) \)& \((2,12)\) & \(1\) & \(3\) \\

		\(16\) & \(12\) & \( \U\oplus\U(3)\oplus \E_8(-1)\oplus\A_2(-1) \)& \(\U^{\oplus2}\oplus\U(3)\oplus\E_6(-1) \)& \((2,12)\) & \(3\) & \(3\) \\

  	\(17\) & \(12\) & \( \U^{\oplus 2}\oplus \A_2(-1)^{\oplus 5} \)& \(\U^{\oplus2}\oplus\U(3)\oplus\A_2(-1)^{\oplus 3} \)& \((2,12)\) & \(5\) & \(3\) \\

	\(18\) & \(12\) & \( \U\oplus\U(3)\oplus \A_2(-1)^{\oplus 5} \)& \(\U^{\oplus2}\oplus \U(3) \oplus \E_6^*(-3) \)& \((2,12)\) & \(7\)  & \(3\) \\

  \(19\) & \(10\) & \( \U^{\oplus2}\oplus \E_6(-1)^{\oplus 2} \)& \(\U^{\oplus3}\oplus\A_2(-1)^{\oplus 2} \)& \((2,14)\) & \(2\) & \(3\) \\

  \(20\) & \(10\) & \( \U\oplus\U(3) \oplus\E_6(-1)^{\oplus 2} \)& \(\U^{\oplus2}\oplus\U(3)\oplus\A_2(-1)^{\oplus 2} \)& \((2,14)\) & \(4\)  & \(3\) \\

  \(21\) & \(10\) & \( \U\oplus\U(3) \oplus\E_6(-1)\oplus \A_2(-1)^{\oplus 3} \)& \(\U\oplus\U(3)^{\oplus2}\oplus\A_2(-1)^{\oplus 2} \)& \((2,14)\) & \(6\) & \(3\) \\

  \(22\) & \(10\) & \(\U\oplus\U(3) \oplus \A_2(-1)^{\oplus 6}  \)& \( \U(3)^{\oplus3}\oplus\A_2(-1)^{\oplus 2}\)& \((2,14)\) & \(8\) & \(3\) \\

	\(23\) & \(8\) & \( \U^{\oplus 2}\oplus \E_8(-1)\oplus\A_2(-1) \)& \(\U^{\oplus 3}\oplus\E_6(-1) \)& \((2,16)\) & \(1\) & \(3\) \\

        \(24\) & \(8\) & \( \U\oplus \U(3)\oplus \E_8(-1)\oplus\E_6(-1) \)& \(\U^{\oplus 2}\oplus\U(3)\oplus\A_2(-1) \)& \((2,16)\) & \(3\) & \(3\) \\

          \(25\) & \(8\) & \( \U^{\oplus 2}\oplus \E_6(-1)\oplus\A_2(-1)^{\oplus4} \)& \(\U\oplus \U(3)^{\oplus 2}\oplus\A_2(-1) \)& \((2,16)\) & \(5\) & \(3\) \\

          \(25\) & \(8\) & \( \U\oplus \U(3)\oplus \E_6(-1)\oplus\A_2(-1)^{\oplus4} \)& \( \U(3)^{\oplus 3}\oplus\A_2(-1) \)& \((2,16)\) & \(7\) & \(3\) \\

		 \(26\) & \(6\) & \( \U^{\oplus 2}\oplus\E_8(-1)^{\oplus2} \)& \(\U^{\oplus 3} \)& \((2,18)\) & \(0\) & \(3\) \\
		
	 \(27\) & \(6\) & \( \U\oplus \U(3)\oplus\E_8(-1)^{\oplus2} \)& \(\U^{\oplus 2}\oplus\U(3) \)& \((2,18)\) & \(2\) & \(3\) \\
     
    \(28\) & \(6\) & \( \U^{\oplus 2}\oplus\E_8(-1)\oplus\A_2(-1)^{\oplus4} \)& \(\U\oplus\U(3)^{\oplus 2} \)& \((2,18)\) & \(4\) & \(3\) \\

          \(29\) & \(6\) & \( \U^{\oplus 2}\oplus\E_6(-1)\oplus\A_2(-1)^{\oplus5} \)& \(\U(3)^{\oplus 3} \)& \((2,18)\) & \(6\) & \(3\) \\

     \(30\) & \(4\) & \( \U^{\oplus 2}\oplus\E_8(-1)^{\oplus2}\oplus\A_2(-1)\)& \(\U\oplus\A_2 \)& \((2,20)\) & \(1\) & \(3\) \\

      \(31\) & \(4\) & \( \U\oplus \U(3)\oplus\E_8(-1)^{\oplus2}\oplus\A_2(-1)\)& \(\U(3)\oplus\A_2 \)& \((2,20)\) & \(3\) & \(3\) \\
\hline

         \(32\) & \(22\) & \( \U\oplus\h_5\)& \(\U^{\oplus 2}\oplus\h_5\oplus\E_8(-1)^{\oplus 2} \)& \((2,2)\) & \(1\) & \(5\) \\

          \(33\) & \(18\) & \( \U\oplus \h_5 \oplus\A_4(-1)\)& \(\U^{\oplus2 }\oplus\h_5\oplus\E_8(-1)\oplus \A_4(-1)\)& \((2,6)\) & \(2\) & \(5\) \\

          \(34\) & \(14\) & \( \U\oplus\h_5\oplus\E_8(-1)\)& \(\U^{\oplus 2}\oplus \h_5\oplus\E_8(-1) \)& \((2,10)\) & \(1\) & \(5\) \\

          \(35\) & \(14\) & \( \U\oplus\h_5\oplus\A_4(-1)^{\oplus 2}\)& \(\U^{\oplus 2}\oplus \h_5\oplus\A_4(-1)^{\oplus 2} \)& \((2,10)\) & \(3\) & \(5\) \\

          \(36\) & \(10\) & \( \U\oplus\h_5\oplus\E_8(-1)\oplus\A_4(-1)\)& \(\U^{\oplus 2}\oplus\h_5\oplus \A_4(-1) \)& \((2,14)\) & \(2\) & \(5\) \\

          \(37\) & \(10\) & \( \U\oplus\h_5\oplus\A_4(-1)^{\oplus3}\)& \(\U\oplus\U(5)\oplus\h_5\oplus \A_4(-1) \)& \((2,14)\) & \(4\) & \(5\) \\

          \(38\) & \(6\) & \( \U\oplus\h_5\oplus\E_8(-1)^{\oplus 2}\)& \(\U^{\oplus 2}\oplus\h_5 \)& \((2,18)\) & \(1\) & \(5\) \\

          \(39
        \) & \(6\) & \( \U\oplus\h_5\oplus\E_8(-1)\oplus\A_4(-1)^{\oplus 2}\)& \(\U\oplus\U(5) \oplus \h_5\)& \((2,18)\) & \(3\) & \(5\) \\
  
           \(40\) & \(6\) & \( \U(5)^{\oplus2}\oplus\A_4(-1)\)& \(\U(5)^{\oplus 2}\oplus\h_5\)& \((2,18)\) & \(5\) & \(5\) \\
\hline

\(41\) & \(20\) & \( \U^{\oplus 2}\oplus\K_7(-1)\)& \(\U^{\oplus 3}\oplus\E_8(-1)\oplus\A_6(-1)\)& \((2,4)\) & \(1\) & \(7\) \\

\(42\) & \(14\) & \( \U^{\oplus 2}\oplus\E_8(-1)\)& \(\U^{\oplus 3}\oplus\E_8(-1)\)& \((2,10)\) & \(0\) & \(7\) \\

\(43\) & \(14\) & \( \U\oplus \U(7)\oplus\E_8(-1)\)& \(\U^{\oplus 2}\oplus\U(7)\oplus\E_8(-1)\)& \((2,10)\) & \(2\) & \(7\) \\

\(44\) & \(8\) & \( \U^{\oplus 2}\oplus\E_8(-1)\oplus\A_6(-1)\)& \(\U^{\oplus 3}\oplus\K_7\)& \((2,16)\) & \(1\) & \(7\) \\

\(45\) & \(8\) & \( \U\oplus \U(7)\oplus\E_8(-1)\oplus\A_6(-1)\)& \(\U^{\oplus 2}\oplus \U(7)\oplus\K_7\)& \((2,16)\) & \(3\) & \(7\) \\

\hline

\(46\) & \(16\) & \( \K_{11}\oplus\E_8(-1)\)& \(\U^{\oplus 3}\oplus  \A_{10}(-1)\)& \((2,8)\) & \(1\) & \(11\) \\

\(47\) & \(6\) & \( \U^{\oplus2}\oplus\E_8(-1)^{\oplus 2}\)& \(\U^{\oplus 3} \)& \((2,18)\) & \(0\) & \(11\) \\

\(48\) & \(6\) & \( \U\oplus\U(11)\oplus\E_8(-1)^{\oplus 2}\)& \(\U^{\oplus 2}\oplus\U(11) \)& \((2,18)\) & \(2\) & \(11\) \\

\hline

\(49\) & \(14\) & \( \U\oplus\h_{13}\oplus\E_8(-1)\)& \(\U^{\oplus 2}\oplus\h_{13}\oplus\E_8(-1) \)& \((2,10)\) & \(1\) & \(13\) \\

\hline 
\(50\) & \(10\) & \( \U^{\oplus 2}\oplus\E_8(-1)\oplus\bL_{17}(-1)\)& \(\U^{\oplus 3}\oplus\bL_{17}(-1)\)& \((2,14)\) & \(1\) & \(17\) \\

\hline 
\(51\) & \(8\) & \( \K_{19}\oplus\E_8(-1)^{\oplus2}\)& \( \U^{\oplus 3}\oplus \K_{19}(-1)\)& \((2,16)\) & \(1\) & \(19\) \\
\hline
\(52\) & \(4\) & \( \U^{\oplus2}\oplus\E_8(-1)^{\oplus 2}\oplus\K_{23}(-1)\)& \(\U\oplus\K_{23}\)& \((2,20)\) & \(1\) & \(23\) \\

	\midrule
    	\end{longtable}}
\end{center}

\begin{center}
\tiny{\begin{longtable}{lllllll}
		
		\caption{Pairs \((\bL^G,\bL_G)\) for \(G \subset \bO(\bL)\) of prime order \(p \geq 3\) and \(\sgn(\bL_G)=\sgn(\bLambda_{G})\).}
	\label{tab:L_p} \\
		
		\toprule
	 No. & \(\rk(\bL^{G})\) &  \(\bL_G\) & \(\bL^G\) & \(\sgn(\bL_G)\) & \(H\) & \(p\)  \\

		\midrule
		\endfirsthead
		
		\multicolumn{7}{c}%
		{\tablename\ \thetable{}, follows from previous page} \\
		\midrule
	No. & \(\rk(\bL^{G})\) &  \(\bL_G\) & \(\bL^G\) & \(\sgn(\bL_G)\) & \(H\) & \(p\)  \\
		\midrule
		\endhead
	
		\multicolumn{7}{c}{Continues on next page} \\
		\endfoot
		
		\bottomrule
		\endlastfoot

		\(1\) & \(22\) & \(\A_2\)& \(\U\oplus \E_8(-1)^{\oplus 2}\oplus \A_2(-1)^{\oplus 2}\)& \((2,0)\) & \(\id\) & \(3\) \\
  
         \(2\) & \(20\) & \(\U^{\oplus2} \)& \(\U\oplus \E_8(-1)^{\oplus 2}\oplus\A_2(-1)\)& \((2,2)\) & \(\id\) & \(3\) \\

		 \(3\) & \(20\) & \(\U\oplus\U(3) \)& \(\U(3)\oplus \E_8(-1)^{\oplus 2}\oplus\A_2(-1)\)& \((2,2)\) & \(\id\) & \(3\) \\
   
    \(4\) & \(20\) & \(\U\oplus\U(3) \)& \(\U\oplus \E_8(-1)^{\oplus 2}\oplus \A_2(-1)\)& \((2,2)\) & \(\mathbb{Z}/3\mathbb{Z}\) & \(3\) \\

   \(5\) & \(18\) & \( \U^{\oplus2}\oplus\A_2(-1) \)& \(\U\oplus \E_8(-1)\oplus \E_6(-1)\oplus\A_2(-1)\)& \((2,4)\) & \(\id\) & \(3\) \\

   \(6\) & \(18\) & \( \U^{\oplus2}\oplus\A_2(-1) \)& \(\U\oplus \E_8(-1)^{\oplus2} \)& \((2,4)\) & \(\mathbb{Z}/3\mathbb{Z}\) & \(3\) \\

		\(7\) & \(18\) & \( \U\oplus\U(3)\oplus\A_2(-1) \)& \(\U\oplus \E_8(-1)\oplus \A_2(-1)^{\oplus 4}\)& \((2,4)\) & \(\id\) & \(3\) \\

  	\(8\) & \(18\) & \( \U\oplus\U(3)\oplus\A_2(-1) \)& \(\U\oplus \E_8(-1)\oplus \E_6(-1)\oplus\A_2(-1)\)& \((2,4)\) & \(\mathbb{Z}/3\mathbb{Z}\) & \(3\) \\

   \(9\) & \(16\) & \( \U^{\oplus 2}\oplus\A_2(-1)^{\oplus 2} \)& \(\U\oplus \E_6(-1)^{\oplus 2}\oplus\A_2(-1)\)& \((2,6)\) & \(\id\) & \(3\) \\

\(10\) & \(16\) & \( \U^{\oplus 2}\oplus\A_2(-1)^{\oplus 2} \)& \(\U\oplus \E_8(-1)\oplus\E_6(-1)\)& \((2,6)\) & \(\mathbb{Z}/3\mathbb{Z}\) & \(3\) \\
   
\(11\) & \(16\) & \( \U\oplus \U(3)\oplus\A_2(-1)^{\oplus 2} \)& \(\U(3)\oplus \E_6(-1)^{\oplus 2}\oplus\A_2(-1)\)& \((2,6)\) & \(\id\) & \(3\) \\

\(12\) & \(16\) & \( \U\oplus \U(3)\oplus\A_2(-1)^{\oplus 2} \)& \(\U\oplus \E_6(-1)^{\oplus 2}\oplus\A_2(-1)\)& \((2,6)\) & \(\mathbb{Z}/3\mathbb{Z}\) & \(3\) \\

    \(13\) & \(14\) & \( \U^{\oplus2}\oplus\E_6(-1) \)& \(\U\oplus\E_8(-1)\oplus \A_2(-1)^{\oplus 2}\)& \((2,8)\) & \(\id\) & \(3\) \\

 \(14\) & \(14\) & \( \U^{\oplus2}\oplus\A_2(-1)^{\oplus 3} \)& \(\U\oplus\E_6(-1)\oplus \A_2(-1)^{\oplus 3}\)& \((2,8)\) & \(\id\) & \(3\) \\

  \(15\) & \(14\) & \( \U^{\oplus2}\oplus\A_2(-1)^{\oplus 3} \)& \(\U\oplus\E_6(-1)^{\oplus 2}\)& \((2,8)\) & \(\mathbb{Z}/3\mathbb{Z}\) & \(3\) \\

  \(16\) & \(14\) & \( \U\oplus\U(3)\oplus\A_2(-1)^{\oplus 3} \)& \( \U(3)\oplus\E_6(-1)\oplus \A_2(-1)^{\oplus 3}\)& \((2,8)\) & \(\id\) & \(3\) \\

   \(17\) & \(14\) & \( \U\oplus\U(3)\oplus\A_2(-1)^{\oplus 3} \)& \( \U\oplus\E_6(-1)\oplus \A_2(-1)^{\oplus 3}\)& \((2,8)\) & \(\mathbb{Z}/3\mathbb{Z}\) & \(3\) \\

 \(18\) & \(12\) & \( \U^{\oplus2}\oplus\E_8(-1) \)& \(\U\oplus\E_8(-1)\oplus\A_2(-1)\)& \((2,10)\) & \(\id\) & \(3\) \\

        \(19\) & \(12\) & \( \U\oplus\U(3)\oplus\E_8(-1) \)& \(\U(3)\oplus\E_8(-1)\oplus \A_2(-1)\)& \((2,10)\) & \(\id\) & \(3\) \\

        \(20\) & \(12\) & \( \U\oplus\U(3)\oplus\E_8(-1) \)& \(\U\oplus\E_8(-1)\oplus\A_2(-1)\)& \((2,10)\) & \(\mathbb{Z}/3\mathbb{Z}\) & \(3\) \\

   \(21\) & \(12\) & \( \U^{\oplus2}\oplus\A_2(-1)^{\oplus 4} \)& \(\U\oplus\A_2(-1)^{\oplus 5} \)& \((2,10)\) & \(\id\) & \(3\) \\

    \(22\) & \(12\) & \( \U^{\oplus2}\oplus\A_2(-1)^{\oplus 4} \)& \(\U(3)\oplus\E_8(-1)\oplus \A_2(-1)\)& \((2,10)\) & \(\mathbb{Z}/3\mathbb{Z}\) & \(3\) \\

	  \(23\) & \(12\) & \( \U\oplus\U(3)\oplus\A_2(-1)^{\oplus 4} \)& \(\U(3)\oplus\A_2(-1)^{\oplus 5} \)& \((2,10)\) & \(\id\) & \(3\) \\

   \(24\) & \(12\) & \( \U\oplus\U(3)\oplus\A_2(-1)^{\oplus 4} \)& \(\U\oplus\A_2(-1)^{\oplus 5} \)& \((2,10)\) & \(\mathbb{Z}/3\mathbb{Z}\) & \(3\) \\

	  \(25\) & \(10\) & \( \U^{\oplus 2}\oplus \E_8(-1)\oplus\A_2(-1) \)& \(\U\oplus\E_6(-1) \oplus\A_2(-1)\)& \((2,12)\) & \(\id\) & \(3\) \\

    \(26\) & \(10\) & \( \U^{\oplus 2}\oplus \E_8(-1)\oplus\A_2(-1) \)& \(\U\oplus\E_8(-1)\)& \((2,12)\) & \(\mathbb{Z}/3\mathbb{Z}\) & \(3\) \\

		\(27\) & \(10\) & \( \U\oplus\U(3)\oplus \E_8(-1)\oplus\A_2(-1) \)& \(\U(3)\oplus\E_6(-1)\oplus\A_2(-1) \)& \((2,12)\) & \(\id\) & \(3\) \\

  \(28\) & \(10\) & \( \U\oplus\U(3)\oplus \E_8(-1)\oplus\A_2(-1) \)& \(\U\oplus\E_6(-1) \oplus\A_2(-1) \)& \((2,12)\) & \(\mathbb{Z}/3\mathbb{Z}\) & \(3\) \\

  	\(29\) & \(10\) & \( \U^{\oplus 2}\oplus \A_2(-1)^{\oplus 5} \)& \(\U(3)\oplus\A_2(-1)^{\oplus 4} \)& \((2,12)\) & \(\id\) & \(3\) \\

   \(30\) & \(10\) & \( \U^{\oplus 2}\oplus \A_2(-1)^{\oplus 5} \)& \( \U(3)\oplus\E_6(-1)\oplus\A_2(-1)\)& \((2,12)\) & \(\mathbb{Z}/3\mathbb{Z}\) & \(3\) \\

	\(31\) & \(10\) & \( \U\oplus\U(3)\oplus \A_2(-1)^{\oplus 5} \)& \( \U(3)\oplus\E_6^*(-3)\oplus\A_2(-1) \)& \((2,12)\) & \(\id\) & \(3\) \\

	\(32\) & \(10\) & \( \U\oplus\U(3)\oplus \A_2(-1)^{\oplus 5} \)& \( \U(3)\oplus\A_2(-1)^{\oplus 4} \)& \((2,12)\) & \(\id\) & \(3\) \\

  \(33\) & \(8\) & \( \U^{\oplus2}\oplus \E_6(-1)^{\oplus 2} \)& \(\U\oplus\A_2(-1)^{\oplus 3} \)& \((2,14)\) & \(\id\) & \(3\) \\

   \(34\) & \(8\) & \( \U^{\oplus2}\oplus \E_6(-1)^{\oplus 2} \)& \(\U\oplus\E_6(-1) \)& \((2,14)\) & \(\mathbb{Z}/3\mathbb{Z}\) & \(3\) \\

   \(35\) & \(8\) & \( \U\oplus\U(3) \oplus\E_6(-1)^{\oplus 2} \)& \(\U(3)\oplus\A_2(-1)^{\oplus 3} \)& \((2,14)\) & \(\id\) & \(3\) \\

  \(36\) & \(8\) & \( \U\oplus\U(3) \oplus\E_6(-1)^{\oplus 2} \)& \(\U(3)\oplus\E_6(-1) \)& \((2,14)\) & \(\mathbb{Z}/3\mathbb{Z}\) & \(3\) \\

  \(37 \) & \(8\) & \( \U\oplus\U(3) \oplus\E_6(-1)\oplus \A_2(-1)^{\oplus 3} \)& \(\U(3)\oplus\E_6^*(-3) \)& \((2,14)\) & \(\id\) & \(3\) \\

  \(38 \) & \(8\) & \( \U\oplus\U(3) \oplus\E_6(-1)\oplus \A_2(-1)^{\oplus 3} \)& \(\U(3)\oplus\A_2(-1)^{\oplus 3} \)& \((2,14)\) & \(\mathbb{Z}/3\mathbb{Z}\) & \(3\) \\

   \(39 \) & \(8\) & \( \U\oplus\U(3) \oplus \A_2(-1)^{\oplus 6}  \)& \( \U(3)\oplus\E_6^*(-3) \)& \((2,14)\) & \(\mathbb{Z}/3\mathbb{Z}\) & \(3\) \\

	\(40\) & \(6\) & \( \U^{\oplus 2}\oplus \E_8(-1)\oplus\A_2(-1) \)& \(\U\oplus \E_6(-1)^{\oplus 2} \)& \((2,16)\) & \(\id\) & \(3\) \\

        \(41\) & \(6\) & \( \U\oplus \U(3)\oplus \E_8(-1)\oplus\E_6(-1) \)& \(\U(3)\oplus\A_2(-1)^{\oplus 2} \)& \((2,16)\) & \(\id\) & \(3\) \\

        \(42\) & \(6\) & \( \U\oplus \U(3)\oplus \E_8(-1)\oplus\E_6(-1) \)& \(\U\oplus \A_2(-1)^{\oplus 2} \)& \((2,16)\) & \(\mathbb{Z}/3\mathbb{Z}\) & \(3\) \\

          \(43\) & \(6\) & \( \U^{\oplus 2}\oplus \E_6(-1)\oplus\A_2(-1)^{\oplus4} \)& \(\U(3)\oplus \A_2(-1)^{\oplus 2} \)& \((2,16)\) & \(\mathbb{Z}/3\mathbb{Z}\) & \(3\) \\

		 \(44\) & \(4\) & \( \U^{\oplus 2}\oplus\E_8(-1)^{\oplus2} \)& \(\U\oplus \A_2(-1)\)& \((2,18)\) & \(\id\) & \(3\) \\
		
	 \(45\) & \(4\) & \( \U\oplus \U(3)\oplus\E_8(-1)^{\oplus2} \)& \(\U(3)\oplus\A_2(-1) \)& \((2,18)\) & \(\id\) & \(3\) \\

  \(46\) & \(4\) & \( \U\oplus \U(3)\oplus\E_8(-1)^{\oplus2}\) &\( \U\oplus \A_2(-1)\)& \((2,18)\) & \(\mathbb{Z}/3\mathbb{Z}\) & \(3\) \\
     
    \(47\) & \(4\) & \( \U^{\oplus 2}\oplus\E_8(-1)\oplus\A_2(-1)^{\oplus4} \)& \(\U(3)\oplus\A_2(-1) \)& \((2,18)\) & \(\mathbb{Z}/3\mathbb{Z}\) & \(3\) \\

     \(48\) & \(2\) & \( \U^{\oplus 2}\oplus\E_8(-1)^{\oplus2}\oplus\A_2(-1)\)& \(\U(3)  \)& \((2,20)\) & \(\id\) & \(3\) \\
     
      \(49\) & \(2\) & \( \U^{\oplus 2}\oplus\E_8(-1)^{\oplus2}\oplus\A_2(-1)\)& \(\U  \)& \((2,20)\) & \(\mathbb{Z}/3\mathbb{Z}\) & \(3\) \\

      \(50\) & \(2\) & \( \U\oplus \U(3)\oplus\E_8(-1)^{\oplus2}\oplus\A_2(-1)\)& \(\U(3) \)& \((2,20)\) & \(\mathbb{Z}/3\mathbb{Z}\) & \(3\) \\

\hline

         \(51\) & \(20\) & \( \U\oplus\h_5\)& \(\h_5\oplus\E_8(-1)^{\oplus 2}\oplus\A_2(-1) \)& \((2,2)\) & \(\id\) & \(5\) \\

          \(52\) & \(16\) & \( \U\oplus \h_5 \oplus\A_4(-1)\)& \(\h_5\oplus\E_8(-1)\oplus \A_4(-1)\oplus\A_2(-1)\)& \((2,6)\) & \(\id\) & \(5\) \\

          \(53\) & \(12\) & \( \U\oplus\h_5\oplus\E_8(-1)\)& \(\h_5\oplus\E_8(-1)\oplus\A_2(-1) \)& \((2,10)\) & \(\id\) & \(5\) \\

          \(54\) & \(12\) & \( \U\oplus\h_5\oplus\A_4(-1)^{\oplus 2}\)& \( \h_5\oplus\A_4(-1)^{\oplus 2}\oplus\A_2(-1) \)& \((2,10)\) & \(\id\) & \(5\) \\

          \(55\) & \(8\) & \( \U\oplus\h_5\oplus\E_8(-1)\oplus\A_4(-1)\)& \(\h_5\oplus \A_4(-1)\oplus\A_2(-1) \)& \((2,14)\) & \(\id\) & \(5\) \\

          \(56\) & \(8\) & \( \U\oplus\h_5\oplus\A_4(-1)^{\oplus3}\)& \(\U(5)\oplus \A_4(-1)\oplus\N_{15}(-1)\)& \((2,14)\) & \(\id\) & \(5\) \\

          \(57\) & \(4\) & \( \U\oplus\h_5\oplus\E_8(-1)^{\oplus 2}\)& \(\h_5 \oplus\A_2(-1)\)& \((2,18)\) & \(\id\) & \(5\) \\

          \(58\) & \(4\) & \( \U\oplus\h_5\oplus\E_8(-1)\oplus\A_4(-1)^{\oplus 2}\)& \(\U(5) \oplus \N_{15}(-1)\)& \((2,18)\) & \(\id\) & \(5\) \\

\hline

\(59\) & \(18\) & \( \U^{\oplus 2}\oplus\h_7\)& \(\U\oplus\E_8(-1)\oplus\A_6(-1)\oplus\A_2(-1)\)& \((2,4)\) & \(\id\) & \(7\) \\

\(60\) & \(12\) & \( \U^{\oplus 2}\oplus\E_8(-1)\)& \(\U\oplus\E_8(-1)\oplus\A_2(-1)\)& \((2,10)\) & \(\id\) & \(7\) \\

\(61\) & \(12\) & \( \U\oplus \U(7)\oplus\E_8(-1)\)& \(\U(7)\oplus\E_8(-1)\oplus\A_2(-1)\)& \((2,10)\) & \(\id\) & \(7\) \\

\(62\) & \(6\) & \( \U^{\oplus 2}\oplus\E_8(-1)\oplus\A_6(-1)\)& \(\U\oplus\K_7(-1)\oplus\A_2(-1)\)& \((2,16)\) & \(\id\) & \(7\) \\

\(63\) & \(6\) & \( \U\oplus \U(7)\oplus\E_8(-1)\oplus\A_6(-1)\)& \(\U(7)\oplus\K_7(-1)\oplus\A_2(-1)\)& \((2,16)\) & \(\id\) & \(7\) \\
\hline

\(64\) & \(14\) & \( \U\oplus\h_{11}\oplus\E_8(-1)\)& \( \h_{11}\oplus\E_8(-1)\oplus\A_2(-1)\)& \((2,8)\) & \(\id\) & \(11\) \\

\(65\) & \(4\) & \( \U^{\oplus2}\oplus\E_8(-1)^{\oplus 2}\)& \(\U\oplus\A_2(-1) \)& \((2,18)\) & \(\id\) & \(11\) \\

\(66\) & \(4\) & \( \U\oplus\U(11)\oplus\E_8(-1)^{\oplus 2}\)& \(\U(11)\oplus\A_2(-1) \)& \((2,18)\) & \(\id\) & \(11\) \\

\hline

\(67\) & \(12\) & \( \U\oplus\h_{13}\oplus\E_8(-1)\)& \(\h_{13}\oplus\E_8(-1)\oplus\A_2(-1) \)& \((2,10)\) & \(\id\) & \(13\) \\
\hline 
\(68\) & \(6\) & \( \U^{\oplus 2}\oplus\E_8(-1)\oplus\bL_{17}\)& \(\U\oplus\bL_{17}\oplus\A_2(-1) \)& \((2,14)\) & \(\id\) & \(17\) \\

\hline 
\(69\) & \(6\) & \( \K_{19}\oplus\E_8(-1)^{\oplus2}\)& \(\U\oplus\A_2(-1)\oplus\K_{19}(-1)\)& \((2,16)\) & \(\id\) & \(19\) \\
\hline
\(70\) & \(2\) & \( \U^{\oplus2}\oplus\E_8(-1)^{\oplus 2}\oplus\K_{23}\)& \(\N_{69}(-1) \)& \((2,20)\) & \(\id\) & \(23\) \\

	\midrule
    	\end{longtable}}
\end{center}

\bibliographystyle{amsplain}
\bibliography{Biblio}

\providecommand{\bysame}{\leavevmode\hbox to3em{\hrulefill}\thinspace}
\providecommand{\MR}{\relax\ifhmode\unskip\space\fi MR }
\providecommand{\MRhref}[2]{%
  \href{http://www.ams.org/mathscinet-getitem?mr=#1}{#2}
}
\providecommand{\href}[2]{#2}
\begin{thebibliography}{10}

\bibitem{beauville2011antisymplectic}
Arnaud Beauville, \emph{Antisymplectic involutions of holomorphic symplectic manifolds}, Journal of topology \textbf{4} (2011), no.~2, 300--304.

\bibitem{Beauville_Donagi_Droites}
Arnaud Beauville and Ron Donagi, \emph{La vari\'{e}t\'{e} des droites d'une hypersurface cubique de dimension {$4$}}, C. R. Acad. Sci. Paris S\'{e}r. I Math. \textbf{301} (1985), no.~14, 703--706.

\bibitem{noi_e_giovenzani}
Simone Billi, Franco Giovenzana, Luca Giovenzana, and Annalisa Grossi, \emph{Fixed loci of natural automorphisms of {LSV} manifolds}, In preparation.

\bibitem{boissiere2012automorphismes}
Samuel Boissi{\`e}re, \emph{Automorphismes naturels de l'espace de {Douady} de points sur une surface}, Canadian Journal of Mathematics \textbf{64} (2012), no.~1, 3--23.

\bibitem{boissiere2016isometries}
Samuel Boissi\`ere, Chiara Camere, Giovanni Mongardi, and Alessandra Sarti, \emph{Isometries of ideal lattices and hyperk\"{a}hler manifolds}, Int. Math. Res. Not. IMRN (2016), no.~4, 963--977.

\bibitem{boissiere2016classification}
Samuel Boissi\`ere, Chiara Camere, and Alessandra Sarti, \emph{Classification of automorphisms on a deformation family of hyper-{K}\"{a}hler four-folds by {$p$}-elementary lattices}, Kyoto J. Math. \textbf{56} (2016), no.~3, 465--499.

\bibitem{Boissier_NW_Sarti_Enriques}
Samuel Boissi\`ere, Marc Nieper-Wi{\ss}kirchen, and Alessandra Sarti, \emph{Higher dimensional {E}nriques varieties and automorphisms of generalized {K}ummer varieties}, J. Math. Pures Appl. (9) \textbf{95} (2011), no.~5, 553--563.

\bibitem{BRS19}
Michele Bolognesi, Francesco Russo, and Giovanni Staglian\`o, \emph{Some loci of rational cubic fourfolds}, Math. Ann. \textbf{373} (2019), no.~1-2, 165--190.

\bibitem{prime_order_brandhorst}
Simon Brandhorst and Alberto Cattaneo, \emph{Prime order isometries of unimodular lattices and automorphisms of {IHS} manifolds}, Int. Math. Res. Not. IMRN (2023), no.~18, 15584--15638.

\bibitem{camere2012symplectic}
Chiara Camere, \emph{Symplectic involutions of holomorphic symplectic four-folds}, Bulletin of the London Mathematical Society \textbf{44} (2012), no.~4, 687--702.

\bibitem{camere2020non}
Chiara Camere and Alberto Cattaneo, \emph{Non-symplectic automorphisms of odd prime order on manifolds of {K}3\(^{[n]}\)-type}, Manuscripta {M}athematica \textbf{163} (2020), no.~3-4, 299--342.

\bibitem{camere2021non}
Chiara Camere, Alberto Cattaneo, and Andrea Cattaneo, \emph{Non-symplectic involutions on manifolds of {K}3\(^{[n]}\)-type}, Nagoya Mathematical Journal \textbf{243} (2021), 278--302.

\bibitem{camere2019verra}
Chiara Camere, Grzegorz Kapustka, Micha\l Kapustka, and Giovanni Mongardi, \emph{Verra four-folds, twisted sheaves, and the last involution}, Int. Math. Res. Not. IMRN (2019), no.~21, 6661--6710.

\bibitem{Cools_Coppens}
Filip Cools and Marc Coppens, \emph{Star points on smooth hypersurfaces}, J. Algebra \textbf{323} (2010), no.~1, 261--286.

\bibitem{donagi1996spectral}
Ron Donagi and Eyal Markman, \emph{Spectral covers, algebraically completely integrable, {H}amiltonian systems, and moduli of bundles}, Integrable systems and quantum groups ({M}ontecatini {T}erme, 1993), Lecture Notes in Math., vol. 1620, Springer, Berlin, 1996, pp.~1--119.

\bibitem{fu2016classification}
Lie Fu, \emph{Classification of polarized symplectic automorphisms of fano varieties of cubic fourfolds}, Glasgow Mathematical Journal \textbf{58} (2016), no.~1, 17--37.

\bibitem{giovenzana2022period}
Franco Giovenzana, Luca Giovenzana, and Claudio Onorati, \emph{On the period of {L}i, {P}ertusi, and {Z}hao’s symplectic variety}, Canadian Journal of Mathematics (2023), 1–22.

\bibitem{2022symplecticrigidity}
Luca Giovenzana, Annalisa Grossi, Claudio Onorati, and Davide~Cesare Veniani, \emph{Symplectic rigidity of {O}'{G}rady's tenfolds}, Proc. Amer. Math. Soc., doi.org/10.1090/proc/16810 (2023), 1--8.

\bibitem{Gonzales_Liendo}
V\'{\i}ctor Gonz\'{a}lez-Aguilera and Alvaro Liendo, \emph{Automorphisms of prime order of smooth cubic {$n$}-folds}, Arch. Math. (Basel) \textbf{97} (2011), no.~1, 25--37.

\bibitem{grossi2022nonsymplectic}
Annalisa Grossi, \emph{Nonsymplectic automorphisms of prime order on {O}’{G}rady’s sixfolds}, Revista Matem{\'a}tica Iberoamericana \textbf{38} (2022), no.~4, 1199--1218.

\bibitem{grossi2020symplectic}
Annalisa Grossi, Claudio Onorati, and Davide~Cesare Veniani, \emph{Symplectic birational transformations of finite order on {O’G}rady’s sixfolds}, Kyoto Journal of Mathematics \textbf{63} (2023), no.~3, 615--639.

\bibitem{Hassett_Tschinkel_Rational}
B.~Hassett and Y.~Tschinkel, \emph{Rational curves on holomorphic symplectic fourfolds}, Geom. Funct. Anal. \textbf{11} (2001), no.~6, 1201--1228.

\bibitem{hassett2000special}
Brendan Hassett, \emph{Special cubic fourfolds}, Compositio Mathematica \textbf{120} (2000), no.~1, 1--23.

\bibitem{Hassett_rationality_questions}
\bysame, \emph{Cubic fourfolds, {K}3 surfaces, and rationality questions}, Rationality problems in algebraic geometry, Lecture Notes in Math., vol. 2172, Springer, Cham, 2016, pp.~29--66. \MR{3618665}

\bibitem{huybrechts2017k3}
Daniel Huybrechts, \emph{The {K3} category of a cubic fourfold}, Compositio Mathematica \textbf{153} (2017), no.~3, 586--620.

\bibitem{huybrechts2023geometry}
\bysame, \emph{The geometry of cubic hypersurfaces}, vol. 206, Cambridge University Press, 2023.

\bibitem{kuznetsov2010derived}
Alexander Kuznetsov, \emph{Derived categories of cubic fourfolds}, Cohomological and geometric approaches to rationality problems, Progr. Math., vol. 282, Birkh\"{a}user Boston, Boston, MA, 2010, pp.~219--243.

\bibitem{kuznetsov2016derived}
\bysame, \emph{Derived categories view on rationality problems}, Rationality problems in algebraic geometry, Lecture Notes in Math., vol. 2172, Springer, Cham, 2016, pp.~67--104. \MR{3618666}

\bibitem{Laza_modulispaces}
Radu Laza, \emph{The moduli space of cubic fourfolds via the period map}, Ann. of Math. (2) \textbf{172} (2010), no.~1, 673--711.

\bibitem{Laza_maximally}
\bysame, \emph{Maximally algebraic potentially irrational cubic fourfolds}, Proc. Amer. Math. Soc. \textbf{149} (2021), no.~8, 3209--3220.

\bibitem{Laza_Perl_Zheng}
Radu Laza, Gregory Pearlstein, and Zheng Zhang, \emph{On the moduli space of pairs consisting of a cubic threefold and a hyperplane}, Adv. Math. \textbf{340} (2018), 684--722.

\bibitem{laza2017hyper}
Radu Laza, Giulia Sacc{\`a}, and Claire Voisin, \emph{{A hyper-Kähler compactification of the intermediate Jacobian fibration associated with a cubic 4-fold}}, Acta Mathematica \textbf{218} (2017), no.~1, 55 -- 135.

\bibitem{laza2022automorphisms}
Radu Laza and Zhiwei Zheng, \emph{Automorphisms and periods of cubic fourfolds}, Math. Z. \textbf{300} (2022), no.~2, 1455--1507.

\bibitem{li2022elliptic}
Chunyi Li, Laura Pertusi, and Xiaolei Zhao, \emph{Elliptic quintics on cubic fourfolds, {O}'{G}rady 10, and {L}agrangian fibrations}, Adv. Math. \textbf{408} (2022), no.~part A, Paper No. 108584, 56.

\bibitem{markman2011survey}
Eyal Markman, \emph{A survey of {T}orelli and monodromy results for holomorphic-symplectic varieties}, Complex and differential geometry, Springer Proc. Math., vol.~8, Springer, Heidelberg, 2011, pp.~257--322.

\bibitem{marquand2023cubic}
Lisa Marquand, \emph{Cubic fourfolds with an involution}, Trans. Amer. Math. Soc. \textbf{376} (2023), no.~2, 1373--1406.

\bibitem{marquand2023classification}
Lisa Marquand and Stevell Muller, \emph{Classification of symplectic birational involutions of manifolds of {OG}10 type}, preprint, \href{https://arxiv.org/abs/2206.13814}{arXiv:2206.13814v4}, (2023).

\bibitem{marquand2024finite}
\bysame, \emph{Finite groups of symplectic birational transformations of ihs manifolds of {OG}10 type}, preprint, \href{https://arxiv.org/abs/2310.06580}{arXiv:2310.06580v2}, (2024).

\bibitem{marquand2023defect}
Lisa Marquand and Sasha Viktorova, \emph{The defect of a cubic threefold}, preprint, \href{https://arxiv.org/abs/2312.05118}{arXiv:2312.05118}, (2023).

\bibitem{morrison2009embeddings}
Rick Miranda and David~R. Morrison, \emph{The number of embeddings of integral quadratic forms. {II}}, Proc. Japan Acad. Ser. A Math. Sci. \textbf{62} (1986), no.~1, 29--32.

\bibitem{mongardi2012symplectic}
Giovanni Mongardi, \emph{Symplectic involutions on deformations of {K}3\(^{[2]}\)}, Open Mathematics \textbf{10} (2012), no.~4, 1472--1485.

\bibitem{Mongardi_Mori}
\bysame, \emph{A note on the {K}\"{a}hler and {M}ori cones of hyperk\"{a}hler manifolds}, Asian J. Math. \textbf{19} (2015), no.~4, 583--591.

\bibitem{mongardi2022birational}
Giovanni Mongardi and Claudio Onorati, \emph{Birational geometry of irreducible holomorphic symplectic tenfolds of {O}'{G}rady type}, Math. Z. \textbf{300} (2022), no.~4, 3497--3526.

\bibitem{mongardi2018prime}
Giovanni Mongardi, K\'{e}vin Tari, and Malte Wandel, \emph{Prime order automorphisms of generalised {K}ummer fourfolds}, Manuscripta Math. \textbf{155} (2018), no.~3-4, 449--469.

\bibitem{mongardi2015induced}
Giovanni Mongardi and Malte Wandel, \emph{Induced automorphisms on irreducible symplectic manifolds}, Journal of the London Mathematical Society \textbf{92} (2015), no.~1, 123--143.

\bibitem{Mongardi_Wandel}
\bysame, \emph{Automorphisms of {O}'{G}rady's manifolds acting trivially on cohomology}, Algebr. Geom. \textbf{4} (2017), no.~1, 104--119. \MR{3592467}

\bibitem{Nik_finitegroups}
Viacheslav~V. Nikulin, \emph{Finite groups of automorphisms of {K}\"{a}hlerian {$K3$} surfaces}, Trudy Moskov. Mat. Obshch. \textbf{38} (1979), 75--137 (Russian), English translation: Trans. Moscow Math. Soc. \textbf{38} (1980), no. 2, 71--35.

\bibitem{nikulin1980integral}
\bysame, \emph{Integer symmetric bilinear forms and some of their geometric applications}, Izv. Akad. Nauk SSSR Ser. Mat. \textbf{43} (1979), no.~1, 111--177 (Russian), English translation: Math USSR-Izv. \textbf{14} (1979), no. 1, 103--167 (1980). \MR{525944}

\bibitem{oguiso2011enriques}
Keiji Oguiso and Stefan Schr\"{o}er, \emph{Enriques manifolds}, J. Reine Angew. Math. \textbf{661} (2011), 215--235.

\bibitem{ohashi2013non}
Hisanori Ohashi and Malte Wandel, \emph{Non-natural non-symplectic involutions on symplectic manifolds of {K}3\(^{[2]}\)-type}, preprint, \href{https://arxiv.org/abs/1305.6353}{arXiv:1305.6353v2}, (2013).

\bibitem{onorati2020monodromy}
Claudio Onorati, \emph{On the monodromy group of desingularised moduli spaces of sheaves on {K}3 surfaces}, J. Algebraic Geom. \textbf{31} (2022), no.~3, 425--465.

\bibitem{OSCAR}
\emph{Oscar -- open source computer algebra research system, version 1.0.0}, 2024.

\bibitem{rapagnetta2006beauville}
Antonio Rapagnetta, \emph{On the {B}eauville form of the known irreducible symplectic varieties}, Mathematische Annalen \textbf{340} (2008), 77--95.

\bibitem{reid1997chapters}
Miles Reid, \emph{Chapters on algebraic surfaces, complex algebraic geometry, park city ut, usa (1993)}, IAS/Park City Math. Ser \textbf{3} (1997), 3.

\bibitem{rudakov1981surfaces}
Aleksei~Nikolaevich Rudakov and Igor~Rostislavovich Shafarevich, \emph{Surfaces of type {K}3 over fields of finite characteristic}, Itogi Nauki i Tekhniki. Seriya" Sovremennye Problemy Matematiki. Noveishie Dostizheniya" \textbf{18} (1981), 115--207.

\bibitem{RS19}
Francesco Russo and Giovanni Staglian\`o, \emph{Congruences of 5-secant conics and the rationality of some admissible cubic fourfolds}, Duke Math. J. \textbf{168} (2019), no.~5, 849--865. \MR{3934590}

\bibitem{sacca2020birational}
Giulia Sacc{\`a}, \emph{Birational geometry of the intermediate jacobian fibration of a cubic fourfold}, Geometry \& Topology \textbf{27} (2023), no.~4, 1479--1538.

\bibitem{voisin1986theoreme}
Claire Voisin, \emph{Th{\'e}oreme de {T}orelli pour les cubiques de {P}5}, Inventiones mathematicae \textbf{86} (1986), no.~3, 577--601.

\bibitem{voisin2016hyper}
\bysame, \emph{Hyper-{K\"a}hler compactification of the intermediate {J}acobian fibration of a cubic fourfold: the twisted case}, Contemporary Mathematics \textbf{712} (2018), 341--355.

\bibitem{yu2020moduli}
Chenglong Yu and Zhiwei Zheng, \emph{Moduli spaces of symmetric cubic fourfolds and locally symmetric varieties}, Algebra \& Number Theory \textbf{14} (2020), no.~10, 2647--2683.

\bibitem{zheng2021orbifold}
Zhiwei Zheng, \emph{Orbifold aspects of certain occult period maps}, Nagoya Mathematical Journal \textbf{243} (2021), 137--156.

\end{thebibliography}
\end{document}